\documentclass[10pt]{amsart}
\usepackage[margin=1in,letterpaper,portrait]{geometry}
\usepackage{latexsym,amssymb,multirow,graphicx,epsfig,enumerate,
  paralist,xypic,hyperref}
\usepackage{subcaption}
\usepackage[normalem]{ulem}

\usepackage{float}

\theoremstyle{plain}
   \newtheorem{theorem}{Theorem}[section]
   \newtheorem{proposition}[theorem]{Proposition}
   \newtheorem{prop}[theorem]{Proposition}
   \newtheorem{lemma}[theorem]{Lemma}
   \newtheorem{corollary}[theorem]{Corollary}
    \newtheorem{cor}[theorem]{Corollary}
   
\theoremstyle{definition}
   \newtheorem{definition}[theorem]{Definition}
   \newtheorem{example}[theorem]{Example}
   \newtheorem{question}[theorem]{Question}
   \newtheorem{remark}[theorem]{Remark}

\numberwithin{equation}{section}

\newcommand\hook[2]{{\left\langle #2-#1, 1^{#1} \right\rangle}}
\newcommand\Symm{\mathfrak{S}}
\newcommand\xx{{\mathbf{x}}}
\newcommand\yy{{\mathbf{y}}}
\newcommand\pp{{\mathbf{p}}}
\newcommand\rr{{\mathbf{r}}}
\newcommand\ee{{\mathbf{e}}}
\newcommand\codim{\operatorname{codim}}
\newcommand\Irr{\operatorname{Irr}}
\newcommand\Aut{\operatorname{Aut}}
\newcommand\fix{\operatorname{fix}}
\newcommand\normchi{\widetilde{\chi}}
\newcommand\CCC{{\mathcal{C}}}
\newcommand\CC{{\mathbb{C}}}
\newcommand\ZZ{{\mathbb{Z}}}
\newcommand\NN{{\mathbb{N}}}
\newcommand\QQ{{\mathbb{Q}}}
\newcommand\GL{{\mathrm{GL}}}
\newcommand{\cn}{\varsigma}
\newcommand{\cnn}{\varsigma_{(n - 1, 1)}}
\newcommand{\poly}[3]{P_{#1}^{(#2)}(#3)}
\newcommand{\z}{\zeta}
\newcommand{\zz}{\mathfrak{z}}
\newcommand{\lR}{\ell_{R}}

\begin{document}

\title{Factorization problems in complex reflection groups}
\author{Joel Brewster Lewis and Alejandro H. Morales}

\maketitle

\begin{center}
  \emph{Dedicated to David M. Jackson in recognition of his 75th birthday}
\end{center}

\begin{abstract}
We enumerate factorizations of a Coxeter element in a well generated complex reflection group into arbitrary factors, keeping track of the fixed space dimension of each factor.  In the infinite families of generalized permutations, our approach is fully combinatorial.  It gives results analogous to those of Jackson in the symmetric group and can be refined to encode a notion of cycle type.  As one application of our results, we give a previously overlooked characterization of the poset of $W$-noncrossing partitions.
\end{abstract}

%\tableofcontents

\section{Introduction}

The motivation for this paper is the following formula of
Chapuy and Stump for the generating function for the number of factorizations of a fixed
Coxeter element by reflections in a complex reflection
group. 

\begin{theorem}[{Chapuy--Stump \cite[Thm.~1.1]{ChapuyStump}}]
\label{thm:cs}
Let $W$ be an irreducible well generated complex reflection group of
rank $n$. Let $c$ be a Coxeter element in $W$, let $R$ and $R^*$ be
the set of all reflections and all reflecting hyperplanes in $W$, and
for $\ell \geq 1$ let $N_{\ell}(W) := \#\{
(\tau_1,\ldots,\tau_{\ell}) \in R^{\ell} \colon \tau_1\cdots
\tau_{\ell} =  c \bigr\}$ be the number of factorizations of $c$ as a
product of $\ell$ reflections in $R$. Then
\[
\sum_{\ell \geq 0} N_{\ell}(W)\frac{t^{\ell}}{\ell!}  = \frac{1}{|W|} \left( e^{t|R|/n}
  - e^{-t |R^*|/n} \right)^n.
\]
\end{theorem}

Near $t=0$, the generating function recovers the fact \cite[Prop.~7.6]{Bessis} that the number of minimum-length reflection factorizations of a Coxeter element is $n! h^n/|W|$, where $h$ is the Coxeter number of $W$. When $W$ is the symmetric group $\Symm_{n}$, Theorem~\ref{thm:cs} reduces
to a result of Jackson \cite{Jackson} counting factorizations of
the $n$-cycle $(12\cdots n)$ into transpositions.  Chapuy and Stump prove their result
by an algebraic approach with irreducible
characters that dates back to Frobenius.

A natural question is whether there are extensions to complex
reflection groups of other factorization results in the symmetric
group. In the same paper \cite{Jackson}, Jackson gave formulas for the
generating polynomial of factorizations of an $n$-cycle as a
product of a fixed number of factors, keeping track of the number of
cycles of each factor. We state the result for two and $k$ factors, as
reformulated by Schaeffer--Vassilieva.

\begin{theorem}[{Jackson \cite{Jackson}, Schaeffer--Vassilieva
    \cite{SchaefferVassilieva}}]
\label{S_n theorem}
\label{S_n theorem k factors}
Let $\cn$ be a fixed $n$-cycle in $\Symm_n$, and for integers $r_1,\ldots,r_k$
let $a_{r_1,\ldots,r_k}$ be the number of factorizations of $\cn$ as a
product of $k$ permutations in $\Symm_n$ such that $\pi_i$ has $r_i$ cycles for $i=1,\ldots,k$. Then
\begin{equation} \label{eqn S_n 2 factors}
\frac{1}{n!} \cdot \sum_{r, s \geq 1} a_{r, s} x^r y^s = \sum_{p,q
  \geq 1} \binom{n-1}{p-1; q-1; n - p - q + 1} \frac{(x)_{p}}{p!}
  \frac{(y)_{q}}{q!}, 
\end{equation}
where $(x)_p$ denotes the falling factorial $(x)_p := x(x - 1)\cdots (x - p + 1)$.  More generally,
\begin{equation}
\label{eqn S_n k factors}
\frac{1}{(n!)^{k-1}} \cdot \sum_{r_1,\ldots,r_k \geq 1} a_{r_1,
  \ldots, r_k} x_1^{r_1} \cdots x_k^{r_k} = \sum_{p_1,\ldots,p_k \geq
  1}  M^{n-1}_{p_1-1,\ldots,p_k-1} \frac{(x_1)_{p_1}}{p_1!} \cdots
  \frac{(x_k)_{p_k}}{p_k!}, 
\end{equation}
where the coefficient $M^n_{p_1,\ldots,p_k}$ is defined in \eqref{definition M coeff}.
\end{theorem}

In this paper we give analogues of these results for two infinite families of complex reflection
groups:  the group $G(d,1,n)$ of $n \times n$ monomial matrices whose
nonzero entries are all $d$th roots of unity (i.e., the wreath
product $(\ZZ/d\ZZ)  \wr \mathfrak{S}_n$; at $d = 2$, the Coxeter
group of type $B_n$) and its subgroup $G(d,d,n)$ of matrices whose nonzero
  entries multiply to $1$ (at $d = 2$, the Coxeter group of type $D_n$).  
The analogue of an $n$-cycle in a complex reflection group is a \emph{Coxeter element}. 
The analogue of number of cycles of a group element is the \emph{fixed space dimension}.  Our results for $G(d, 1, n)$ are in terms of the polynomials
\begin{equation}
\label{first basis}
    (x - 1)_k^{(d)} := (x - 1)(x-1-d)(x-1-2d)\cdots (x-1-(k-1)d) = \prod_{i = 1}^k (x - e^*_i).
\end{equation}
Here the roots $e^*_i$ are the \emph{coexponents} of this group, one
of the fundamental sets of invariants associated to every complex
reflection group. (All these terms are defined in Section~\ref{subsec:crg}.) 

\begin{theorem} \label{thm: main G(d,1,n)}
For $d > 1$, let $G = G(d, 1, n)$ and let $a^{(d)}_{r_1,\ldots,r_k}$ be the number of factorizations of a fixed Coxeter element $c$ in $G$ as a product of $k$ elements of $G$ with fixed space dimensions $r_1,\ldots,r_k$,
respectively. Then
\begin{align}
\label{eq main G(d,1,n)}
\frac{1}{|G|} \cdot \sum_{r,s \geq 0} a^{(d)}_{r,s} x^r y^s  &= \sum_{p,q \geq 0}
 \binom{n}{p; q; n - p - q} \frac{(x - 1)^{(d)}_{p}}{d^pp!} \frac{(y - 1)^{(d)}_{q}}{d^qq!},\\ \intertext{and more generally}
\frac{1}{|G|^{k-1}}\cdot \sum_{r_1,\ldots,r_k \geq 0}
a^{(d)}_{r_1,\ldots,r_k} x_1^{r_1} \cdots x_k^{r_k} &= 
 \sum_{p_1,\ldots,p_k \geq 0}
M^n_{p_1,\ldots,p_k} \frac{(x_1 - 1)^{(d)}_{p_1}}{d^{p_1}p_1!} \cdots \frac{(x_k - 1)^{(d)}_{p_k}}{d^{p_k}p_k!},
 \end{align}
where $M^n_{p_1,\ldots,p_k}$ is defined in \eqref{definition M coeff}.
\end{theorem}

In Section~\ref{sec:G(d, 1, n)}, we give a combinatorial proof of the
result for $k$ factors by reducing it to
  the case of the symmetric group for $k$ factors. The latter case has a
  combinatorial proof \cite{BM1,BM2}. Our proof works directly with the
group elements and permutations. However, the proof could also be
written in terms of {\em maps} (e.g., see \cite{LZ,GS}) through a
coloring argument as done in \cite{B,SchaefferVassilieva,CFF} for the
case of factorizations with two factors and in \cite{BM1} for the
case of factorizations with $k$ factors. 

Our main results for the subgroup $G(d,d,n)$ of $G(d,1,n)$ count {\em transitive} factorizations (for the natural action on a set of size $dn$) of a Coxeter element. They are written in terms of the polynomials
\[
\poly{k}{d}{x} := (x - (k - 1)(d - 1)) \cdot (x - 1)^{(d)}_{k - 1} = \prod_{i = 1}^k (x - e^*_i),
\]
where again the $e^*_i$ are the coexponents of the
group. 

\begin{theorem} \label{thm:mainG(d,d,n)}
For $d > 1$, let $G = G(d, d, n)$ and let 
 $b^{(d)}_{r_1,\ldots,r_k}$ be the number of transitive factorizations of a Coxeter element $c_{(d, d, n)}$ in $G$ as a product of $k$
elements of $G$ with fixed space dimensions $r_1,\ldots,r_k$, respectively. 
Then
\begin{align} \label{eq main G(d,d,n)}
\frac{n}{(n - 1) \cdot |G|} \cdot \sum_{r,s \geq 0} b^{(d)}_{r,s} x^r y^s &= \sum_{p,q \geq 1}
 \binom{n-2}{p-1; q-1; n -p - q} \frac{\poly{p}{d}{x}}{d^{p-1} p!}
 \frac{\poly{q}{d}{x}}{d^{q-1} q!},\\ \intertext{and more generally}
\frac{n^k}{|G|^{k-1}}\cdot \sum_{r_1,\ldots,r_k
  \geq 0} b^{(d)}_{r_1,\ldots,r_k} x_1^{r_1}\cdots x_k^{r_k}
 &= \sum_{p_1,\ldots,p_k \geq 1}
 M^n_{p_1,\ldots,p_k} \frac{\poly{p_1}{d}{x_1}}{d^{p_1-1} (p_1-1)!} \cdots 
 \frac{\poly{p_k}{d}{x_k}}{d^{p_k-1} (p_k-1)!},
 \end{align}
 where $M^n_{p_1,\ldots,p_k}$ is defined in \eqref{definition M coeff}.
 \end{theorem}

In Section~\ref{sec:ddn}, we give a combinatorial proof of these results.  The proof relies on an enumeration of transitive factorizations of an $(n - 1)$-cycle in $\Symm_n$ into $k$ factors that appears to be new -- its proof may be found in Section~\ref{sec:n - 1 cycle}.

In the remainder of the paper, we consider a variety of extensions and
applications of these results.  In Section~\ref{sec:exceptional}, we
explore the same question in the exceptional complex reflection
groups, using an algebraic approach.  This produces results that are
strikingly similar to the results from the infinite families in many cases, but
ultimately no uniform formula along the lines of Theorem~\ref{thm:cs}.
The question of whether a uniform theorem exists is raised in Section~\ref{sec:remarks}.

In Section~\ref{sec:applications}, we show how to derive the
Chapuy--Stump result from our main results, giving a fully
combinatorial proof in the case of $G(d, 1, n)$.  We also consider the special case of \emph{genus-$0$ factorizations}, which are extremal with respect to a natural subadditivity of fixed space.  As a consequence, we derive a characterization of the \emph{poset of $W$-noncrossing partitions} that has (surprisingly) been overlooked before now.

In Section~\ref{cycle type}, we refine the result for $G(d, 1, n)$ by
the group orbit of the fixed space, or equivalently by an appropriate
notion of cycle type.  The proof is again fully combinatorial.  In the
genus-$0$ case, this result gives an analogue of the Goulden--Jackson cactus
formula \cite[Thm.~3.2]{GJ} and specializes to a result of Krattenthaler--M\"uller in type $B$ \cite[Thm.~7(i)]{KMtypeB}. 
Finally, in Section~\ref{sec:remarks}, we end with a number of remarks and open questions, including constructions of maps associated to factorizations in $G(d, 1, n)$.

An extended abstract of this work appeared in \cite{FPSAC}.

\subsection*{Acknowledgements}
We are indebted to Drew Armstrong, Olivier Bernardi, Theo
Douvropoulos, John Irving, Vic Reiner, Gilles Schaeffer,  Christian
Stump, and the anonymous referees whose comments, questions, and suggestions led to significant
improvements in the paper.  JBL was partially supported by a grant from the Simons Foundation (634530) and by an ORAU Powe award.  AHM was partially supported by the NSF grant DMS-1855536.

\section{Background}
\label{sec:background}

\subsection{Known factorization results in $\mathfrak{S}_n$}

We begin by discussing in more detail the background behind Theorem~\ref{S_n theorem}.  Let $\cn$ be a fixed $n$-cycle in $\Symm_n$, and for integers
$r_1,\ldots,r_k$ let $a_{r_1,\ldots,r_k}$ be the number of $k$-tuples
$(\pi_1,\ldots,\pi_k)$ of elements in $\Symm_n$ such that $\pi_i$ has $r_i$
cycles for $i=1,\ldots,k$ and $\pi_1 \cdots  \pi_k = \cn$.
Theorem~\ref{S_n theorem} is a corollary of a result obtained by Jackson \cite[Thm.~4.3]{Jackson}; our formulation follows Schaeffer and Vassilieva \cite[Thm.~1.3]{SchaefferVassilieva}.  The coefficients on the RHS of \eqref{eqn S_n k factors} are defined as follows.  Given a positive integer $k$ and nonnegative integers $n$ and $p_1,\ldots,p_k$, let
\begin{equation} \label{definition M coeff}
M^n_{p_1,\ldots,p_k} := \sum_{t=0}^{\min(p_i)} (-1)^t \binom{n}{t}
\prod_{i=1}^k \binom{n-t}{p_i-t} = [x_1^{p_1} \cdots x_k^{p_k}] \Bigl(
  (1+x_1)\cdots (1+x_k) - x_1\cdots x_k\Bigr)^n,
\end{equation}
where the square brackets in the third expression represent coefficient extraction.  This number counts $n$-tuples $(S_1,\ldots,S_n)$ of proper subsets of $[k] := \{1, \ldots, k\}$ such that exactly $p_j$ of the sets contain $j$. It is easy to see that the
 $M^n_{\pp}$  satisfy the following recurrence.

\begin{proposition} \label{prop:recurrence Ms}
One has
$ \displaystyle
M^n_\pp
=\sum_{S \subsetneq [k]} M^{n-1}_{\pp -  \ee_{S}}
 = \sum_{\varnothing \neq T\subseteq [k]} M^{n-1}_{\pp - {\bf 1} + \ee_{T}}  
$
where $\ee_{S}$ denotes the indicator vector for the set $S$ and ${\bf 1} := \ee_{[k]}$ is the all-ones vector. 
\end{proposition}

Also, from the enumerative interpretation of $M^n_{\pp}$ one has that $M^n_{p_1,p_2}$ is given by the
multinomial coefficient $\binom{n}{p_1; p_2; n-p_1-p_2}$, and that $M^n_{\pp} = 0$ whenever $p_1 + \dots + p_k > n(k - 1)$. For
$k \geq 3$, the $M^n_{\pp}$ are not given by a multinomial coefficient
or other product formula, except for the following extremal case.

\begin{proposition} \label{prop:Ms2multinomial}
If $p_1+\dots +p_k = n(k-1)$ then $M^n_{p_1, \ldots, p_k} =
\binom{n}{n-p_1; \cdots; n-p_k}$.
\end{proposition}

\begin{proof}
Suppose that $p_1+ \dots + p_k = n(k-1)$ and that $(S_1,\ldots,S_n)$ is an 
$n$-tuple of proper subsets of $[k]$ counted
by $M^n_{p_1,\ldots,p_k}$.  Let $c_i:= |S_i|$
for $i=1,\ldots,n$. Since $c_i \leq k-1$ and $c_1+ \dots + c_n =
p_1+ \dots +p_k = n(k - 1)$, it must be the case that $c_i = k-1$ for $i=1,\ldots,n$. Thus $S_i = [k]
\smallsetminus \{a_i\}$ for some $a_i \in [k]$.  Then the correspondence $(S_1, \ldots, S_n) \longleftrightarrow  (a_1,\ldots,a_n)$ is a bijection between
the $n$-tuples of sets counted by $M^n_{\pp}$ and $n$-tuples of elements of
$[k]$ such that $j$ appears $n-p_j$ times for $j = 1, \ldots, k$. The latter set is obviously counted by the desired multinomial coefficient, and the result follows.
\end{proof}

Jackson's proof of Theorem~\ref{S_n theorem} uses an algebraic approach based on work of Frobenius from the late 19th century; these methods are described in Section~\ref{sec:Frobenius} below, after which we apply them to give a similar result for the $(n - 1)$-cycle in $\Symm_n$. Bijective proofs of the case $k=2$ were given by Schaeffer--Vassilieva \cite{SchaefferVassilieva}, Chapuy--F\'eray--Fusy \cite{CFF}, and
Bernardi \cite{B}.  Bernardi and Morales \cite{BM1,BM2} extended Bernardi's approach to give a combinatorial proof of Jackson's formula for all $k$ in terms of maps. These combinatorial proofs use an interpretation of the change of basis in \eqref{S_n theorem k factors} that we describe now. 

Let $\CCC^{\langle n\rangle}_{p_1,\ldots,p_k}$ be the set of factorizations in
$\mathfrak{S}_n$ of the fixed $n$-cycle $\cn$ as a product $\pi_1\cdots
\pi_k$ such that for each $i$, the cycles of $\pi_i$ are colored
with $p_i$ colors with each of the colors being used at least once to
color a cycle of $\pi_i$.  (In particular, once the factorization is
fixed, this means that the colorings of the cycles in the different
factors are completely independent of each other.) Let
  $C^{\langle n \rangle}_{p_1,\ldots,p_k} = |\CCC^{\langle
    n\rangle}_{p_1,\ldots,p_k}|$ be the number of such colored
  factorizations. 

\begin{remark} \label{remark: independence color sets}
Note that the number $C^{\langle n \rangle}_{p_1,\ldots,p_k}$ of
colored factorizations does not depend on the set of
  $p_i$ colors used to color the cycles of
  $u_i$. In an abuse of notation, depending on context, we will use
  the same symbols $\CCC^{\langle n\rangle}_{p_1,\ldots,p_k}$ and
  $C^{\langle n\rangle}_{p_1,\ldots,p_k}$ to denote the set and the
  number of colored factorizations where for each $\pi_i$ we use color set $[p_i]$, or
  $\{0,1,\ldots,p_i-1\}$, or a $p_i$-subset of a larger set.
\end{remark}

\begin{proposition}
\label{prop:change of basis}
With $a_{r_1, \ldots, r_k}$ and $C^{\langle n \rangle}_{p_1, \ldots, p_k}$ as above, one has
\begin{equation} \label{eq change of basis}
\sum_{r_1,\ldots,r_k \geq 1} a_{r_1,
  \ldots, r_k} x_1^{r_1} \cdots x_k^{r_k} = \sum_{p_1,\ldots,p_k \geq
  1} C^{\langle n \rangle}_{p_1,\ldots,p_k} \binom{x_1}{p_1}\cdots \binom{x_k}{p_k}.
\end{equation}
\end{proposition}

\begin{proof}
Let each $x_i$ be a nonnegative integer. The LHS of \eqref{eq change of basis} counts  
factorizations of the cycle $\cn = (12\cdots n)$ as a product
  $\cn = \pi_1\cdots \pi_k$, where
for $i=1,\ldots,k$, each cycle of $\pi_i$ is colored with a
  color in $[x_i]$.
These colored factorizations are also counted by the RHS of \eqref{eq change of basis}: for
$p_1,\ldots,p_k \in \mathbb{N}$ and $i=1,\ldots,k$, choose $p_i$
colors from $x_i$ colors available and a colored factorization in
$\CCC^{\langle n\rangle}_{p_1,\ldots,p_k}$ where exactly those $p_i$ colors are used in
the factor $\pi_i$.  Finally, since \eqref{eq change of basis} is valid for all nonnegative integer values of the $x_i$, it is valid as a polynomial identity.
\end{proof}

The following corollary is an immediate consequence of Theorem~\ref{S_n theorem k factors} and Proposition~\ref{prop:change of basis}.

\begin{corollary}
\label{S_n corollary}
The number of colored factorizations of an $n$-cycle in $\Symm_n$ is $C^{\langle n \rangle}_{p_1,\ldots,p_k}  =  (n!)^{k-1} M^{n-1}_{p_1-1,\ldots,p_k-1}$.
\end{corollary}

\subsection{Complex reflection groups} \label{subsec:crg}

In this section, we give an account of the complex reflection groups,
paying particular attention to the ``combinatorial'' groups $G(d, 1,
n)$ and $G(d, d, n)$.  For more thorough background, see
\cite{LehrerTaylor}.  

\subsubsection{Basic definitions}

Given a complex vector space $V$ of dimension $n$, a linear transformation $r$ on $V$ is called a \emph{reflection} if the dimension of the fixed space of $r$ (i.e., the set of vectors $v$ such that $r(v) = v$) is $n - 1$, that is, if $r$ fixes a hyperplane.  A \emph{complex reflection group} is a finite subgroup of $\GL(V)$ generated by its subset of reflections.  A complex reflection group is \emph{irreducible} if it does not stabilize any nontrivial subspace of $V$, and every complex reflection group decomposes uniquely as a direct product of irreducibles.  A complex reflection group is \emph{well-generated} if it acts irreducibly on a space of dimension $m$ and has a generating set consisting of $m$ reflections.

\subsubsection{Key examples}

The most common examples of complex reflection groups are the finite
Coxeter groups, including the dihedral groups, the symmetric group
$\Symm_n$ (type $A_{n - 1}$), and the hyperoctahedral group of signed
permutations (type $B_n$) and its index-$2$ subgroup of ``even-signed
permutations'' (type $D_n$), whose elements have an even number
of negative entries.  All real reflection groups are well-generated.  (For $\Symm_n$, the space on which it acts irreducibly has dimension $m = n - 1$; for the signed and even-signed permutations, $m = n$.)  

There are two infinite families of well generated irreducible complex
reflection groups.  The groups in the first family are the wreath products $G(d, 1, n) = (\ZZ/d\ZZ) \wr \Symm_n$ of the symmetric group by a cyclic group of order $d$.  Concretely, the elements of this group may be realized as \emph{generalized permutation matrices} with one nonzero entry in each row and column, each of which is a complex $d$th root of unity -- see Figure~\ref{fig:Cox}.  Thus $G(2, 1, n)$ is the hyperoctahedral group of signed permutations.  More compactly, elements of $G(d, 1, n)$ may be identified with pairs $[\pi; a]$ where $\pi$ is a permutation in $\Symm_n$ and $a = (a_1, \ldots, a_n)$ is a tuple of elements of $\ZZ/d\ZZ$.  We say that $a_i$ is the \emph{weight} of $i$ in $[\pi; a]$ and that $\pi$ is the \emph{underlying permutation}.  In this notation, the product of two group elements is given by
\[
[\pi; a] \cdot [\sigma; b] = [\pi\sigma; \sigma(a) + b]
\]
where $\sigma(a) := (a_{\sigma(1)}, \ldots, a_{\sigma(n)})$.  The
underlying permutation $\pi$ is the image of $[\pi; a]$ under the
natural \emph{projection map} $G(d, 1, n) \twoheadrightarrow \Symm_n$.
A \emph{cycle} in $[\pi; a] \in G(d, 1, n)$ means a cycle in $\pi$.
The \emph{weight} of a cycle is the sum in
  $\mathbb{Z}/d\mathbb{Z}$ of the weights of the elements in the cycle.

\begin{figure}
\[
\begin{bmatrix}
&& & \omega \\
1&&& \\
& 1 && \\
&& 1 &
\end{bmatrix}
\hspace{2in}
\begin{bmatrix}
&& \omega & \\
1 &&& \\
& 1 && \\
&&& \omega^{-1}
\end{bmatrix}
\]
\caption{The matrix representation of the generalized permutations $[(1234); (0, 0, 0, 1)]$ in $G(d, 1, 4)$ (left) and $c_{(d, d, 4)} := [(123)(4); (0, 0, 1, -1)]$ in $G(d, d, 4) \subset G(d, 1, 4)$ (right).  Here $d > 1$ and $\omega = \exp(2\pi i/d)$ is a primitive complex $d$th root of unity.  In $c_{(d, d, 4)}$, the cycle $(123)$ has weight $1$ and the cycle $(4)$ has weight $-1$.  
}
\label{fig:Cox}
\end{figure}

When $d > 1$, there are two ``flavors'' of reflections in $G(d, 1, n)$: the diagonal reflections, which have the identity as underlying permutation and a single diagonal entry of nonzero weight, and the transposition-like reflections, whose underlying permutation is a transposition and whose cycles all have weight $0$ -- see Figure~\ref{fig:reflections}.
The second infinite family of well generated complex reflection groups contains, for each $d > 1$ and $n \geq 2$, the subgroup of $G(d, 1, n)$ generated by the transposition-like reflections; it is denoted $G(d, d, n)$.  Equivalently, $G(d, d, n)$ contains those elements of $G(d,1, n)$ of weight $0$ (i.e., generalized permutation matrices in which the product of the nonzero entries is $1$).  In the case $d = 2$, it is exactly the Coxeter group of type $D_n$, and in the case $n = 2$, it is the dihedral group of order $2d$.  The group $G(d, d, n)$ is always well generated, and it is irreducible except when $d = n = 2$.

\begin{figure}
\[
\begin{bmatrix}
a && & \\
&1&& \\
&& 1 & \\
&&& 1
\end{bmatrix}
\hspace{2in}
\begin{bmatrix}
& b^{-1}& & \\
b&&& \\
&& 1 & \\
&&& 1
\end{bmatrix}
\]
\caption{A diagonal reflection and a transposition-like reflection in $G(d, 1, 4)$.  Here $a$ represents a $d$th root of $1$ other than $1$ itself, and $b$ represents an arbitrary $d$th root of $1$.}
\label{fig:reflections}
\end{figure}

In addition to the infinite families of irreducible complex reflection groups, there are $34$ exceptional groups.  Of these, $26$ are well generated, including the six exceptional Coxeter groups (of types $H_3, F_4, H_4, E_6, E_7$, and $E_8$).  These are not the main focus of this paper, but they are discussed further in Section~\ref{sec:exceptional}.

\subsubsection{Fixed space dimension}

It is easy to see that conjugacy classes in $G(d, 1, n)$ are uniquely determined by the cycle type of the underlying permutation together with the multiset of weights of the cycles of each length.  Equivalently, for each weight $j = 0, \ldots, d - 1$, we have a partition (possibly empty) recording the lengths of the cycles of weight $j$.  Thus, conjugacy classes in $G(d, 1, n)$ are unambiguously indexed by tuples $(\lambda^{(0)}, \ldots, \lambda^{(d - 1)})$ of partitions of total size $n$. The following proposition is straightforward.
\begin{prop}
\label{prop:reflection length}
The fixed space dimension of an element $w$ in $G(d, 1, n)$ whose conjugacy class is indexed by $(\lambda^{(0)}, \ldots, \lambda^{(d - 1)})$ is equal to $\ell(\lambda^{(0)})$, the number of cycles of weight $0$ in $w$.
\end{prop}
Since $G(d, d, n) \subset G(d, 1, n)$, the same combinatorial formula gives the fixed space dimension for elements of the smaller group.  For the symmetric group $\Symm_n$ acting on $\CC^n$, the fixed space dimension of $w$ is exactly the number of cycles of $w$.

\subsubsection{Coxeter elements}
\label{sec:Coxeter elements}

An element $w$ in a complex reflection group $G$ is \emph{regular} if it has an eigenvector that does not lie on the fixed plane of any reflection in the group.  Since $w$ is of finite order, the associated eigenvalue is a root of unity; if it is a primitive root of order $k$, then we say that the integer $k$ is \emph{regular} as well.  The \emph{Coxeter number} $h$ of $G$ is the largest regular integer, and a \emph{Coxeter element} of $G$ is a regular element of order $h$ (but see Remark~\ref{rem:reflection automorphisms} below).  In the case of the symmetric group $\Symm_n$, the regular elements are the $n$-cycles, the $(n - 1)$-cycles, and their powers, the Coxeter number is $h = n$, and the Coxeter elements are exactly the $n$-cycles.  For $d > 1$, one can make the following concrete choice of Coxeter elements in the infinite families, illustrated in Figure~\ref{fig:Cox}: in $G(d, 1, n)$, take
\[
c := [(12\cdots n); (0, 0, \ldots, 0, 1)], 
\]
while in $G(d, d, n)$, take
\[
c_{(d,d,n)} := [(12\cdots(n - 1))(n); (0, \ldots, 0, 1, -1)].
\]

\begin{remark}
\label{rem:reflection automorphisms}
At least two nonequivalent definitions of Coxeter elements have
appeared in the literature (compare, e.g., the definitions in \cite[\S1]{BessisReiner} and \cite[\S2.3]{Theo}): under the more restrictive definition, every Coxeter element in $G(d, 1, n)$ (respectively, $G(d, d, n)$) is conjugate to the element $c$ (respectively, $c_{(d, d, n)}$) selected above.  Since conjugacy descends to a bijection between factorizations that preserves the fixed space dimension of each factor, it follows that all Coxeter elements (under the restrictive definition) have the same enumerations, and it is enough to consider just one.

Under the more general definition, one should also allow in $G(d, 1, n)$ the possibility of replacing the weight $1$ in $c$ with any cyclic generator of $\ZZ/d\ZZ$ (and then taking conjugates), and similarly for $G(d, d, n)$.  The resulting elements are \emph{not} all conjugate in the group $G(d, 1, n)$, so it is not \emph{a priori} clear that different Coxeter elements yield the same enumeration.  This difficulty may be resolved in two different ways.  One resolution is to examine the actual proofs presented in Sections~\ref{sec:G(d, 1, n)} and~\ref{sec:ddn} below.  Because of the combinatorial nature of these proofs, it is not difficult to see that they work equally well if the weights $1$ and $-1$ are replaced by $a$ and $-a$ for any nonzero element $a$ in $\ZZ/d\ZZ$, so all Coxeter elements in the more general sense (and even some elements that are not Coxeter elements under any definition) have the same enumeration.

An alternative, and more conceptual, resolution is based on the work
of Reiner--Ripoll--Stump \cite{RRS} that we describe now.  Given a
complex reflection group $G$ of rank $n$, understood to be represented
by a concrete choice of matrices in $\GL_n(\CC)$, define the
\emph{field of definition} $K_G$ to be the subfield of $\CC$ generated
by the traces of elements of $G$.  One can show that $G$ can be
represented over $\GL_n(K_G)$, i.e., that the representing matrices
can be taken to have entries in $K_G$.  The Galois group $\Gamma =
\operatorname{Gal}(K_G / \QQ)$ acts on $\GL_n(K_G)$ entrywise, and so
each member $\gamma$ of $\Gamma$ gives an isomorphism between $G$ and
some, possibly different, representation of $G$ over $K_G$.  By
\cite[Cor.\ 2.3]{RRS}, the group $\gamma(G)$ is conjugate in
$\GL_n(\CC)$ to $G$, i.e., there is some $g \in \GL_n(\CC)$ such that
$G = g \cdot \gamma(G) \cdot g^{-1}$.  Both $\gamma$ and conjugation
by $g$ preserve fixed-space dimension, so the combined automorphism $w
\mapsto g \cdot \gamma(w) \cdot g^{-1}$ does as well.  It is part of
the main result of \cite{RRS} that these \emph{reflection
  automorphisms} (called \emph{Galois automorphisms} in
\cite{MarinMichel}) act transitively on the Coxeter elements of $G$
under the more general definition.  These automorphisms descend to
bijections between factorizations that preserve fixed space dimension.  Consequently, the answers to the questions we consider are the same for all Coxeter elements, and it suffices to compute with a single, fixed Coxeter element.
\end{remark}

\subsubsection{Degrees and coexponents}
\label{sec:degrees and coexponents}

To each complex reflection group there are associated fundamental invariants of several kinds, two of which will appear below (particularly in Section~\ref{sec:exceptional}).  We define them now.

Much interest in complex reflection groups relates to their role in
invariant theory: the complex reflection groups of rank $n$ are
exactly the groups $G$ whose invariant ring $\CC[x_1, \ldots, x_n]^G$
is again a polynomial ring $\CC[f_1, \ldots, f_n]$, generated by $n$
algebraically independent homogeneous polynomials. (For example, the invariant ring $\CC[x_1,
\ldots, x_n]^{S_n}$ of the symmetric group is the ring of symmetric
polynomials in $n$ variables, generated over $\CC$ by the elementary
symmetric polynomials $\{e_1, \ldots, e_n\}$.) The basic
invariants $f_1$, \ldots, $f_n$ are not uniquely determined, but their
degrees $d_1 \leq d_2 \leq \dots \leq d_n$ are, and we call these the
\emph{degrees of $G$}.   For a well generated group, it is always the
case that $d_n$ is equal to the Coxeter number $h$ mentioned above.
For $G(d, 1, n)$, the degrees are given by $d_i = d \cdot i$, while for $G(d, d, n)$ they are given by $\{d_1, \ldots, d_n\} = \{d, 2d, \ldots, (n - 1)d\} \cup \{n\}$.  For any complex reflection group $G$, one has $d_1 \cdots d_n = |G|$ and, more generally \cite[5.3]{ShephardTodd},
\[
\sum_{w \in G} x^{\dim \fix(w)} = \prod_{i = 1}^n (x - 1 + d_i).
\]

A second sequence of invariants, the \emph{coexponents} of $G$, will
also appear below.  These may be equivalently defined in several ways:
in terms of invariant theory, they are degrees of generators appearing in the covariant space $(\CC[V] \otimes V)^G$; for a definition in terms of the hyperplane arrangement associated to $G$, see Remark~\ref{Orlik--Terao remark}.  
Perhaps the simplest definition is that the coexponents $e^*_1 \leq
e^*_2 \leq \ldots \leq e^*_n$ are the positive integers that satisfy
the identity \cite[(3.10)]{OrlikSolomon1980}
\[
\sum_{w \in G} \det(w) \cdot x^{\dim \fix(w)} = \prod_{i = 1}^n (x - e^*_i).
\]
One consequence of this formula is that the sum $\sum_{i = 1}^n e^*_i$ is equal to the number $|R^*|$ of reflecting hyperplanes of reflections in $G$.
For $G(d, 1, n)$, the coexponents are given by $e^*_i = 1 + d\cdot (i - 1)$, while for $G(d, d, n)$ they are given by $\{e^*_1, \ldots, e^*_n\} = \{1, d + 1, \ldots, (n - 2)d + 1\} \cup \{(n - 1)(d - 1)\}$.

\subsection{Counting factorizations with representation theory}
\label{sec:Frobenius}

The character-theory approach to counting factorizations is based on the following lemma, expositions of which appear in numerous sources, e.g., \cite[Ex.~7.67(b)]{EC2}.

\begin{lemma}[{{Frobenius \cite{Frobenius}}}]
\label{frobenius}
Let $G$ be a finite group, $g$ an element of $G$, and $A_1, \ldots, A_k$ subsets of $G$ that are closed under conjugacy by $G$.  Then the number of factorizations of $g$ as a product $g = u_1 \cdots u_k$ such that for each $i$ the factor $u_i$ is required to lie in the set $A_i$ is equal to
\[
\frac{1}{|G|} \sum_{\lambda \in \Irr(G)} \dim(\lambda) \chi_\lambda(g^{-1}) \normchi_\lambda(\zz_1) \cdots \normchi_{\lambda}(\zz_k),
\]
where $\Irr(G)$ is the set of irreducible complex representations of
$G$, $\dim(\lambda)$ is the dimension of the representation $\lambda$,
$\chi_\lambda$ is the character associated to $\lambda$ extended
linearly from the group to the group algebra, $\normchi_\lambda = \frac{\chi_\lambda}{\dim(\lambda)}$ is the normalized character associated to $\lambda$, and for $i = 1, \ldots, k$, $\zz_i$ is the formal sum in the group algebra of elements in $A_i$.
\end{lemma}

Thus, counting factorizations with no transitivity conditions can be reduced to
  a problem of being able to compute enough of the
  character table of the group under consideration.  This is precisely
  the approach followed by Jackson and by Chapuy--Stump, using
  respectively the character theory of the symmetric group and of complex reflection
  groups. We also consider factorizations with a transitivity condition. 
  Their counts can be written as a difference of two numbers of factorizations without
  transitivity conditions, where the character approach can be
  used. We make use of this character approach in the next subsection,
  as well as in the case of the exceptional complex reflection groups
  (Section~\ref{sec:exceptional}). 

\subsection{Factoring an $(n - 1)$-cycle in $\mathfrak{S}_n$}
\label{sec:n - 1 cycle}

If one factors an $(n - 1)$-cycle $\cnn$ in $\Symm_n$ as a product of
other permutations, there are two possibilities: either every factor
shares a fixed point with $\cnn$, or not.  The factorizations in the
former case are in natural bijection with factorizations of an
$(n-1)$-cycle in $\Symm_{n - 1}$.  The factorizations in the latter
case have a more elegant description: they are exactly the
factorizations in which the factors act transitively on the set
$[n]$. The study of transitive
factorizations is present already in the work of Hurwitz \cite{H}
from the late 19th century and plays an important role in the study of permutation
factorizations; see \cite{GJ-stanleyvol} for a recent survey.  
Our first result is to enumerate transitive factorizations of the $(n - 1)$-cycle.

\begin{theorem}
\label{S_n (n-1) cycle theorem k factors}
Let $\cnn$ be a fixed $(n-1)$-cycle in $\mathfrak{S}_n$. For integers
$r_1,\ldots,r_k$, let $b_{r_1,\ldots,r_k}$ be the number of $k$-tuples
$(\pi_1,\ldots,\pi_k)$ of elements in $\Symm_n$ such that $\pi_i$ has $r_i$
cycles for $i=1,\ldots,k$, $\pi_1 \cdots  \pi_k = \cnn$, and the tuple is a transitive factorization. Then 
\begin{align}
\label{n - 1 cycle k factors transitive factorizations}
\sum_{r_1,\ldots,r_k \geq 1} b_{r_1,
  \ldots, r_k} x_1^{r_1} \cdots x_k^{r_k} &= \frac{(n!)^{k - 1}}{n^k}\sum_{p_1,\ldots,p_k \geq
  1}  M^{n}_{p_1,\ldots,p_k} \frac{(x_1)_{p_1}}{(p_1-1)!} \cdots
  \frac{(x_k)_{p_k}}{(p_k-1)!},
\end{align}
where $M^n_{p_1,\ldots,p_k}$ is defined in \eqref{definition M coeff}.
\end{theorem}

\begin{remark}
We were surprised not to find this statement in the literature. We
give an algebraic proof.  In Section~\ref{S_n questions}, we give
a combinatorial proof in the case of $k = 2$ factors; it would be
of interest to find a combinatorial proof for all $k$.
\end{remark}

\begin{remark}
\label{rem:(n - 1)-cycle}
Theorem~\ref{S_n (n-1) cycle theorem k factors} can be interpreted as
a statement about colored factorizations: the LHS counts
transitive factorizations of $\cnn$ in which the cycles of factor $i$
are colored with any of $x_i$ colors, and the coefficient $C^{\langle
  n - 1, 1 \rangle}_{p_1, \ldots, p_k} := (n!)^{k - 1}\frac{p_1 \cdots
  p_k}{n^k} M^n_{p_1, \ldots, p_k}$ of $\binom{x_1}{p_1} \cdots
\binom{x_k}{p_k}$ on the RHS is the number of these factorizations in which a
prescribed set of $p_i$ colors is used in the $i$th factor. 
\end{remark}

In the proof of Theorem~\ref{S_n (n-1) cycle theorem k factors}, we assume a familiarity with symmetric
functions as in \cite[Ch.~7]{EC2}.  Let $p_{\lambda}$, $s_{\lambda}$ and $m_{\lambda}$ denote
the power sum, Schur and monomial symmetric
functions in the variables $\xx = \{x_1, x_2, \ldots\}$. We will use the {\em stable principal
  specializations} $\xx \mapsto 1^x$ of these functions, setting $x$ of the variables $\{x_i\}$ equal to $1$ and all others equal to $0$.  One has
\begin{equation} \label{eq: stable specialization}
p_{\lambda}(1^x) = x^{\ell(\lambda)}, \qquad s_{\lambda}(1^x) =
   \prod_{(i,j) \in \lambda} \frac{x+j-i}{h(i,j)}, \qquad \textrm{and} \qquad m_{\lambda}(1^x)
 = \binom{x}{\ell(\lambda)}
\frac{\ell(\lambda)!}{\Aut(\lambda)},
\end{equation}
where $h(i,j) :=\lambda_i -i
     +\lambda'_j-j+1$ is the {\em hook length} of the cell $(i,j)$ in
     the Young diagram of
     $\lambda$ and $\Aut(\lambda):=n_1!n_2!\cdots$ for $\lambda=\langle 1^{n_1},2^{n_2},\ldots \rangle$.  The middle specialization is the \emph{hook content formula}.

\begin{proof}[Proof of Theorem~\ref{S_n (n-1) cycle theorem k factors}] 
Let $\cnn$ be the fixed $(n-1)$-cycle $(1,2,\ldots,n-1)(n)$ and
let $d_{r_1,\ldots,r_k}$ be the number of $k$-tuples
$(\pi_1,\ldots,\pi_k)$ of elements in
$\mathfrak{S}_n$ such that $\pi_i$ has $r_i$ cycles for $i=1,\ldots,k$
 and $\pi_1\dots \pi_k = \cnn$. As mentioned above, the difference $c_{r_1,\ldots,r_k}:=d_{r_1,\ldots,r_k}-b_{r_1,\ldots,r_k}$ is the number of $k$-tuples $(\pi_1,\ldots,\pi_k)$ of
elements in $\mathfrak{S}_n$ such that $\pi_i$ has $r_i$ cycles,
$\pi_1\cdots \pi_k = \cnn$, and $n$ is a fixed point of each $\pi_i$. Thus
$c_{r_1,\ldots,r_k} = a_{r_1-1,\ldots,r_k-1}$ is the number of
factorizations of an $(n-1)$-cycle $\cn_{(n-1)}$ as a product of $k$ permutations in
$\mathfrak{S}_{n-1}$ such that the $i$th factor has $r_i - 1$ cycles. 

Since the set of permutations with prescribed number of cycles is closed
under conjugation, we apply the Frobenius formula  (Lemma~\ref{frobenius}) to compute
both $d_{r_1,\ldots,r_k}$ and $c_{r_1,\ldots,r_k}$:
\begin{align*}
d_{r_1,\ldots,r_k} &= \frac{1}{n!} \sum_{\lambda \in
                     \Irr(\mathfrak{S}_n)} \dim(\lambda)
                     \chi_\lambda(\cnn^{-1}) \normchi_\lambda(\zz_{r_1})
                     \cdots \normchi_{\lambda}(\zz_{r_k}),\\
c_{r_1,\ldots,r_k} &= \frac{1}{(n-1)!} \sum_{\lambda \in
                     \Irr(\mathfrak{S}_{n-1})} \dim(\lambda)
                     \chi_\lambda(\cn_{(n-1)}^{-1}) \normchi_\lambda(\zz_{r_1-1})
                     \cdots \normchi_{\lambda}(\zz_{r_k-1}),
\end{align*}
where $\zz_r$ is the formal sum in the group algebra of the symmetric
group (of size $n$ or size $n-1$ depending on the context) of all elements with $r$
cycles. Thus the generating functions $G({\bf x}):= 
\sum_{r_1,\ldots,r_k} d_{r_1,\ldots,r_k} x_1^{r_1}\cdots x_k^{r_k}$
and $G'({\bf x}) := \sum_{r_1,\ldots,r_k} c_{r_1,\ldots,r_k}
x_1^{r_1}\cdots x_k^{r_k}$ are given by
\begin{align}
G({\bf x}) &= \frac{1}{n!} \sum_{\lambda \in
                     \Irr(\mathfrak{S}_n)} \dim(\lambda)
                     \chi_\lambda(\cnn^{-1}) g_{\lambda}(x_1)\cdots
  g_{\lambda}(x_k),  \label{eq: gf all n-1 cycle factorizations}\\
G'({\bf x}) &= \frac{x_1\cdots x_k}{(n-1)!}
\sum_{\lambda \in
                     \Irr(\mathfrak{S}_{n-1})} \dim(\lambda)
                     \chi_\lambda(\cn_{(n-1)}^{-1}) g_{\lambda}(x_1)
  \cdots g_{\lambda}(x_k) \label{eq: gf all non transitive n-1 cycle factorizations},
\end{align}
where 
\[
g_{\lambda}(x) := \sum_{k=1}^n \normchi_{\lambda}(\zz_k) x^k. 
\]
By the Murnaghan--Nakayama rule \cite[Thm.~7.17.1]{EC2}, one has 
$\chi_\lambda(\cnn^{-1})=0$  unless $\lambda$ equals $\langle n\rangle$, $\langle 1^n\rangle$,
or one of the {\em near hooks} $\langle n-m-1,2,1^{m-1} \rangle$. These representations
have dimensions $\dim(\lambda)=1, 1$, and $\frac{(n-2-m)m}{(n-1)!}
\binom{n-2}{m}$  and character values $\chi_{\lambda}(\cnn^{-1}) = 1,
(-1)^n$, and $(-1)^m$, respectively.
Similarly,
$\chi_{\lambda}(\cn_{(n-1)}^{-1})=0$ unless $\lambda$ equals a hook
$\langle n-1-m,1^m\rangle$. These representations have dimension $\dim(\lambda)=\binom{n-2}{m}$
and character values $\chi_{\lambda}(\cn^{-1}_{(n-1)})=(-1)^m$. Thus
\eqref{eq: gf all n-1 cycle factorizations} and \eqref{eq: gf all non
  transitive n-1 cycle factorizations} become
\begin{align}
\label{eq:JacksonSn-1}
G({\bf x}) &=
                                                                     \frac{1}{n!}\Bigg(
                                                                     g_{\langle n\rangle}(x_1)\cdots 
  g_{\langle n\rangle}(x_k) + (-1)^n g_{\langle 1^n\rangle}(x_1)\cdots g_{\langle 1^n\rangle}(x_k) \\
 & + \sum_{m=1}^{n-3}
  \frac{(n-2-m)m}{n-1} \binom{n}{m+1} (-1)^m g_{\langle n - m - 1,2,
   1^{m-1}\rangle}(x_1)\cdots  g_{\langle n - m - 1,2,
   1^{m-1}\rangle}(x_k) \Bigg), \notag \\
G'({\bf x}) &= \frac{x_1\cdots x_k}{(n-1)!}
\sum_{m=0}^{n-2} (-1)^m \binom{n-2}{m} g_{\langle n-m-1,1^m\rangle}(x_1)\cdots
            g_{\langle n-m-1,1^m\rangle}(x_k). \label{eq: gf all non
  transitive n-1 cycle factorizations simplified} 
\end{align}
The next lemma evaluates the $g_{\lambda}(x)$ when $\lambda$ is a hook or near hook.

\begin{lemma} \label{lemm:Sncasehooks}
We have that
\begin{align}
\label{eq:hooks}  g_{\hook{m}{n}}(x) &= (x-m)(x-m+1)\cdots (x-m+n-1) 
     = n! \sum_{k=m+1}^n \binom{n-1-m}{k-1-m} \cdot \binom{x}{k}, \\
g_{\langle n-m-1,2,1^{m-1}\rangle }(x)  &= x \cdot g_{\hook{m}{n-1}}(x) \label{eq:near-hooks}.
\end{align}
\end{lemma}

\begin{proof}
By the stable principal specialization \eqref{eq: stable specialization} of $p_\lambda$, the expansion of a Schur function into
power sum symmetric functions, and the hook length formula for $\dim(\lambda)$, we have that $g_{\lambda}(x)$ equals a stable
principal specialization of a Schur function scaled by the product $H_{\lambda}$ of
the hook lengths of $\lambda$. That is,
\begin{equation} \label{eq: g is spe of schurs}
g_{\lambda}(x) = \sum_{\mu} \frac{n!}{z_{\mu}} \cdot
\normchi_{\lambda}(\mu)\cdot 
p_{\mu}(1^x) = H_{\lambda} \cdot s_{\lambda}(1^x).
\end{equation}
The first equality of \eqref{eq:hooks} then follows by
applying the stable principal specialization \eqref{eq: stable specialization} of Schur functions for $\lambda=\hook{m}{n}$. 
The second equality of \eqref{eq:hooks} is
obtained by using the Chu--Vandermonde identity.\footnote{Alternatively, one is expanding $H_{\hook{m}{n}}\cdot s_{\hook{m}{n}}$ into the monomial
basis and doing a stable principle specialization.}
Next we consider \eqref{eq:near-hooks}. By \eqref{eq: g is spe of schurs} we have that $g_{\langle n-m-1,2,1^{m-1}\rangle}(x)$
is the stable principal specialization of $H_{\langle n-m-1,2,1^{m-1}\rangle} \cdot
s_{\langle n-m-1,2,1^{m-1}\rangle}$.  By the hook-content formula and the first
equality of \eqref{eq:hooks} we obtain 
\begin{align*}
g_{\langle n-m-1,2,1^{m-1}\rangle}(x) &=  H_{\langle n-m-1,2,1^{m-1}\rangle} \cdot
s_{\langle n-m-1,2,1^{m-1}\rangle}(1^x) \\
&= x \cdot (x-m)(x-m+1)\cdots
                                       (x-m+n-2)\\
&= x \cdot g_{\hook{m}{n-1}}(x). \qedhere
\end{align*}
\end{proof}

We continue with the proof of Theorem~\ref{S_n (n-1) cycle theorem k factors}. Let $\widetilde{G}({\bf
  x}) := \sum_{r_1,\ldots,r_k}
  b_{r_1,\ldots,r_k} x_1^{r_1} \cdots x_k^{r_k}$. Since
  $\widetilde{G}({\bf x})$ is the difference of $G({\bf x})$ and
  $G'({\bf x})$, using \eqref{eq:JacksonSn-1}, \eqref{eq: gf all non
  transitive n-1 cycle factorizations simplified}  with
  \eqref{eq:near-hooks} 
  gives
\begin{multline*}
\widetilde{G}({\bf x}) 
  =   G({\bf x}) - G'({\bf x}) 
= \frac{1}{n!}\Bigg( g_{\langle n\rangle}(x_1)\cdots 
  g_{\langle n\rangle}(x_k) + (-1)^n g_{\langle 1^n\rangle}(x_1)\cdots g_{\langle 1^n\rangle}(x_k) \\
  - x_1\cdots x_k\sum_{m=0}^{n-2}
 \binom{n}{m+1}(-1)^m g_{\langle n - m - 1, 1^{m}\rangle}(x_1)\cdots  g_{\langle n - m - 1, 1^{m}\rangle}(x_k) \Bigg).
\end{multline*}
We rewrite the expression in terms of the basis $(x_1)_{p_1}\cdots (x_k)_{p_k}$.  In the case 
$p_1=\cdots=p_k=n$, we see directly from the definition of $b_{\pp}$ that $b_{n,n,\ldots,n}=0$, and so
\[
\left[\frac{(x_1)_n}{n!}\cdots \frac{(x_k)_n}{n!}\right]
  \,\,\widetilde{G}({\bf x}) \, = \, 0.
\]
For any other tuple $p_1,\ldots,p_k\geq 1$ we use
\eqref{eq:hooks} 
and obtain
\begin{multline*}
\left[\frac{(x_1)_{p_1}}{p_1!}\cdots \frac{(x_k)_{p_k}}{p_k!}\right]
  \,\,\widetilde{G}({\bf x})  = \\ (n!)^{k-1}
                                        \binom{n-1}{p_1-1}\cdots \binom{n-1}{p_k-1} + 
\frac{(n-1)!^k}{n!}   p_1\cdots p_k\sum_{m=0}^{n-2}
                                        (-1)^{m+1}\binom{n}{m+1}
                                        \binom{n-1-m}{p_1-1-m}\cdots
                                        \binom{n-1-m}{p_k-1-m}.
\end{multline*}
The first term on the RHS corresponds to the term $m=-1$ in
the sum. We absorb this term to the sum and reindex it using $t=m+1$
to obtain
\begin{align*}
\left[\frac{(x_1)_{p_1}}{p_1!}\cdots \frac{(x_k)_{p_k}}{p_k!}\right]
  \,\,\widetilde{G}({\bf x}) & =
\frac{(n-1)!^k}{n!}   p_1\cdots p_k\sum_{t=0}^{n-1}
                                        (-1)^{t}\binom{n}{t}
                                        \binom{n-t}{p_1-t}\cdots
                                        \binom{n-t}{p_k-t} \\
& = \frac{(n-1)!^k}{n!}   p_1\cdots p_k M^n_{p_1,\ldots,p_k},
\end{align*} 
where $M^n_{p_1,\ldots,p_k}$ is 
defined in \eqref{definition M coeff}. This yields the desired result.
\end{proof}

\section{Factorization results for the group $G(d,1,n)$}
\label{sec:G(d, 1, n)}

We recall the statement our main enumerative theorem  for $G(d, 1, n)$.

\begin{theorem}
\label{G(d, 1, n) many factors theorem}
For $d > 1$, let $G = G(d, 1, n)$, let $c$ be a fixed Coxeter element in $G$, and let $a^{(d)}_{r_1,\ldots,r_k}$ be the number of factorizations of $c$ as a product of $k$ elements of $G$ with fixed space dimensions $r_1,\ldots,r_k$,
respectively.  Then
\[
 \sum_{r_1,\ldots,r_k}
a^{(d)}_{r_1,\ldots,r_k} x_1^{r_1} \cdots x_k^{r_k} = 
|G|^{k-1} \cdot \sum_{p_1,\ldots,p_k \geq 0}
M^n_{p_1,\ldots,p_k} \frac{(x_1 - 1)^{(d)}_{p_1}}{d^{p_1}p_1!} \cdots \frac{(x_k - 1)^{(d)}_{p_k}}{d^{p_k}p_k!},
 \]
where the polynomial $(x - 1)^{(d)}_p$ is defined in \eqref{first basis} and $M^n_{p_1,\ldots,p_k}$ is defined in \eqref{definition M coeff}.
\end{theorem}

The case of two factors (in Theorem~\ref{thm: main G(d,1,n)}) follows immediately as a corollary, taking $k = 2$ in Theorem~\ref{G(d, 1, n) many factors theorem} and using the fact that  $M^n_{p_1,p_2} = \binom{n}{p_1; p_2; n-p_1-p_2}$.

The main step in the proof of Theorem~\ref{G(d, 1, n) many factors theorem} is a lemma involving certain cycle-colored factorizations of the element $c$, which we now describe.  (The remainder of the proof, which is in the spirit of Proposition~\ref{prop:change of basis}, follows the proof of the lemma.)
\begin{definition}
\label{def:coloring}
Given a nonnegative integer $x$ and an element $u$ in $G(d, 1, n)$, let $\chi$ denote the \emph{color set} $\chi = \{0, 1, 2, \ldots, xd\}$.  Within the color set $\chi$, a \emph{$d$-strip} is any of the following collections of $d$ consecutive colors: $\{1,\ldots,d\}$, $\{d+1,\ldots,2d\}$, \ldots, $\{(x-1)d+1,\ldots,xd\}$.  Thus, there are exactly $x$ disjoint $d$-strips in $\chi$, and the color $0$ does not belong to any $d$-strip.

A \emph{cycle-coloring} of $u$ is an assignment of a color in $\chi$ to each cycle of $u$ so that cycles of nonzero weight receive color $0$.  Given nonnegative integers $x_1, \ldots, x_k$, a \emph{colored factorization} of an element $g$ of $G(d, 1, n)$ is a factorization $g = u_1 \cdots u_k$ together with a cycle-coloring of $u_i$ by color set $\chi_i = \{0, 1, \ldots, x_i d\}$ for $i = 1, \ldots, k$.

Given nonnegative integers $p_1, \ldots, p_k$, let $\CCC^{(d, 1, n)}_{p_1,\ldots,p_k}$ be the set of colored factorizations of the Coxeter element $c$ in $G(d, 1, n)$ with color sets $\chi_i = \{0, \ldots, p_i d\}$ so that at least one color from each $d$-strip in $\chi_i$ is actually used to color a cycle in $u_i$.  Let $C^{(d, 1, n)}_{\pp} = |\CCC^{(d, 1, n)}_{\pp}|$. 
\end{definition}

\begin{remark}
Observe that for factorizations in $\CCC^{(d, 1, n)}_{\pp}$, the coloring of cycles of different factors is completely independent: the requirement in the definition that each strip be used is factor-by-factor, so whether a color from the strip $\{1, 2, \ldots, d\}$ (for example) is used to color a cycle in $u_1$ has no bearing on the requirement that a color from that strip be used to color a cycle in $u_2$ (if $p_2 \geq 1$).
\end{remark}

The first part of the proof of Theorem~\ref{G(d, 1, n) many factors theorem} is a formula for $C^{(d, 1, n)}_{\pp}$ in terms of the counts $C^{\langle n \rangle}_{{\bf q}}$ of colored factorizations of the $n$-cycle in $\Symm_n$ introduced just before Proposition~\ref{prop:change of basis}.

\begin{lemma}
\label{G(d, 1, n) lemma}
For any ${\pp} = (p_1, \ldots, p_k)$ in $\NN^k$, we have
\[
C^{(d, 1, n)}_{\pp} = d^{(k-1)n} \sum_{\varnothing \neq S\subseteq [k]}
C^{\langle n \rangle}_{\pp + \ee_S}
\]
where $\ee_S$ is the indicator vector for $S$.
\end{lemma}

\begin{proof}
Given a colored factorization $c = u_1 \cdots u_k$ of the Coxeter
element $c$ for $G(d, 1, n)$, we associate to it
a colored factorization $\cn = \pi_1 \cdots \pi_k$ of the $n$-cycle $\cn$ in $\Symm_n$, as follows: $\pi_i$ is the projection of $u_i$ in $\Symm_n$; if a cycle of $u_i$ is colored with a color in the $d$-strip $\{(a - 1)d + 1, \ldots, ad\}$, then the corresponding cycle of $\pi_i$ is colored with color $a$; if a cycle of $u_i$ is colored with color $0$, then the corresponding cycle of $\pi_i$ is colored with color $0$.  Thus, in the resulting colored factorization of $\cn$, the $i$th factor is colored in either $p_i$ or $p_i + 1$ colors, with every color appearing.  Let $S \subseteq [k]$ denote the set of indices $i$ such that $\pi_i$ is colored in $p_i + 1$ colors (rather than $p_i$); equivalently, it is the set of indices $i$ such that some cycle of $u_i$ is colored with $0$.

First, we observe that $S \neq \varnothing$: the product $c = u_1 \cdots u_k$ has nonzero total weight, so at least one of the factors $u_i$ has
nonzero weight, and this factor must have a cycle with nonzero weight. Such a cycle is colored with
the special color $0$, and so at least for this value of $i$ we have $i \in S$, as claimed. 
Thus the image of the set $\CCC^{(d, 1, n)}_{\pp}$ of colored factorizations of the Coxeter element $c$ under projection is contained in the disjoint union of pieces
$\bigsqcup_{S} \CCC^{\langle n\rangle}_{{\pp} + {\ee}_S}$ for nonempty
sets $S\subseteq [k]$, where each piece $\CCC^{\langle
  n\rangle}_{{\pp} + {\ee}_S}$ consists of colored factorizations of the
$n$-cycle $\cn$ with the appropriate set of colors used in each
factor (see Remark~\ref{remark: independence color sets}).  

Second, we consider how many preimages each factorization in
$\CCC^{\langle n\rangle}_{{\pp} + {\ee}_S}$ has under this map.  To choose a
preimage, we must assign weights $(a_{i, 1}, \ldots, a_{i, n})$ to the entries of each factor $\pi_i$
in such a way that the product of the resulting factors $u_i$ really
is the Coxeter element $c$, \emph{and} so that in each $u_i$, any
cycle of nonzero weight was originally colored by the color $0$;
and we must choose one of $d$ colors from a $d$-strip for each of the
cycles in $u_i$ that corresponds to a cycle in $\pi_i$ of nonzero
color.

In order to do this, we consider a too-large set of colored factorizations in $G(d, 1, n)$, initially disregarding the requirement that the factored element be $c$.  Given a colored factorization $\cn = \pi_1 \cdots \pi_k$ in $\CCC^{\langle n\rangle}_{{\pp} + {\ee}_S}$ with $S$ nonempty, choose a total order on the set $\{(i, m) \colon i \in [k], m \in [n]\}$ of indices of weights to be assigned, in such a way that the last index in the total order belongs to a cycle of color $0$ in some factor.  (Such cycles must exist, since $S \neq \varnothing$.)  Then we assign values to the weights one-by-one according to the chosen order, choosing the weights arbitrarily except in two cases: if an element belongs to a cycle of nonzero color and is the last element (in the given total order) in its cycle, we assign it the unique weight so that the total weight of its cycle is $0$; and we choose the weight of the specially selected final element so that the total weight of all elements is $1$.  (These two exceptions never conflict because the special element was chosen in a cycle of color $0$.)  The number of ways to perform these choices is $d^{nk - \#(\textrm{colored cycles}) - 1}$.  Finally, for each cycle of $\pi_i$ that is colored some nonzero color, there are $d$ choices for the color in the associated $d$-strip of the corresponding cycle of the lift $u_i$ of $\pi_i$; this contributes a factor of $d^{\#(\textrm{colored cycles})}$, for a total of $d^{nk- 1}$ lifts of the fixed $\Symm_n$-factorization.

Each lift is a colored factorization $u_1 \cdots u_k$ in $G(d, 1, n)$ of some element $c^*$ of weight $1$ whose underlying permutation is the $n$-cycle $\cn$, and every colored factorization of such a $c^*$ with the correct collection of colors is produced by such a lift.  The number of such $c^*$ is $d^{n - 1}$; they are exactly the elements conjugate to $c$ by some diagonal matrix $a$ in $G(d, 1, n)$. Moreover, since $a$ is a diagonal matrix, conjugating any $w \in G(d, 1, n)$ by $a$ preserves the weight of every cycle of $w$.  Consequently, conjugation by $a$ extends to a bijection between factorizations of $c$ and factorizations of $c^*$ that respects the underlying permutation of each factor and the weight of each cycle of each factor.  Thus, it gives in particular a bijection between the lifts of $\pi_1 \cdots \pi_k$ that factor $c$ and those that factor $c^*$.  Hence, of the total $d^{nk - 1}$ lifts, exactly $\frac{1}{d^{n - 1}} \cdot d^{nk - 1} = d^{n(k - 1)}$ of them are factorizations of $c$.  Since this holds for every nonempty $S \subseteq [k]$ and every factorization in $\CCC^{\langle n\rangle}_{{\pp} + {\ee}_S}$, the lemma is proved.
\end{proof}

\begin{proof}[Proof of Thm.~\ref{G(d, 1, n) many factors theorem}]
For any nonnegative integer $p$, we have $((xd + 1) - 1)^{(d)}_p = d^p p! \cdot \binom{x}{p}$.  
Therefore, the desired statement is equivalent to the equality
\begin{equation}\label{key identity for G(d, 1, n)}
\sum_{r_1,\ldots,r_k} a^{(d)}_{r_1,\ldots,r_k} (x_1d+1)^{r_1} \cdots
(x_kd+1)^{r_k} = d^{(k - 1)n}(n!)^{k-1} \sum_{p_1,\ldots,p_k} M^n_{p_1,\ldots,p_k}
\binom{x_1}{p_1} \cdots \binom{x_k}{p_k}.
\end{equation}
Consider the case that each $x_i$ is a nonnegative integer.
In this case, the LHS of \eqref{key identity for G(d, 1, n)} exactly counts the colored factorizations $c = u_1 \cdots u_k$ of the Coxeter element $c$ in which the cycles of factor $u_i$ are colored with color set $\chi_i = \{0, 1, \ldots, x_id\}$ so that cycles of nonzero weight receive color $0$.  Now, we count those factorizations by the number of $d$-strips that are actually used: for nonnegative integers $p_1, \ldots, p_k$, there are $\binom{x_1}{p_1} \cdots \binom{x_k}{p_k}$ ways to choose $p_i$ $d$-strips to use in the $i$th factor, and $C^{(d, 1, n)}_{p_1,\ldots,p_k}$ colored factorizations using exactly these strips.  Thus
\[
\sum_{r_1,\ldots,r_k} a^{(d)}_{r_1,\ldots,r_k} (x_1d+1)^{r_1} \cdots
(x_kd+1)^{r_k} = \sum_{p_1,\ldots,p_k} C^{(d, 1, n)}_{p_1,\ldots,p_k}
\binom{x_1}{p_1} \cdots \binom{x_k}{p_k}.
\]
By Lemma~\ref{G(d, 1, n) lemma}, 
\[
C^{(d, 1, n)}_{\pp} = d^{(k-1)n} \sum_{\varnothing \neq S\subseteq [k]}
C^{\langle n \rangle}_{{\pp} + {\ee}_S}.
\]
By Corollary~\ref{S_n corollary}, we can rewrite this as
\[
C^{(d, 1, n)}_{\pp} = d^{(k-1)n}(n!)^{k - 1} \sum_{\varnothing \neq S\subseteq [k]} M^{n-1}_{{\pp} -{\bf 1} + {\ee}_S}
\]
where $\bf 1$ is the all-ones vector.  Then the desired equality follows by Proposition~\ref{prop:recurrence Ms}.  Finally, since this identity is valid for all nonnegative integer values of the $x_i$, it is valid as a polynomial identity, as well.  This completes the proof.
\end{proof}

\section{Factorization results for the subgroup $G(d,d,n)$}
\label{sec:ddn}

As in the case of the $(n - 1)$-cycle in $\Symm_n$, factorizations of
a Coxeter element in $G(d, d, n)$ can be separated into two classes
based on a transitivity property that we describe now.  The wreath
product $G(d, 1, n)$ carries a natural permutation action: it acts on
$d$ copies of $[n]$ indexed by $d$th roots of unity, or equivalently,
on the set $\{z^i e_j \colon 0 \leq i < d, 1 \leq j \leq n\}$ where
$z$ is a primitive $d$th root of unity and $e_j$ are the standard
basis vectors for $\CC^n$.  The Coxeter elements for $G(d, 1, n)$ act
transitively on this set, and consequently every factorization of a
Coxeter element in $G(d, 1, n)$ is a transitive factorization.
However, the same is not true for the subgroup $G(d, d, n)$, where the
underlying permutations of the Coxeter elements are $(n - 1)$-cycles.  The action of the Coxeter element $c_{(d, d, n)}$ divides $\{z^i e_j \colon 0 \leq i < d, 1 \leq j \leq n\}$ into two orbits, $\{z^i e_j \colon 0 \leq i < d, 1 \leq j < n\}$ and $\{z^i e_n \colon 0 \leq i < d\}$.  Thus, a factorization of $c_{(d, d, n)}$ will be transitive if and only if some factor sends an element of the second orbit to an element of the first, or equivalently if the underlying factorization in $\Symm_n$ is transitive. In enumerating factorizations of $c_{(d, d, n)}$, we handle the transitive and nontransitive factorizations separately.

\subsection{Transitive factorizations}

As mentioned in the introduction, our generating function in this case is in terms of the polynomials $\poly{k}{d}{x}$ defined by 
$\poly{0}{d}{x} = 1$, $\poly{1}{d}{x} = x$, and for $k > 1$
\begin{equation} 
\label{eq: equivalent expressions G(d,d,n) basis}
\poly{k}{d}{x} :=  \prod_{i = 1}^k (x - e^*_i) = (x - (k - 1)(d - 1)) \cdot (x - 1)^{(d)}_{k - 1}  = (x-1)^{(d)}_k + k(x-1)^{(d)}_{k-1}
\end{equation}
where the $e^*_i$ are the coexponents of the group $G(d, d, k)$.  Next, we recall the statement to be proved.

\begin{theorem} \label{thm:mainG(d,d,n) k factors}
For $d > 1$, let $G = G(d, d, n)$ and let
 $b^{(d)}_{r_1,\ldots,r_k}$ be the number of transitive factorizations of a Coxeter element $c_{(d, d, n)}$ in $G$ as a product of $k$
elements of $G$ with fixed space dimensions $r_1,\ldots,r_k$, respectively. 
Then
\[
\sum_{r_1,\ldots,r_k
  \geq 0} b^{(d)}_{r_1,\ldots,r_k} x_1^{r_1}\cdots x_k^{r_k}
 = \frac{|G|^{k-1}}{n^k}\sum_{p_1,\ldots,p_k \geq 1}
 M^n_{p_1,\ldots,p_k} \frac{\poly{p_1}{d}{x_1}}{d^{p_1-1} (p_1-1)!} \cdots 
 \frac{\poly{p_k}{d}{x_k}}{d^{p_k-1} (p_k-1)!},
 \]
where $M^n_{p_1,\ldots,p_k}$ is as defined in \eqref{definition M coeff}.
 \end{theorem}

The case of two factors (in Theorem~\ref{thm:mainG(d,d,n)}) follows immediately as a corollary, taking $k = 2$ in Theorem~\ref{thm:mainG(d,d,n) k factors} and using the fact that $M^{n}_{p,q} = \binom{n}{p; q; n-p-q} = \frac{n(n - 1)}{pq} \binom{n - 2}{p - 1; q - 1; n - p - q}$.

As in Section~\ref{sec:G(d, 1, n)}, we split the proof into two parts.  The first concerns colored factorizations of the kind defined in Definition~\ref{def:coloring}.  Fix the standard $(n - 1)$-cycle $\cnn$ in $\Symm_n$ and Coxeter element $c_{(d, d, n)} = [\cnn; (0, \ldots, 0, 1, -1)]$ in $G(d, d, n)$.  Given nonnegative integers $p_1, \ldots, p_k$, let $C^{(d, d, n)}_{p_1,\ldots,p_k}$ be the number of colored transitive factorizations $c_{(d, d, n)} = u_1 \cdots u_k$ in $G(d, d, n)$ with color sets $\chi_i = \{0, 1, \ldots, p_id\}$ so that at least one color from each $d$-strip in $\chi_i$ is actually used to color a cycle in $u_i$.

\begin{lemma}
\label{(d, d, n) transitive lemma}
We have
\begin{equation}
\label{ddn equation to be proved}
C^{(d, d, n)}_{p_1,\ldots,p_k} = \sum_{S\subseteq [k]} d^{(k-1)n-|S|+1} C^{\langle n - 1, 1 \rangle}_{\pp + \ee_S}, 
\end{equation}
where ${\ee}_S$ is the indicator vector for $S$ and $C^{\langle n - 1, 1 \rangle}_{\pp} =
(n!)^{k - 1}\frac{p_1 \cdots p_k}{n^k} \cdot M^n_{\pp}$ is
the coefficient of $\binom{x_1}{p_1}\cdots \binom{x_k}{p_k}$ in the
RHS of \eqref{n - 1 cycle k factors transitive factorizations}. 
\end{lemma}
\begin{proof}
As in the proof of Lemma~\ref{G(d, 1, n) lemma}, we use the natural projection from colored factorizations of this sort to colored factorizations of the $(n - 1)$-cycle $\cnn$ in $\Symm_n$, where under projection a cycle that is colored by a color in the $d$-strip $\{(a - 1)d + 1, \ldots, ad\}$ gets sent to a cycle colored $a$, while a cycle colored by color $0$ gets sent to a cycle of color $0$.  If the original factorization is colored with colors $\pp$, the projected factorization is colored with colors ${\pp} + {\ee}_S$ for some subset $S \subseteq [k]$.  

Fix a subset $S \subseteq [k]$ and fix a colored transitive factorization $\pi_1 \cdots \pi_n = \cnn$ using ${\pp} + {\ee}_S$ colors.  We count preimages of this factorization under the projection.  As before, we consider a too-large set of colored factorizations, initially disregarding the requirement that the factored element be $c_{(d, d, n)}$.   We start by describing a total order on the $nk$-element set $\{(i, m) \colon i \in [k], m \in [n]\}$ of indices of weights to be chosen; we will then see that the elements may be assigned weights in $\ZZ/d\ZZ$ in this order in such a way that the number of valid choices of weights for each entry does not depend on earlier selections.

First, consider the $k$ indices $T = \{(k,n), (k - 1, \pi_k(n)), (k - 2, \pi_{k-1}\pi_k(n)), \ldots, (1, \pi_2\cdots \pi_k(n))\}$.  These are the coordinates on the ``thread'' connecting $n$ to $n$ in the braid diagram of the factorization -- see Figure~\ref{fig:thread}.  Say that the pair $(i, \pi_{i + 1} \cdots \pi_k(n))$ is \emph{problematic} if the number $\pi_{i + 1} \cdots \pi_k(n)$ is a fixed point of $\pi_i$; extend the adjective ``problematic'' to the $1$-cycle $\left(\pi_{i + 1} \cdots \pi_k(n)\right)$ of $\pi_i$.  Since we started with a \emph{transitive} factorization of $\cnn$, not all values in $T$ can be problematic.  The $k$ values in $T$ will form the first $k$ values in our linear order; moreover, we choose a nonproblematic value to be the last among these $k$.
\begin{figure}
\centering
\includegraphics{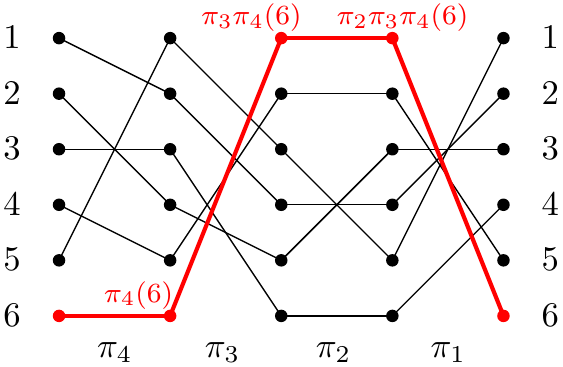}
\caption[A transitive factorization in $\Symm_6$.]{The transitive factorization 
of the $5$-cycle $\cn_{(5, 1)} = (12345)(6)$ in $\Symm_6$.  The thread involving the fixed point $6$ is highlighted.}
\label{fig:thread}
\end{figure}

For any $i$ and any nonproblematic cycle of $\pi_i$, there is an index $(i, m)$ corresponding to an entry of this cycle that is not among those already selected: at most one entry was selected from each cycle, and the only $1$-cycles from which an entry was selected are the problematic ones, by definition.  Therefore, for each $i$, we may select from every nonproblematic cycle of $\pi_i$ an index pair $(i, m)$ that was not already selected in the previous step.  These values form the \emph{last} values in our linear order.  Moreover, for each $i$ such that the color $0$ is used in a nonproblematic cycle in $\pi_i$, we arrange our order so that the index from one such cycle is last among the indices coming from that factor $\pi_i$.

All the remaining indices go in the middle, in any order.

Now we assign weights one-by-one, according to the selected total order.  We choose these arbitrarily, with the following exceptions:
\begin{itemize}
\item among the first $k - 1$ indices, a problematic index is assigned weight $0$ if its cycle has nonzero color, or if its cycle has color $0$ \emph{and} this is the only cycle of color $0$ in its factor;
\item the $k$th entry is assigned the unique value so that the sum of the first $k$ weights assigned (those that appear in the special thread connecting $n$ to $n$) is $-1$;
\item an index that is the last (in the total order) in its cycle, and this cycle has nonzero color, is assigned the unique weight so that its cycle has weight $0$; and
\item an index $(i, m)$ that is the last (in the total order) in its factor $\pi_i$, and $\pi_i$ has a nonproblematic cycle of color $0$, is assigned the unique weight so that the weights assigned to $\pi_i$ sum to $0$.
\end{itemize}
One can check that such an assignment of weights is always well defined (i.e., the different cases are disjoint); that the resulting product $u_1 \cdots u_k$ is equal to an element $c^*_{(d, d, n)}$ that is conjugate to $c_{(d, d, n)}$ by a diagonal matrix; and moreover that every transitive factorization of every such element $c^*_{(d, d, n)}$ arises from this construction.

Now we turn this into a counting argument.  By splitting the first bullet above into two cases, we see that each of the $nk$ indices has exactly $d$ choices of weight, with the following exceptions:
\begin{enumerate}
\item the (unique) values in problematic cycles with nonzero color;
\item the (unique) values in problematic cycles with color $0$ for which there are no other cycles of color $0$ in the permutation;
\item the $k$th value;
\item the last value in each nonproblematic cycle of nonzero color; and
\item the last value in each permutation that has a nonproblematic cycle of color $0$.
\end{enumerate}
Observe that the number of values in (1) and (4) together is exactly
the number of cycles with nonzero color, while the number of values in
(2) and (5) together is exactly the number of factors that have cycles
of color $0$.  Therefore, the total contribution from the weighting is 
\[
d^{nk - \#(\textrm{nonzero colored cycles}) - \#(\textrm{factors with cycle
    colored 0}) - 1}.
\]

In addition, each colored cycle gets assigned one of $d$ colors from its strip, hence we must multiply by $d^{\#\textrm{nonzero colored cycles}}$.  Since the number of factors with a cycle colored $0$ is exactly $|S|$, we have by Remark~\ref{rem:(n - 1)-cycle} that the total number of factorizations produced in this way is
\[
\sum_{S \subseteq [k]} d^{nk - |S| - 1} C^{\langle n - 1, 1 \rangle}_{{\pp} + {\ee}_S}.
\]
Finally, we must account for the fact that we have counted not merely factorizations of the Coxeter element $c_{(d, d, n)}$, but also factorizations of all of its conjugates by diagonal matrices.  The number of these is precisely $d^{n - 2}$, and as in the case of $G(d, 1, n)$, conjugation by a diagonal matrix is a bijection from factorizations to factorizations that preserves weights of all cycles, so they all have the same numbers of colored factorizations.  Consequently
\[
C^{(d, d, n)}_{\pp} = \sum_{S \subseteq [k]} d^{n(k - 1) - |S| + 1} C^{\langle n - 1, 1 \rangle}_{{\pp} + {\ee}_S},
\]
as claimed.
\end{proof}

\begin{proof}[Proof of Theorem~\ref{thm:mainG(d,d,n) k factors}]
First consider the RHS of the equation to be proved.  Make the substitution $x_i \mapsto x_id + 1$ for each $i$.
From \eqref{eq: equivalent expressions G(d,d,n) basis} we have that 
\[
 \frac{\poly{r}{d}{xd+1}}{d^{r-1} r!} =  d\binom{x}{r} + \binom{x}{r-1}.
\]
Therefore, we have
\begin{align*}
\frac{|G|^{k - 1}}{n^k} 
\sum_{p_1,\ldots,p_k} M^n_{p_1,\ldots,p_k} 
\prod_{i = 1}^k \frac{\poly{p_i}{d}{x_id + 1}}{d^{p_i-1} (p_i - 1)!} %\cdots \frac{\poly{p_k}{d}{x_kd + 1}}{d^{p_k-1} (p_k - 1)!}
%& = d^{(n - 1)(k - 1)} 
%\sum_{p_1,\ldots,p_k} (n!)^{k - 1} \frac{p_1 \cdots p_k}{n^k} M^n_{p_1,\ldots,p_k} 
%\prod_{i = 1}^k  \frac{\poly{p_i}{d}{x_id + 1}}{d^{p_i-1} p_i!} %\cdots \frac{\poly{p_k}{d}{x_kd + 1}}{d^{p_k-1} p_k!} 
%\\
& = d^{(n - 1)(k - 1)} 
\sum_{p_1,\ldots,p_k} C^{\langle n - 1, 1\rangle}_{p_1, \ldots, p_k}
\prod_{i = 1}^k \frac{\poly{p_i}{d}{x_id + 1}}{d^{p_i-1} p_i!} % \cdots \frac{\poly{p_k}{d}{x_kd + 1}}{d^{p_k-1} p_k!} 
\\
& = d^{(n - 1)(k - 1)} 
\sum_{p_1,\ldots,p_k} C^{\langle n - 1, 1\rangle}_{p_1, \ldots, p_k}
\prod_{i = 1}^k \left(d\binom{x_i}{p_i} + \binom{x_i}{p_i-1}\right) %\cdots \left(d\binom{x_k}{p_k} + \binom{x_k}{p_k-1}\right) 
\\
 & = d^{(n - 1)(k - 1)} 
 \sum_{p_1, \ldots, p_k} \left(\sum_{S \subseteq [k]} d^{k - |S|} C^{\langle n - 1, 1\rangle}_{\pp + \ee_S} \right) \binom{x_1}{p_1} \cdots \binom{x_k}{p_k},
\end{align*}
where in the first step we use Remark~\ref{rem:(n - 1)-cycle} and in
the last step ${\ee}_S$ is the indicator vector for the set
$S$, as usual.  Collecting the power of $d$ and applying
Lemma~\ref{(d, d, n) transitive lemma}, this becomes
\begin{equation}
\label{Thm 4.1 RHS}
\frac{|G|^{k - 1}}{n^k} 
\sum_{p_1,\ldots,p_k} M^n_{p_1,\ldots,p_k} 
\prod_i \frac{\poly{p_i}{d}{x_id + 1}}{d^{p_i-1} (p_i - 1)!} %\cdots \frac{\poly{p_k}{d}{x_kd + 1}}{d^{p_k-1} (p_k - 1)!}
=
\sum_{p_1,\ldots,p_k} C^{(d, d, n)}_{p_1,\ldots,p_k} \binom{x_1}{p_1}\cdots \binom{x_k}{p_k}.
\end{equation}

Now consider the LHS of the equation to be proved.  Again make the substitution $x_i \mapsto x_id + 1$, and think of $x_i$ as representing a nonnegative integer.  Then the LHS is the number of colored transitive factorizations of $c_{(d, d, n)}$.  The number of ways to choose $p_i$ $d$-strips that are actually used in the $i$th factor is $\binom{x_i}{p_i}$, and the number of colored factorizations in this case is exactly $C^{(d, d, n)}_{p_1,\ldots,p_k}$.  Thus
\[
\sum_{r_1,\ldots,r_k
  \geq 0} b^{(d)}_{r_1,\ldots,r_k} (x_1d+1)^{r_1}\cdots (x_kd+1)^{r_k} =
\sum_{p_1,\ldots,p_k} C^{(d, d, n)}_{p_1,\ldots,p_k} \binom{x_1}{p_1}\cdots \binom{x_k}{p_k},
\]
exactly the expression \eqref{Thm 4.1 RHS} we found for the RHS.  Since the
two sides are equal for nonnegative integers $x_i$, their equality is
valid as a polynomial identity, as well.  This completes the proof.
\end{proof}

\subsection{Nontransitive factorizations}

In the case of nontransitive factorizations of the Coxeter element $c_{(d, d, n)}$ in $G(d, d, n)$, we are again able to give a combinatorial formula for the number of colored factorizations, and thereby produce a generating function for factorization counts.  However, the lemma is more complicated than those in the preceding cases, and does not seem to yield a ``nice'' formula in a single basis, as in Theorems~\ref{G(d, 1, n) many factors theorem} and~\ref{thm:mainG(d,d,n) k factors}.

We consider cycle-colored factorizations of the kind defined in Definition~\ref{def:coloring}.  Given nonnegative integers $p_1, \ldots, p_k$, let $B^{(d, d, n)}_{p_1,\ldots,p_k}$ be the number of nontransitive factorizations $c_{(d, d, n)} = u_1 \cdots u_k$ in $G(d, d, n)$ such that, for $i = 1,\ldots, k$, the weight-$0$ cycles of $u_i$ are colored with colors from color set $\chi_i = \{0, 1, \ldots, p_id\}$ so that at least one color from each $d$-strip in $\chi_i$ is actually used to color a cycle.

\begin{lemma}
\label{(d, d, n) nontransitive lemma}
For integers $n, d \geq 2$, with $B^{(d, d, n)}_{p_1, \ldots, p_k}$ defined above, we have
\[
B^{(d, d, n)}_{p_1,\ldots,p_k} = d^{n(k - 1) + 1} \sum_{\substack{S, T, U \subseteq [k]: \\ S \cap T \neq \varnothing \\ S \cap U = \varnothing}}
d^{- |S \cup T|} \cdot \prod_{i \not \in S} p_i \cdot
C^{\langle n - 1 \rangle}_{\pp + \ee_T - \ee_U},
\]
where ${\ee}_S$ is the indicator vector for $S$ and $C^{\langle n - 1 \rangle}_{\pp} =
((n - 1)!)^{k - 1} \cdot M^{n - 2}_{\pp - {\bf 1}}$ is the count of colored factorizations of the $(n - 1)$-cycle in $\Symm_{n - 1}$ introduced just before Proposition~\ref{prop:change of basis}.
\end{lemma}

\begin{proof}
We use the same projection as in the previous two cases.  Consider a colored nontransitive factorization $u_1 \cdots u_k$ of the Coxeter element $c_{(d, d, n)}$ in $G(d, d, n)$, in which factor $u_i$ uses a color set of size $dp_i + 1$ and every $d$-strip is used at least once.  Since the factorization is nontransitive, $n$ is a fixed point in each $\pi_i$, i.e., $(n)$ is a cycle.  Let $S$ be the subset of $[k]$ recording factors in which the cycle $(n)$ has color $0$, and let $T$ be the subset of $[k]$ recording factors in which there is a cycle of color $0$ that is not the cycle $(n)$.  Then the image under projection is a colored nontransitive factorization of $\cnn$ in $\Symm_n$ using $\pp + \ee_{S \cup T}$ colors.

We have two calculations to make: computing the number of preimages of each of these $\Symm_n$-factorizations under projection, and counting the total number of $\Symm_n$-factorizations in the image.  We begin by some basic observations about the possibilities for $S$ and $T$.

First, $S$ must be nonempty: the standard basis vector $e_n$  is an eigenvector for all of $u_1, \ldots, u_k$ and $c_{(d, d, n)}$, and its eigenvalue for $c_{(d, d, n)}$ is not $1$; therefore its eigenvalue cannot be $1$ for all of the $u_i$.  Thus there must be at least one $i$ for which the singleton cycle $(n)$ has nonzero weight in $u_i$, and consequently has color $0$.  Thus for this $i$ we have $i \in S$.  In fact, we can say more: since $u_i \in G(d, d, n)$, the remaining cycles of $u_i$ have weights that sum to something nonzero, so at least one of them must have nonzero weight and hence color $0$, as well, and thus $i \in S \cap T$.

Now suppose that $S, T$ are subsets of $[k]$ with nonempty
intersection and $\pi_1 \cdots \pi_k$ is a colored factorization of
the form just described.  That is, it is a colored nontransitive
factorization of the $(n - 1)$-cycle $\cnn$ in which the cycles of
factor $\pi_i$ are colored with color set $\{1, \ldots, p_i\}$ if $i
\not \in S \cup T$ or with color set $\{0, 1, \ldots, p_i\}$ if $i \in
S \cup T$, every color being used at least once, and moreover in the latter case the cycle $(n)$ is colored with color $0$ if and only if $i \in S$, and some other cycle is colored with color $0$ if and only if $i \in T$. 

We proceed as in the previous arguments, aiming first to give a
factorization of any element $c^*$ that is a conjugate of $c_{(d, d,
  n)}$ by a diagonal matrix.  Choose a fixed index $i$ in $S \cap T$.
Order the $nk$ weights to be assigned as follows: first the weights for
every $u_j$ with $j \not \in S\cup T$; then the weights for every
$u_j$ with $j \in (S \cup T) \smallsetminus \{i\}$; then the weights
for $u_i$, arranged so that the last weight is in a cycle of color $0$
that is \emph{not} the fixed point.  (Such a choice is possible by the
definition of $S$ and $T$.)  In this order, every weight can be chosen
freely in $\ZZ/d\ZZ$ (giving $d$ choices), with the following exceptions:
\begin{enumerate}
\item if $j \not \in S \cup T$, every cycle is colored with a nonzero color, so has weight $0$; so the last element (in the total order) in each cycle is chosen with the unique weight that works;
\item if $j \in (S \cup T) \smallsetminus \{i\}$ and the weight being chosen is the last element (in the total order) in a cycle of nonzero color, it is assigned the unique value that gives its cycle weight $0$;
\item if $j \in (S \cup T) \smallsetminus \{i\}$ and the weight being chosen is the last element (in the total order) among those that appear in cycles of color $0$, it is assigned the unique value so that the weight of $u_j$ is $0$;
\item if $j = i$ and the weight being chosen is the weight of the fixed point, it is assigned the unique value so that the weight of the fixed point in the product $u_1 \cdots u_k$ is $-1$;
\item if $j = i$ and the weight being chosen is the last element (in the total order) in a cycle of nonzero color, it is assigned the unique value that gives its cycle weight $0$; and
\item if $j = i$ and the weight being chosen is the very last element
  in the total order, it is assigned the unique value that gives $u_j$ weight $0$.
\end{enumerate}
The number of values in (1), (2), and (5) together is exactly the number of cycles with nonzero color, while the number of values in (3), (4), and(6) together is exactly $|S \cup T| + 1$. 
Therefore, the total contribution from the weighting is 
$d^{nk - \# \textrm{nonzero colored cycles} - |S \cup T| - 1}$.  In addition, each cycle with nonzero color gets assigned one of $d$ colors from its strip, hence we must multiply by $d^{\# \textrm{nonzero colored cycles}}$.  Finally, the number of conjugates of $c_{(d, d, n)}$ by diagonal matrices is $d^{n - 2}$.  Thus, the number of preimages of a factorization with the given pair $(S, T)$ is exactly 
\[
d^{nk - \#(\textrm{nonzero colored cycles}) - |S \cup T|  - 1} \cdot d^{\# \textrm{nonzero colored cycles}} \cdot d^{2 - n}
= d^{n(k - 1) - |S \cup T| + 1}.
\]

Now we must compute the actual size of the image, that is, how many $\Symm_n$-factorizations are in the projection of $G(d, d, n)$-factorizations (with the sets $S, T$ fixed).  To do this, we further project factorizations of $\cnn$ in $\Symm_n$ to factorizations of the $(n - 1)$-cycle $\cn_{n - 1}$ in $\Symm_{n - 1}$, by removing the fixed point.  In this projection, there will be some subset $U \subseteq [k] \smallsetminus S$ of indices in which the (necessarily nonzero) color used to color the fixed point $(n)$ is the unique appearance of that color; thus, the image of this projection belongs to $\CCC^{\langle n - 1 \rangle}_{\pp + \ee_T - \ee_U}$.  Moreover, the number of preimages of each factorization is easily seen to be $\prod_{i \in [k] \smallsetminus S} p_i$, since we must choose which of the original nonzero colors was attached to the fixed point (whether that involves choosing an existing color or adding a new color and possibly relabeling).

Combining these two projections, we have that for fixed $S, T$, the number of nontransitive colored factorizations in $G(d, d, n)$ is
\[
d^{n(k - 1) + 1} \sum_{\substack{S, T, U \subseteq [k]: \\ S \cap T \neq \varnothing \\ S \cap U = \varnothing}}
d^{- |S \cup T|} \cdot \prod_{i \not \in S} p_i \cdot
C^{\langle n - 1 \rangle}_{\pp + \ee_T - \ee_U},
\]
as claimed.
\end{proof}

\begin{remark} \label{rem: algebraic non transitive}
As in Section~\ref{sec:G(d, 1, n)}, we have that Lemma~\ref{(d, d, n)
  nontransitive lemma} gives the coefficient for the polynomial
counting nontransitive factorizations of the Coxeter element when
expressed in the basis $\frac{(x - 1)^{(d)}_p}{d^p p!}$.  Obviously,
though, these coefficients are much messier (involving a triple sum
over subsets) than those of earlier sections.  In principle, it is
possible to express this polynomial in other bases; however, all
expressions we have produced in a single basis seem inherently
complicated.  Using the algebraic approach (with
Lemma~\ref{frobenius}, building on the arguments in
\cite[\S5]{ChapuyStump}), we have been able to produce the expression
\[
\frac{|G|^{k-1}}{n^{k-1}} \sum_{p_1,\ldots,p_k \geq 1}
 M^{n-2}_{p_1-1,\ldots,p_k-1} \left(  \sum_{S\subseteq [k]}
  \left((d-1)^{k-|S|} - (-1)^{k-|S|}\right) \prod_{i\in S} \frac{x_i
  \poly{p_i}{d}{x_i}}{d^{p_i} p_i!} \prod_{i\in \overline{S}}
  \frac{(x_i-1)_{p_i-1}^{(d)}}{d^{p_i} (p_i-1) !}  \right)
\]
for the generating function for all nontransitive factorizations. This
expression can be rewritten in many ways, for instance as
\begin{equation} \label{eq: algebraic non transitive}
\frac{|G|^{k-1}}{n^{k-1}} \sum_{p_1,\ldots,p_k \geq 1}
 M^{n-2}_{p_1-1,\ldots,p_k-1} 
 \left( \prod_{i = 1}^k Q_{p_i}(x_i)
-
 \prod_{i = 1}^k Q'_{p_i}(x_i) \right),
\end{equation}
where
$Q_p(x) := \left(x^2 - (d - 1)(p - 1)x + p(d - 1)\right) \frac{(x - 1)^{(d)}_{p - 1}}{d^{p} p!}$ and
$Q'_p(x) := \left(x^2 - (d - 1)(p - 1)x - p\right) \frac{(x - 1)^{(d)}_{p - 1}}{d^{p} p!}$. 
Frustratingly, however, we have not been able to derive this formula directly from Lemma~\ref{(d, d, n) nontransitive lemma}.
\end{remark}

\section{Exceptional complex reflection groups}
\label{sec:exceptional}

In this section, we record some tantalizing data that suggests that Theorems~\ref{G(d, 1, n) many factors theorem} and~\ref{thm:mainG(d,d,n) k factors} could be particular cases of a more general, uniform statement, along the lines of the Chapuy--Stump result (Theorem~\ref{thm:cs}).  

In all sections below, we use the following fixed notations: $G$ represents an irreducible well generated complex reflection group; $k$ represents a positive integer; $\rr = (r_1, \ldots, r_k)$ and $\pp = (p_1, \ldots, p_k)$ represent tuples in $\{0, 1, \ldots, n\}^k$, where $n$ is the rank of $G$; $a_{\rr}$ represents the number of factorizations of a fixed Coxeter element $c$ in $G$ as a product $u_1 \cdots u_k = c$ where $\dim \fix u_i = r_i$; and
\[
F_G(x_1, \ldots, x_k) := \sum_{\rr} a_{\rr} \cdot x_1^{r_1} \cdots x_k^{r_k}.
\]

\subsection{Rank two}
There are two infinite families of irreducible rank-$2$ well generated complex reflection groups, the wreath products $(\ZZ/d\ZZ) \wr \Symm_2$ (of type $G(d, 1, 2)$) and the dihedral groups (of type $G(d, d, 2)$), as well as twelve exceptional groups (Shephard--Todd classes $G_4$, $G_5$, $G_6$, $G_8$, $G_9$, $G_{10}$, $G_{14}$, $G_{16}$, $G_{17}$, $G_{18}$, $G_{20}$, and $G_{21}$).  For each such group, define the polynomials
\[
P_2(x) = \frac{(x - e^*_1)(x - e^*_2)}{|G|}, \qquad
P_1(x) = \frac{x - 1}{d_1}, \qquad
P_0(x) = 1.
\]
In this basis, we have the following result, which should be compared with Theorem~\ref{G(d, 1, n) many factors theorem}.
\begin{theorem}
\label{thm:rank 2}
Let $G$ be an irreducible well generated complex reflection group of rank $2$, and let $a_{\rr}$, $F_G$, and $P_i$ be defined as above.  Then one has
\[
F_G(x_1, \ldots, x_k) = |G|^{k- 1} \sum_{\pp % = (p_1, \ldots, p_k) \in \{0, 1, 2\}^k
} M^2_{\pp} \cdot P_{p_1}(x_1) \cdots P_{p_k}(x_k).
\]
In the case $k = 2$, this may be written in the form
\[
F_G(x, y) = (x - 1)(x - e^*_2) + N(x - 1)(y - 1) + (y - 1)(y - e^*_2) + 2h(x - 1) + 2h(y - 1) + |G|
\]
where $h = d_2$ is the Coxeter number of $G$ and $N = \frac{2h}{d_1}$ is the number of reflections that can appear in a shortest reflection factorization of $c$.
\end{theorem}
\begin{proof}
We handle the two infinite families and the exceptional cases
separately:

\subsubsection*{The infinite family $G(d, 1, 2)$} For this group this is exactly the statement of Theorem~\ref{G(d, 1, n) many factors theorem}.  

\subsubsection*{The infinite family $G(d,d,2)$} For $G(d, d, 2)$, the basis $\{P_2(x), P_1(x), P_0(x)\}$ here is related to the basis $\{\poly{2}{d}{x}, \poly{1}{d}{x}, \poly{0}{d}{x}\}$ of Theorem~\ref{thm:mainG(d,d,n) k factors} by the equations
$\frac{\poly{2}{d}{x}}{d^1 \cdot 1!} = %\frac{(x - 1)(x - d + 1)}{d} = 
2 \cdot P_2(x)$, $\poly{1}{d}{x} = %x = 
2P_1(x) + P_0(x)$, and $\poly{0}{d}{x} = P_0(x)$.  So from
Theorem~\ref{thm:mainG(d,d,n) k factors}, the contribution to $F_G$ of the transitive
factorizations is
\[
\frac{|G|^{k - 1}}{2^k} \sum_{\pp : p_i \geq 1} M^2_{\pp} \prod_{i : p_i = 2} 2 P_2(x_i) \prod_{i: p_i = 1} (2P_1(x_i) + P_0(x_i)).
\]
If $\rr$ is a permutation of the multiset $\{ 2^a, 1^b, 0^c\}$, then the coefficient of $\prod_i P_{r_i}(x_i)$ in this polynomial is $\frac{|G|^{k - 1}}{2^k} \cdot 2^a \cdot 2^b \cdot 1^c M^2_{2^a, 1^{b + c}} = \frac{|G|^{k - 1}}{2^c}M^2_{2^a, 1^{b + c}}$.  

For the contribution to $F_G$ from the nontransitive factorizations, we use the expression in \eqref{eq: algebraic non transitive}. The only contribution we
get is for the tuple ${\bf p}={\bf 1}$, where $M^0_{\bf 0} = 1$.  We obtain
\[
\frac{|G|^{k-1}}{2^{k-1}} \left(\prod_{i=1}^k Q_1(x_i) - \prod_{i=1}^k Q'_1(x_i)\right),
\]
where $Q_1(x) = \frac{x^2+d-1}{d} =
2P_2(x) + 2P_1(x)+1$ and $Q'_1(x) =  \frac{x^2-1}{d} = 2P_2(x) +
2P_1(x)$. If  $\rr$ is a permutation of the multiset $\{2^a,1^b, 0^c\}$, then the coefficient
of $\prod_i P_{r_i}(x_i)$ in this polynomial is $d^{k-1} 2^{a+b}= 2 \cdot \frac{|G|^{k - 1}}{2^c}$ if
$c>0$ and $0$ otherwise.

Finally, we sum these two contributions: when $c = 0$, the transitive case contributes exactly the desired $|G|^{k - 1} M^2_{\rr}$ while the nontransitive case contributes $0$.  From the explicit formula \eqref{definition M coeff} for $M^n_{\pp}$ we can see that when $c > 0$, $M^2_{2^a, 1^{b + c}} = 2^{b + c} - 2 = 2^c M^2_{\rr} - 2$, so in this case the transitive case contributes $|G|^{k - 1} M^2_{\rr} - 2 \frac{|G|^{k - 1}}{2^c}$ while the nontransitive case contributes the missing $2 \cdot \frac{|G|^{k - 1}}{2^c}$.

\subsubsection*{Exceptional rank $2$ cases} In each of the exceptional
cases, we use the algebraic technique outlined in
Section~\ref{sec:Frobenius} together with character tables for the
exceptional groups available in  Sage \cite{sage} via its interface with the GAP \cite{GAP} package Chevie, as follows: for each group $G$ and each irreducible character $\chi$ for $G$ such that $\chi(c^{-1}) \neq 0$, we compute the \emph{character polynomial}
\[
f_\chi(x) := \sum_{g \in G} \normchi(g) \cdot x^{\dim \fix(g)},
\]
where $\normchi = \chi/\chi(1)$ is the normalized character.  (Tables of these polynomials are collected in Appendix~\ref{app:tables of character polynomials}.) Then by Lemma~\ref{frobenius} we have that
\begin{align}
F_G(x_1, \ldots, x_k) &= \sum_{\rr} \frac{1}{|G|} \sum_{\lambda \in \Irr(G)} \dim(\lambda) \chi_\lambda(c^{-1}) \prod_{i = 1}^k \left(x_i^{r_i} \sum_{g \colon \dim \fix(g) = r_i} \normchi_\lambda(g) \right) \notag \\
& = \frac{1}{|G|}  \sum_{\lambda \in \Irr(G)} \dim(\lambda) \chi_\lambda(c^{-1}) \prod_{i = 1}^k f_{\chi_\lambda}(x_i).
\label{exceptional G_6}
\end{align}
Next we express each character polynomial in the $P_i$s and extract the coefficient.

For example, in the group $G_6$ (see Table~\ref{table:G_6}), with coexponents $1, 9$, degrees $4, 12$, and order $|G_6| = 4 \cdot 12 = 48$, there are $12$ characters $\chi$ for which $\chi(c^{-1}) \neq 0$, but only four different character polynomials:
\begin{center}
\begin{tabular}{lp{.3in}l}
$x^2 + 14x + 33$ && (associated only to the trivial representation), \\
$x^2 -10x + 9$ && (associated to two $1$-dimensional representations), \\
$x^2 - 4x + 3$ && (associated to two $2$-dimensional representations), and \\
$x^2 + 2x - 3$  && (associated to three $1$-dimensional and four $2$-dimensional representations).
\end{tabular}
\end{center}
In terms of the basis $P_2(x) = \frac{(x - 1)(x - 9)}{4 \cdot 12}$, $P_1(x) = \frac{x - 1}{4}$, $P_0(x) = 1$, these can be rewritten respectively as
\[
48(P_2 + 2P_1 + P_0), \qquad 48 P_2, \qquad 24(2P_2 + P_1), \qquad \textrm{ and } \qquad 48(P_2 + P_1).
\]
The respective character values $\chi(c^{-1})$ that appear with these
polynomials are $\{1\}$; $\{\exp(2\pi i / 6), \exp(5\pi i/3)\}$; $\{
\pm i\}$; and $\{-1, \exp(2\pi i/3), \exp(2 \pi i \cdot 2/3)\}$ (for
the $1$-dimensional representations) and $\{\exp(2
\pi i / 12), \exp(2 \pi i \cdot 5/12), \exp(2 \pi i \cdot 7/12),
\exp(2 \pi i \cdot 11/12)\}$ (for the $2$-dimensional representations).  Plugging these in to \eqref{exceptional G_6} yields
\begin{align}
\notag
F_{G_6}(x_1, \ldots, x_k) & = \frac{1}{|G_6|} \left(
1 \cdot 1 \cdot \prod_{i = 1}^k 48(P_2(x_i) + 2P_1(x_i) + P_0(x_i)) + 
1 \cdot 1 \cdot \prod_{i = 1}^k 48P_2(x_i) + {}\right. \\\notag
& \hspace{1in} \left. 2 \cdot 0 \cdot \prod_{i = 1}^k 24(2P_2(x_i) + P_1(x_i)) +
(1 \cdot (-2) + 2 \cdot 0) \cdot \prod_{i = 1}^k 48(P_2(x_i) + P_1(x_i))
\right) \\\label{last step}
& = |G_6|^{k - 1} \left(\prod_{i = 1}^k (P_2(x_i) + 2P_1(x_i) + P_0(x_i)) - 2\prod_{i = 1}^k (P_2(x_i) + P_1(x_i)) + \prod_{i = 1}^k P_2(x_i)\right).
\end{align}
Finally, the coefficient of $\prod_{i = 1}^k P_{r_i}(x_i)$ in the expansion of this expression is (up to the power of $|G_6|$) the same as the coefficient of $\prod_{i = 1}^k x_i^{r_i}$ in the polynomial 
\[
\prod_{i = 1}^k (x_i^2 + 2x_i + 1) - 2\prod_{i = 1}^k (x_i^2 + x_i) + \prod_{i = 1}^k x_i^2 = \left( \prod_{i = 1}^k (x_i + 1) - \prod_{x = 1}^k x_i \right)^2,
\]
and comparing the last expression with the definition of $M^n_{\pp}$ in \eqref{definition M coeff} gives the result.  

The calculations in the other eleven cases are analogous.
\end{proof}

\subsection{Rank three}
\label{sec:rank 3}

There are two infinite families of well generated irreducible complex
reflection groups of rank $3$ (the wreath product $(\ZZ/d\ZZ) \wr
\Symm_3$ of type $G(d, 1, 3)$ and its weight-zero subgroup $G(d, d,
3)$) and five exceptional groups (Shephard--Todd classes $G_{23}$
(also the Coxeter group of type $H_3$), $G_{24}$, $G_{25}$, $G_{26}$,
and $G_{27}$).  Our next result covers all of these groups
\emph{except} $G_{25}$. The choice of polynomial basis assigned to
each group is discussed further in Remark~\ref{Orlik--Terao remark}.

\begin{definition} \label{funny polynomials}
For $G$ of type $G(d, 1, 3)$ or $G_{23}$ or $G_{26}$, define for $i=0,1,2,3$ the polynomials
\[
P^G_i(x) = \prod_{j = 1}^i \frac{x - e^*_j}{d_j}.
\]
For $G$ of type $G(d, d, 3)$, define
\begin{multline*}
P^{G(d, d, 3)}_0(x) = 1, \quad P^{G(d, d, 3)}_1(x) = \frac{(d + 1)(x
  - 1)}{3d}  \left(= \frac{d + 1}{d} \cdot \frac{x - 1}{d_1} \text{ for }
d\geq 3\right), \\ P^{G(d, d, 3)}_2(x) = \frac{(x - 1)(x - d)}{3d}, \qquad \textrm{and} \qquad P^{G(d, d, 3)}_3(x) = \frac{(x - 1)(x - d - 1)(x - 2d + 2)}{6d^2} = \prod_{i = 1}^3 \frac{x - e^*_i}{d_i}.
\end{multline*}
For $G$ of type $G_{24}$, define
\begin{multline*}
P^{G_{24}}_0(x) = 1, \qquad P^{G_{24}}_1(x) = \frac{x - 1}{3} = \frac{4}{3} \cdot \frac{x - 1}{d_1}, \qquad P^{G_{24}}_2(x) = \frac{(x - 1)(x - 7)}{24} = \frac{(x - 1)(x - 7)}{d_1d_2}, \\ \textrm{and} \qquad P^{G_{24}}_3(x) = \frac{(x - 1)(x - 9)(x - 11)}{336} = \prod_{i = 1}^3 \frac{x - e^*_i}{d_i}.
\end{multline*}
For $G$ of type $G_{27}$, define
\begin{multline*}
P^{G_{27}}_0(x) = 1, \qquad P^{G_{27}}_1(x) = \frac{2(x - 1)}{9} = \frac{4}{3} \cdot \frac{x - 1}{d_1}, \qquad P^{G_{27}}_2(x) = \frac{(x - 1)(x - 15)}{72} = \frac{(x - 1)(x - 15)}{d_1d_2}, \\ \textrm{and} \qquad P^{G_{27}}_3(x) = \frac{(x - 1)(x - 19)(x - 25)}{2160} = \prod_{i = 1}^3 \frac{x - e^*_i}{d_i}.
\end{multline*}
\end{definition}

\begin{proposition}
\label{rank 3 best}
If $G$ is of type $G(d, 1, 3)$, $G(d,d,3)$, $G_{23}$, $G_{24}$ , $G_{26}$ or $G_{27}$, then
\[
F_G(x_1, \ldots, x_k) = |G|^{k - 1} \sum_{\pp} M^3_{\pp} \cdot P^G_{p_1}(x_1) \cdots P^G_{p_k}(x_k).
\]
\end{proposition}

\begin{proof}

We again proceed case-by-case.

\subsubsection*{The infinite family $G(d, 1, 3)$} 
For this group this is exactly the statement of Theorem~\ref{G(d, 1, n) many factors theorem}. 

\subsubsection*{The infinite family $G(d, d, 3)$}

Write $P_i$ for the polynomial $P^{G(d, d, 3)}_i$ of Definition~\ref{funny polynomials}.  We begin with the contribution from the transitive factorizations.  The formulas $\frac{\poly{3}{d}{x}}{2d^2} = 3 P_3(x)$, $\frac{\poly{2}{d}{x}}{d} = 3P_2(x) + \frac{3}{d + 1} P_1(x)$, and $\poly{1}{d}{x} = \frac{3d}{d + 1}P_1(x) + P_0(x)$ express the basis of Theorem~\ref{thm:mainG(d,d,n) k factors} in terms of this basis. 
Thus, for a sequence $\rr$ that is a rearrangement of $\{3^a, 2^b, 1^c, 0^e\}$ with $a + b + c + e = k$, the coefficient of $\prod_i P_{r_i}(x_i)$ in
\[
\frac{|G|^{k-1}}{3^k}\sum_{p_1,\ldots,p_k \geq 1}
 M^3_{p_1,\ldots,p_k} \frac{\poly{p_1}{d}{x_1}}{d^{p_1-1} (p_1-1)!} \cdots 
 \frac{\poly{p_k}{d}{x_k}}{d^{p_k-1} (p_k-1)!}
\]
is
\begin{align}
T(\rr) & := \frac{|G|^{k-1}}{3^k} \cdot 3^a \cdot 3^b \cdot \sum_{S \subseteq [c]} M^3_{(1^c, 3^a, 2^b, 1^e) + \ee_S} \left(\frac{3}{d + 1}\right)^{|S|} \left(\frac{3d}{d + 1}\right)^{c - |S|} \notag \\
&= \frac{|G|^{k-1}}{3^k} \cdot \frac{3^{a + b + c}}{(d + 1)^c}
\sum_{S \subset [c]} M^3_{(1^c, 3^a, 2^b, 1^e) + \ee_S} d^{c - |S|}. \label{eq: contribution G(d,d,3) transitive}
\end{align}
Next, we simplify using the formula 
\[
M^3_{3^a,2^b,1^c,0^e} = \begin{cases}
3^{b+c} &\text{ if } e>0\\
3^{b+c}-3\cdot 2^b & \text{ if } e=0, c>0\\
3^{b}-3\cdot 2^b+3 & \text { if } c = e=0, b>0\\
0 & \text{ if } b = c = e = 0
\end{cases} 
\]
for $M^3_{\pp}$ derived from the explicit definition \eqref{definition M coeff} to obtain 
\begin{equation}
\label{eq:T}
T(\rr) = |G|^{k-1} \cdot 
\begin{cases}
  3^{b+c} - 3^{1-e} \cdot 2^b \cdot \left(\frac{d+2}{d+1}\right)^c & \text{ if } e >0 \\
  3^{b + c} - 3 \cdot 2^b \cdot \left(\frac{d + 2}{d + 1}\right)^c + \frac{3}{( d+ 1)^c}& \text{ if } e = 0, c>0\\
 3^{b} -3\cdot 2^b +3& \text{ if } c = e = 0, b>0\\
 0 & \text{ if } b = c = e = 0.
\end{cases}
\end{equation}

Now we calculate the nontransitive contribution using the expression in \eqref{eq: algebraic non transitive}. From this expression we get a
contribution for the tuples $\pp \in \{1, 2\}^k$ other than ${\bf 2} = (2, \ldots, 2)$, and for each such tuple we have $M^1_{\pp - {\bf 1}}=1$. Thus we obtain the generating function
\begin{equation} \label{eq:contribNTGdd3}
\frac{|G|^{k-1}}{3^{k-1}} \sum_{\pp \neq {\bf 2}, p_i \in \{1,2\}} \left(\prod_{i=1}^k Q_{p_i}(x_i) - \prod_{i=1}^k Q'_{p_i}(x_i)\right).
\end{equation}
We have the change of basis formulas
\begin{align*}
Q_1(x)  &= 3P_2(x) + 3P_1(x)+1, &
Q'_1(x) & =  \frac{x^2-1}{d} = 3P_2(x) + 3P_1(x), \\
Q_2(x) &= ({x}^{2}-dx+x+2\,d-2)(x-1)/(2d^2) & 
Q'_2(x) &= ({x}^{2}-dx+x-2)(x-1)/(2d^2) \\
& =3P_3(x)+3P_2(x)+\frac{3}{d+1}P_1(x), &
& =3P_3(x)+3P_2(x).
\end{align*}
For a sequence $\rr$ that is a rearrangement of $\{3^a, 2^b, 1^c,
0^e\}$ with $a + b + c + e = k$, let $N(\rr)$ denote the coefficient
of $\prod_i P_{r_i}(x_i)$ in \eqref{eq:contribNTGdd3}. Using the
change of basis formulas above, after some calculations we obtain the
following formula for $N({\bf r})$:
\begin{equation*}
N(\rr)  = |G|^{k-1} \cdot \begin{cases}
  2^b \cdot 3^{1-e} \cdot \left(\frac{d + 2}{d + 1}\right)^c   & \text{ if } e >0,\\
 3\cdot 2^b \left(\frac{d+2}{d+1}\right)^c - 3\cdot 2^b + \frac{3}{(d + 1)^c} & \text{ if } e =0, c>0,\\
 0 & \text { otherwise}.
\end{cases}
\end{equation*}
Combining this with \eqref{eq:T}, it is easy to see that $T(\rr)+N(\rr)$ equals $|G|^{k-1} M^3_{\rr}$ in each of the cases.

\subsubsection*{Exceptional rank $3$ cases}
In the four exceptional cases, we use exactly the same approach as in the proof of Theorem~\ref{thm:rank 2}.  In place of \eqref{last step}, one ends up in all four cases with the expression 
\begin{multline*}
|G|^{k- 1} \left( \prod_i (P_3(x_i) + 3P_2(x_i) + 3P_1(x_i) + 1) - 3\prod_i (P_3(x_i) + 2P_2(x_i) + P_1(x_i)) + {} \right. \\ \left. {} + 3\prod_i (P_3(x_i) + P_2(x_i)) - \prod_i P_3(x_i) \right),
\end{multline*}
from which the result follows.
\end{proof}

\begin{remark}
\label{Orlik--Terao remark}
The choice of basis polynomials $P_i$, especially for $G_{24}$ and $G_{27}$, appears somewhat mysterious, and indeed they were initially discovered experimentally.  However, the roots that appear are not arbitrary.  Given a complex reflection group $G$ acting on a space $V$, consider the intersection lattice $L$ of the reflecting hyperplanes of $G$, ordered by reverse-inclusion.  Let $\chi(L, x) := \sum_{X \in L} \mu(V, X) x^{\dim X}$ be the characteristic polynomial of $L$ (where $\mu$ is the poset-theoretic M\"obius function).  In \cite{OrlikSolomon1982, OrlikSolomon1983}, Orlik and Solomon show that $\chi(L, x)$ factors as $\prod_i (x - e^*_i)$.  But, in fact, they show more: if $X \in L$ is any particular intersection of reflecting hyperplanes, then the upper interval $L^X := [X, \{0\}]$ in $L$ (which is isomorphic to the lattice of the arrangement of intersections of the reflecting hyperplanes with $X$) satisfies $\chi(L^X, x) = \prod_{i = 1}^{\dim X} (x - b^X_i)$ for some positive integers $b^X_i$.  (Tables of these \emph{Orlik--Solomon coexponents} are collected in \cite[App.~C]{OrlikTerao}.)  For the group $G(d, 1, n)$, when $X$ is a subspace of dimension $m$, Orlik and Solomon show that $(b^X_i)_{1 \leq i \leq m} = (1, 1 + d, \ldots, 1 + (m - 1)d)$ depends only on the dimension of $X$.  Similarly, in the other irreducible groups of rank $2$ and $3$ (both exceptional and $G(d, d, n)$), the tuple of integers $(b^X_i)$ depends only on $\dim X$.  In fact, in all of these groups except $G(d, d, 3)$, $G_{24}$, and $G_{27}$, these Orlik--Solomon coexponents are a prefix of the sequence $(e^*_i)$ of coexponents.  In $G_{24}$, the coexponents are $(1, 9, 11)$ and the Orlik--Solomon coexponents in dimension $2$ are $(1, 7)$; in $G_{27}$, the coexponents are $(1, 19, 25)$ and the Orlik--Solomon coexponents in dimension $2$ are $(1, 15)$; and in $G(d, d, n)$, the coexponents are $(1, d + 1, 2d - 2)$ and the Orlik--Solomon coexponents in dimension $2$ are $(1, d)$.  One immediately notes the relationship to the polynomials in Definition~\ref{funny polynomials}.  (We do not have an explanation (or even an ``explanation'') for the factors $\frac{4}{3}$ and $\frac{d + 1}{d}$ that appear there.)
\end{remark}

\begin{remark}
For the group $G_{25}$, one can show that there are \emph{no} polynomials $P_i$
of degree $i$ for $i=0,1,2,3$ such that 
\[
F_{G_{25}}(x_1, \ldots, x_k) = |G|^{k - 1} \sum_{\pp} M^3_{\pp} \cdot P_{p_1}(x_1) \cdots P_{p_k}(x_k).
\]
Nevertheless, one can use the character tables to compute a formula for $F_{G_{25}}$: the group has coexponents $(e^*_1, e^*_2, e^*_3) = (1, 4, 7)$ and degrees $(d_1, d_2, d_3) = (6, 9, 12)$.  Taking $P_i(x) = \prod_{j = 1}^i \frac{x - e^*_i}{d_i}$, one has
\begin{multline*}
\frac{F_{25}(x_1, \cdots, x_k)}{|G_{25}|^{k - 1}} = \prod_{i = 1}^k(P_3(x_i) + 3P_2(x_i) + 3P_1(x_i) + 1) 
- 3 \cdot \prod_{i = 1}^k(P_3(x_i) + 2P_2(x_i) + P_1(x_i) )  \\ 
{} - 3 \cdot \prod_{i = 1}^k(P_3(x_i) + P_2(x_i)  )
- \prod_{i = 1}^k P_3(x_i)
- 3 \cdot \prod_{i = 1}^k(P_3(x_i) + P_2(x_i) + P_1(x_i) /3)    
+ 9 \cdot \prod_{i = 1}^k(P_3(x_i) + P_2(x_i) + P_1(x_i) /9),
\end{multline*}
and so the coefficients that appear in this expansion in place of the $M^3_\pp$ are the same as the coefficients in the expansion of the polynomial
\[
\prod_{i = 1}^k(x_i + 1)^3 - 3 \cdot \prod_{i = 1}^k x_i (x_i + 1)^2  - 3 \cdot \prod_{i = 1}^k x_i^2(x_i + 1  )  - \prod_{i = 1}^k x_i^3 - 3 \cdot \prod_{i = 1}^k(x_i^3 + x_i^2 + x_i /3)   + 9 \cdot \prod_{i = 1}^k(x_i^3 + x_i^2 + x_i /9) .
\]
\end{remark}

\subsection{Higher ranks}
\label{sec:exceptional higher ranks}

There are ten exceptional complex reflection groups of rank $4$ or
larger, of which all but one (the rank-$4$ group $G_{31}$) are well
generated.  Of these, only $G_{32}$ has the property that the sequence
of Orlik--Solomon coexponents $(b^X_i)_i$ is a function of the dimension of $X$ alone \cite[App.~C]{OrlikTerao}.  For this group, one sees from Table~\ref{table:G_32} that the terms that one would predict (based on Theorems~\ref{G(d, 1, n) many factors theorem} and~\ref{thm:rank 2} and Proposition~\ref{rank 3 best}) survive, but that other terms are present as well (as for $G_{25}$ above).  Moreover, for the eight other relevant groups, one sees in each case that there is not enough cancellation to give a result in the form of those just mentioned.  Section~\ref{sec:uniform} offers some open questions in this direction.

\section{Applications}
\label{sec:applications}

In this section, we discuss a number of corollaries of our results.
First, we show how to derive the result of Chapuy and Stump
(Theorem~\ref{thm:cs}) for the groups $G(d, 1, n)$ and $G(d, d, n)$
from Theorems~\ref{G(d, 1, n) many factors theorem}
and~\ref{thm:mainG(d,d,n) k factors}.  In the case of $G(d, 1, n)$,
this makes the proof fully elementary (that is, with no use of
representation theory); in the case of $G(d, d, n)$, it reduces this
question to an elementary proof of a result in the symmetric group: 
Theorem~\ref{S_n (n-1) cycle theorem k factors} (see Question~\ref{n - 1 cycle question}).  

Second, we use our results to extract highest-degree coefficients.
These count the so-called \emph{genus $0$ factorizations}, or
equivalently they count chains in the \emph{lattice of $G$-noncrossing
  partitions}.  For the group $G(d, d, n)$, this includes a new
result, extending work of Athanasiadis--Reiner and Bessis--Corran
\cite{AthanasiadisReiner, BessisCorran}.  As a corollary, we give a
variant characterization of the $G$-noncrossing partition lattice that
seems not to have appeared in the literature.

\subsection{Rederiving the Chapuy--Stump formula}

Suppose that, for a well generated complex reflection group $G$, one
has computed (in some form) the generating functions $F_G(x_1, \ldots,
x_\ell) = \sum_{\rr} a_{\rr} \cdot x_1^{r_1} \cdots x_\ell^{r_\ell}$
for the number of factorizations of a Coxeter element in $G$ by
fixed space dimension $r_i$ for each of the factors.  From this series, it is straightforward to extract the number of length-$\ell$ reflection factorizations of a Coxeter element in $G$: it is exactly the coefficient $[x_1^{n - 1} \cdots x_\ell^{n - 1}]F_G$.  In particular, suppose that the expression for $F_G$ is in terms of a basis $P_0, \ldots, P_n$ of polynomials with $\deg(P_i) = i$, so that
\[
F_G(x_1, \ldots, x_\ell) = \sum_{\pp} A_{\pp} \cdot P_{p_1}(x_1) \cdots P_{p_\ell}(x_\ell),
\]
and suppose that the degree-$(n - 1)$ coefficients of $P_n$ and $P_{n - 1}$ are $a$ and $b$, respectively.  From general considerations (either the algebraic formula, or, that the \emph{Hurwitz action} of the braid group provides explicit bijections), one has that the coefficient $A_{\pp}$ depends only on the multiset of values of the $p_i$, not on their order.
Therefore, the number $N_{\ell}$ of length-$\ell$ reflection factorizations of $c$ is
\[
N_\ell = [x_1^{n - 1} \cdots x_\ell^{n - 1}]F_G = \sum_{k = 0}^\ell \binom{\ell}{k} A_{\underbrace{\scriptstyle n - 1, \ldots, n - 1}_{\textstyle k}, \, \underbrace{\scriptstyle n, \ldots, n}_{\textstyle \ell - k}} a^{\ell - k} b^{k}.
\]
From this general framework, we now derive the main theorem of Chapuy
and Stump in many cases.
\begin{remark}
Suppose that $G$ is $G(d, 1, n)$ or is one of the irreducible well generated groups of rank $2$ or $3$, other than $G_{25}$; then by Theorem~\ref{G(d, 1, n) many factors theorem}, Theorem~\ref{thm:rank 2}, and Proposition~\ref{rank 3 best} we have that 
\[
A_\pp  = |G|^{\ell - 1} M^n_\pp, \qquad
a = -\frac{1}{|G|}\sum e^*_i = -\frac{|R^*|}{|G|}, \qquad
\textrm{ and }\qquad
 b  = \frac{1}{d_1 \cdots d_{n - 1}} = \frac{h}{|G|}.
 \]
Moreover, from \eqref{definition M coeff} we have $M^n_{n, \ldots, n} = 0$ and $M^n_{n - 1, \ldots, n - 1, n, \ldots, n} %= k! \cdot S(n, k) 
= \sum_{j = 0}^{n - 1} (-1)^j \binom{n}{j} (n - j)^k$,
and so for these groups
\begin{align*}
\sum_{\ell = 0}^\infty N_\ell \frac{t^\ell}{\ell!}  
%= \sum_{\ell = 0}^\infty  \frac{t^\ell}{\ell!} \sum_{k = 0}^\ell \binom{\ell}{k} |G|^{\ell - 1} M^n_{\underbrace{\scriptstyle n - 1, \ldots, n - 1}_{\textstyle k}, \, \underbrace{\scriptstyle n, \ldots, n}_{\textstyle \ell - k}} a^{\ell - k} b^{k} \\
%& = \sum_{\ell = 0}^\infty \frac{t^\ell}{\ell!} \sum_{k = 1}^\ell \binom{\ell}{k} |G|^{\ell - 1} \sum_{j = 0}^{n - 1} (-1)^j \binom{n}{j} (n - j)^k \left(-\frac{|R^*|}{|G|}\right)^{\ell - k} \left(\frac{h}{|G|}\right)^{k} \\
%& = \frac{1}{|G|} \sum_{\ell = 0}^\infty \frac{t^\ell}{\ell!} \sum_{k = 1}^\ell \sum_{j = 0}^{n - 1}  \binom{\ell}{k} (-1)^j \binom{n}{j} (n - j)^k \left(-|R^*|\right)^{\ell - k} h^{k} \\
& = \frac{1}{|G|} \sum_{j = 0}^{n - 1} (-1)^j \binom{n}{j} \sum_{\ell = 0}^\infty \frac{t^\ell}{\ell!} \sum_{k = 1}^\ell   \binom{\ell}{k}  (n - j)^k \left(-|R^*|\right)^{\ell - k} h^{k} \\
& = \frac{1}{|G|} \sum_{j = 0}^{n - 1} (-1)^j \binom{n}{j} \sum_{\ell = 0}^\infty \frac{t^\ell}{\ell!} \left( ((n - j)h - |R^*|)^\ell - \left(-|R^*|\right)^{\ell}\right) \\
& = \frac{1}{|G|} \sum_{j = 0}^{n - 1} (-1)^j \binom{n}{j} \left( \exp\left(((n - j)h - |R^*|)t\right) - \exp \left(-|R^*|t\right)\right) \\
%& = \frac{1}{|G|} \left( \exp((nh - |R^*|)t) \sum_{j = 0}^{n - 1} (-1)^j \binom{n}{j} \exp(-ht)^j  -  \sum_{j = 0}^{n - 1} (-1)^j \binom{n}{j}\exp \left(-|R^*|t\right)\right) \\
%& =  \frac{1}{|G|}\left( \exp((nh - |R^*|)t) \left(\left( 1 - \exp(-ht) \right)^n - (-1)^n \exp(-nht)\right) + (-1)^n \exp \left(-|R^*|t\right)\right) \\
%& = \frac{1}{|G|}\left( \exp((nh - |R^*|)t)\left( 1 - \exp(-ht) \right)^n\right) \\
& = \frac{1}{|G|} \left( \exp((nh - |R^*|)t/n) - \exp(- |R^*|t/n) \right)^n.
\end{align*}
Finally, using the general fact that $ |R^*| +  |R| = nh$, we recover
Theorem~\ref{thm:cs} for these groups.  In particular, in the case of
$G(d, 1, n)$, this gives a fully elementary proof of the
theorem.
\end{remark}
\begin{remark}
In the case of $G(d, d, n)$ for $n > 3$, the calculation in the preceding remark does not apply.  Nevertheless, since every reflection factorization of a Coxeter element in $G(d, d, n)$ is transitive, it is possible to recover the Chapuy--Stump formula for this group from Theorem~\ref{thm:mainG(d,d,n) k factors}.  Following the notation above, we have for this polynomial that $A_{\pp} = \frac{|G|^{k - 1}}{n^k} M^n_{\pp}$, $a = -\frac{1}{d^{n - 1}(n - 1)!}\sum e^*_i =  - \frac{n|R^*|}{|G|}$, and $b = \frac{1}{d^{n- 2}(n - 2)!} = \frac{n(n - 1)d}{|G|} = \frac{nh}{|G|}$, and so the calculation goes through in the same way after the cancellation of the $n$s.  The resulting argument is elementary except for the proof of Theorem~\ref{S_n (n-1) cycle theorem k factors} (see Question~\ref{n - 1 cycle question}).
\end{remark}

\subsection{Genus $0$ factorizations: chains in the poset of noncrossing partitions}

For any two linear transformations $U, T$ on a finite-dimensional vector space $V$, one has $\codim \fix(UT) \leq \codim \fix(U) + \codim\fix(T)$.  If $G$ is a well generated complex reflection group of rank $n$ and $c$ is a Coxeter element in $G$, it follows that in order for there to be a factorization $c = u_1 \cdots u_k$ of $c$ where $u_i$ has fixed space dimension $r_k$, one must have $(n - r_1) + \ldots + (n - r_k) \geq n$, or equivalently $r_1 + \ldots + r_k \leq n(k - 1)$.  Consequently, the polynomial $F_G(x_1, \ldots, x_k)$ counting factorizations of $c$ is of degree $n(k - 1)$, with top-degree coefficients counting factorizations $u_1 \cdots u_k = c$ such that $\codim \fix(u_1) + \ldots + \codim \fix(u_k) = \codim \fix(c) (= n)$.  Such factorizations are often said to have \emph{genus $0$} -- see Section~\ref{sec:maps} for a discussion of this topological terminology.  They also arise in the context of the \emph{$G$-noncrossing partition lattice}, which is the subject of the rest of this section.

A subadditive function on a group $G$ such as\footnote{In particular,
  one needs that the function takes nonnegative values and that $f(x)
  = 0$ if and only if $x$ is the identity.} $\codim \fix(-)$ gives
rise to a partial order $\leq$ on $G$, as follows: one defines $x \leq
y$ if $\codim\fix(x) + \codim \fix(x^{-1}y) = \codim\fix(y)$.  With
this definition, one has that for each fixed $g \in G$, genus-$0$
factorizations $u_1 \cdots u_k$ of $g$ are in bijection with (multi)chains $1 \leq g_1 \leq \ldots \leq g_k = g$ in the interval $[1, g]_{\leq}$ via the map $g_i := u_1 \cdots u_i$.  Thus, the top-degree coefficient $a_{r_1, \ldots, r_k}$ in $F_G(x_1, \ldots, x_k)$ also counts multichains in the interval $[1, c]_{\leq}$ whose successive rank-jumps are $n - r_1, \ldots, n - r_k$.  We now compute these numbers for $G(d, 1, n)$ and $G(d, d, n)$, as well as the \emph{zeta polynomial} $Z([1, c]_{\leq}, k)$ that counts all multichains in $[1, c]_{\leq k}$ of length $k$.  The result for $G(d, 1, n)$ is known.

\begin{cor}[{\cite[Prop.~7 and remark on p.~199]{Reiner}}] \label{corollary: genus 0 G(d,1,n)}
For $d > 1$, let $G = G(d, 1, n)$, and let $c$ be a Coxeter element in $G$.  Given nonnegative integers $s_1, \ldots, s_k$ with sum $n$, the number of chains in $[1, c]_{\leq}$ having rank-jumps $s_1, \ldots, s_k$ is 
\[
a^{(d)}_{n - s_1, \ldots, n - s_k} = \binom{n}{s_1} \cdots \binom{n}{s_k},
\]
with zeta polynomial 
\[
Z([1, c]_{\leq}, k) = \binom{nk}{n}.
\]
\end{cor}
\begin{proof}
Set $r_i := n - s_i$ for $i = 1, \ldots, k$.  
Since $M^n_{p_1, \ldots, p_k} = 0$ if $p_1 + \ldots + p_k > n(k - 1)$ and the change of basis in Theorem~\ref{G(d, 1, n) many factors theorem} is triangular, the equation $a^{(d)}_{\rr} = [x_1^{r_1}\cdots x_k^{r_k}] F_G =  [x_1^{r_1}\cdots x_k^{r_k}] \left(|G|^{k - 1} \sum_{\pp} M^n_{\pp} \prod_i \frac{(x - 1)^{(d)}_{r_i}}{d^{r_i} r_i!}\right)$ simplifies to 
\[
 a^{(d)}_{\rr} = |G|^{k - 1} M^n_{\rr} \prod_{i = 1}^k \frac{1}{d^{r_i} r_i!}.
 \]
 By Proposition~\ref{prop:Ms2multinomial}, we have in this case that
 $M^n_{\rr} = \frac{n!}{(n - r_1)! \cdots (n - r_k)!}$, and
 consequently $a^{(d)}_{\rr} = \binom{n}{r_1} \cdots \binom{n}{r_k} =
 \binom{n}{s_1} \cdots \binom{n}{s_k}$, as claimed.  Summing $a^{(d)}_{\rr}$ over all nonnegative integer solutions to $s_1 + \ldots + s_k = n$ gives the zeta polynomial.
\end{proof}

The second result, for $G(d, d, n)$, is new when $d > 2$ (but see Remark~\ref{nc remark}); when $d = 2$ (type $D_n$), it is \cite[Thm.~1.2(ii)]{AthanasiadisReiner}.  

\begin{cor}
\label{ddn genus 0}
For $d > 1$, let $G = G(d, d, n)$, and let $c_{(d, d, n)}$ be a Coxeter element in $G$.  Given nonnegative integers $s_1, \ldots, s_k$ with sum $n$, the number of chains in $[1, c_{(d, d, n)}]_{\leq}$ having rank-jumps $s_1, \ldots, s_k$ is 
\[
\binom{n - 1}{s_1} \cdots \binom{n - 1}{s_k}\left(d + \sum_{i = 1}^k \frac{\binom{n - 2}{s_i - 2}}{\binom{n - 1}{s_i}}\right),
\]
with zeta polynomial
\[
Z([1, c_{(d, d, n)}]_{\leq}, k) = d\binom{(n - 1)k}{n} + k \binom{(n - 1)k - 1}{n - 2} = \frac{n + d(n - 1)(k - 1)}{n} \binom{(n - 1)k}{n - 1}. 
\]
\end{cor}
\begin{proof}
Let $r_i = n - s_i$.  Since $M^n_{p_1, \ldots, p_k} = 0$ if $p_1 +
\dots + p_k > n(k - 1)$ and the change of basis is triangular, when we
extract the coefficient of $x_1^{r_1} \cdots x_k^{r_k}$ from
Theorem~\ref{thm:mainG(d,d,n) k factors} and
  Proposition~\ref{prop:Ms2multinomial} we get
\begin{align*}
b^{(d)}_{r_1,\ldots,r_k} & = \frac{|G|^{k-1}}{n^k} M^n_{r_1,\ldots,r_k} \frac{1}{d^{r_1-1} (r_1-1)!} \cdots  \frac{1}{d^{r_k-1} (r_k-1)!} \\
& = \frac{d^{(n - 1)(k - 1)} (n!)^{k - 1}}{n^k} \cdot \frac{n!}{(n - r_1)! \cdots (n - r_k)!} \cdot \frac{1}{d^{n(k - 1) - k} \cdot (r_1 - 1)! \cdots (r_k - 1)!} \\
& = d \cdot \binom{n - 1}{r_1 - 1} \cdots \binom{n - 1}{r_k - 1}.
 \end{align*}

The analysis of the nontransitive case is slightly more delicate: by Lemma~\ref{(d, d, n) nontransitive lemma}, we have that the desired coefficient is
\begin{multline*}
d^{n(k - 1) + 1} \sum_{\substack{S, T, U \subseteq [k]: \\ S \cap T \neq \varnothing \\ S \cap U = \varnothing}} d^{- |S \cup T|} \cdot \prod_{i \not \in S} r_i \cdot 
C^{\langle n - 1 \rangle}_{\rr + \ee_T - \ee_U}
\cdot \frac{1}{d^{r_1}r_1! \cdots d^{r_k}r_k!}
= {}
\\
\frac{d}{ r_1! \cdots r_k!} \sum_{\substack{S, T, U \subseteq [k]: \\ S \cap T \neq \varnothing \\ S \cap U = \varnothing}} d^{- |S \cup T|} \cdot \prod_{i \not \in S} r_i \cdot 
((n - 1)!)^{k - 1} \cdot M^{n - 2}_{\rr + \ee_T - \ee_U - {\bf 1}}.
\end{multline*}
The coefficient $M^{n - 2}_{\rr + \ee_T - \ee_U - {\bf 1}}$ is equal to $0$ whenever the sum of the lower indices is larger than $(n - 2)(k - 1)$. This sum is 
\[
\left(\sum r_i\right) + |T| - |U| - k = n(k - 1) + |T| - |U| - k \geq n(k - 1) + 1 - (n - 1) - k = (n - 2)(k - 1),
\]
with equality if and only if $S = T = \{i\}$ and $U = [k]
\smallsetminus \{i\}$ for some $i \in [k]$.  In this case, by
  Proposition~\ref{prop:Ms2multinomial} we have that the coefficient is $M^{n - 2}_{\rr + \ee_T - \ee_U - {\bf 1}} = \frac{(n - 2)! \cdot (n - r_i)!}{(n - r_1)! \cdots (n - r_k)! \cdot (n - r_i - 2)!}$, and so the number of nontransitive genus-$0$ factorizations is
\begin{multline*}
\frac{d}{ r_1! \cdots r_k!} \sum_{i = 1}^k d^{-1} \cdot \prod_{j \neq i} r_j \cdot ((n - 1)!)^{k - 1} \cdot\frac{(n - 2)! \cdot (n - r_i)!}{(n - r_1)! \cdots (n - r_k)! \cdot (n - r_i - 2)!} =  {}\\
{} = \frac{((n - 1)!)^k}{(r_1 - 1)! \cdots (r_k - 1)! \cdot (n - r_1)! \cdots (n - r_k)! } \sum_{i = 1}^k  \frac{(r_i - 1)! \cdot (n - 2)! \cdot  (n - r_i)! }{(n - 1)! \cdot r_i! \cdot (n - r_i - 2)!} =  {} \\
{} =  \binom{n - 1}{r_1 - 1} \cdots \binom{n - 1}{r_m - 1} \sum_{i = 1}^k \frac{\binom{n - 2}{r_i}}{\binom{n - 1}{r_i - 1}}.
\end{multline*}
Substituting $r_i = n - s_i$ gives the first result, and summing over all $(s_i)$ of sum $n$ gives the zeta polynomial.
\end{proof}

\begin{remark}
\label{nc remark}
In the literature on the poset of noncrossing partitions, they are typically introduced in the following way (e.g., see \cite[\S2.4]{Armstrong} or \cite[\S2.3]{BessisReiner}): in place of the function $\codim\fix(-)$, one defines the \emph{reflection length} $\lR(-)$ by $\lR(g) = \min \{k \colon \exists (t_1, \ldots, t_k) \in R^k, t_1 \cdots t_k = g\}$.  Reflection length is clearly subadditive, and one defines a partial order $\leq_R$ on $G$ by $x \leq_R y$ if $\lR(x) + \lR(x^{-1}y) = \lR(y)$.  Then the \emph{lattice $NC(G, c)$ of $G$-noncrossing partitions} is defined to be the interval $[1, c]_{\leq_R}$.  (All such intervals are isomorphic, for the same sort of reasons as discussed in Remark~\ref{rem:reflection automorphisms}.)

For any element $g$ of any subgroup $G$ of $\GL(V)$ generated by reflections (not necessarily finite; over an arbitrary field), it is easy to see that $\codim\fix(g) \leq \lR(g)$.  If $G$ is a finite real reflection group, Carter proved \cite[Lem.~2]{Carter} that in fact $\codim \fix(g) = \lR(g)$, and consequently the two orders $\leq$ and $\leq_R$ coincide in this case.  The same is true in the wreath product $G(d, 1, n)$ \cite[Rem.~2.3]{Shi}, and in a variety of other settings \cite{BradyWattOrthogonal, HLR}.  However, equality does \emph{not} hold in any other irreducible complex reflection group \cite{FG}.

In several places in the literature (e.g., in
\cite[Lem.~4.1(ii)]{BessisCorran}), one finds versions the following
deduction: from the inequality $\lR(g) \geq \codim\fix(g)$ for all $g \in G$, the equality $\lR(c) = \codim\fix(c)$, and the subadditivity of $\codim\fix(-)$, it follows that if $g \in NC(G, c)$ then $\lR(g) = \codim\fix(g)$ and so $g \in [1, c]_{\leq}$, with the same rank.  However, we were unable to find the following question addressed in the literature: is there a complex reflection group $G$ and an element $g$ of $G$ such that $g \in [1, c]_{\leq}$ but $g \not \in NC(G, c)$?

In \cite[Thm.~8.1]{BessisCorran}, Bessis and Corran prove that the
zeta polynomial for the lattice $NC(G(d, d, n), c_{(d, d, n)})$ is
equal to $\frac{n + d(n - 1)(k - 1)}{n} \prod_{i = 1}^{n - 1} \frac{i
  + (n - 1)(k - 1)}{i} = \frac{n + d(n - 1)(k - 1)}{n} \binom{(n-
  1)k}{n - 1}$; that is, it agrees with the zeta polynomial for the
interval $[1, c_{(d, d, n)}]_{\leq}$ from Corollary~\ref{ddn genus 0}.
In particular, taking $k = 2$, the two intervals have the same size,
and so in fact they must have the same set of elements.
\end{remark}

This coincidence holds in a strong form for all well generated complex reflection groups.

\begin{corollary}
\label{cor:alt NC}
If $G$ is an irreducible well generated complex reflection group and $c$ is a Coxeter element, then the interval $[1, c]_{\leq}$ in the $\codim\fix(-)$-order $\leq$ is identical to the $G$-noncrossing partition lattice $NC(G, c)$.
\end{corollary}
\begin{proof}
The situation in the case that $G$ is real or is in the infinite families $G(d, 1, n)$, $G(d, d, n)$ for $d > 1$ is described in Remark~\ref{nc remark}.  For $G$ of rank $2$, the conditions $1 < g < c$ and $1 <_{R} g <_{R} c$ are both equivalent to the statement that $g$ and $g^{-1}c$ are reflections. For the remaining exceptional groups ($G_{24}$, $G_{25}$, $G_{26}$, $G_{27}$, $G_{29}$, $G_{32}$, $G_{33}$, $G_{34}$), we checked by a brute-force calculation in Sage \cite{sage} that the number of elements in $[1, c]_{\leq}$ is the same as the \emph{$G$-Catalan number}\footnote{The precise attribution of this equality is complicated.  Already in the real case, it was first handled independently in several cases -- see \cite[\S3]{Chapoton} for a summary.  The proofs for $G(d, 1, n)$ and $G(d, d, n)$ may be found in \cite{Reiner} and \cite{BessisCorran}, and the proof for exceptional groups follows from \cite[Thm.~2.2 and Cor.~13.2]{Bessis}.}
$|NC(G, c)| = \prod_i \frac{h + d_i}{d_i}$.  (In the very large group
$G_{34}$, our computation make use of the fact that the interval $[1,
c]_{\leq}$ has symmetric rank sizes: the map $g \mapsto g^{-1}c$ is easily seen to be a rank-reversing bijection from $[1, c]_{\leq}$ to itself.)  Thus the underlying sets of the two intervals are equal.  Since $\lR(w) = \codim\fix(w)$ and $\lR(w^{-1}c) = \codim\fix(w^{-1}c)$ for every $w \in NC(W, c)$, the two intervals are equal as posets.
\end{proof}

There does not seem to be an obvious reason for the inclusion $[1,
c]_{\leq} \subseteq NC(G, c)$ to hold.  Nevertheless, we are very
surprised not to find it remarked upon in the literature!  For open questions along these lines, see Section~\ref{sec:intervals in two posets}.

\section{More refined counting: cycle type for $G(d, 1, n)$}
\label{cycle type}
In \cite{dHR}, the Chapuy--Stump result (Theorem~\ref{thm:cs}) was
refined as follows: rather than lumping all reflections together, they
were divided into classes according to the orbit of their fixed space
under the action of the group.  For example, in $G(d, 1, n)$, this
separates the reflections into two classes: the transposition-like
reflections (which form a conjugacy class) and the diagonal reflections
(which are \emph{not} all conjugate if $d > 2$, but whose fixed spaces
$\{ (x_1, \ldots, x_n) \colon x_i = 0\}$ form one orbit under the
action of the group).  This refinement makes perfect sense in our 
setting: one could ask to count arbitrary factorizations of a
Coxeter element, tracking the orbit of the fixed space of each factor.
In this section, we consider this question for the group $G(d, 1, n)$, refining our main theorem (Theorem~\ref{G(d, 1, n) many factors theorem}) in this case.

In the case of the symmetric group, fixed-space orbits correspond exactly to
cycle types (i.e., conjugacy classes). Beginning with the work \cite{GJ} of Goulden--Jackson for the genus-$0$ case, 
numerous authors have tackled this problem, using a mix of algebraic and combinatorial
techniques, counting the factorizations directly
\cite{goupil-schaeffer, PS,CFF} or counting colored factorizations \cite{MV,B} (as in our approach).  
These works culminated in the following theorem of Bernardi--Morales, 
where the generating polynomial is
written in terms of the power sum $p_{\lambda}$ and monomial symmetric
functions $m_{\mu}$ in $k$ distinct sets of variables
$\yy_i=(y_1^{(i)},y_2^{(i)},\ldots)$.
 
\begin{theorem}[{\cite[Cor.~1.4]{BM1}}] 
\label{thm: Jackson cycle type}
For $\lambda^{(1)},\ldots,\lambda^{(k)}$ partitions of $n$, let
$a_{\lambda^{(1)},\ldots,\lambda^{(k)}}$ be the number of $k$-tuples
$(\pi_1,\ldots,\pi_k)$ of elements in $\mathfrak{S}_n$ such that
$\pi_i$ has cycle type $\lambda^{(i)}$ for $i=1,\ldots,k$, and
$\pi_1\cdots \pi_k = \cn$.  We have
\begin{equation} \label{equation  Jackson cycle type}
\sum_{\lambda^{(1)},\ldots,\lambda^{(k)}}
a_{\lambda^{(1)},\ldots,\lambda^{(k)}} p_{\lambda^{(1)}}(\yy_1)\cdots
p_{\lambda^{(k)}}(\yy_k) =
(n!)^{k-1}\sum_{\mu^{(1)},\ldots,\mu^{(k)}}
\frac{M^{n-1}_{\ell(\mu^{(1)})-1,\ldots,\ell(\mu^{(k)})-1}}{
  \binom{n-1}{\ell(\mu^{(i})-1}\cdots \binom{n-1}{\ell(\mu^{(k})-1} } m_{\mu^{(1)}}(\yy_1)\cdots
m_{\mu^{(k)}}(\yy_k),
\end{equation}
where both sums are over partitions of $n$.
\end{theorem}

In $G = G(d, 1, n)$ for $d > 1$, two elements $u$ and $w$ have fixed spaces in the same $G$-orbit if and only if they have the same number of $k$-cycles of weight $0$ for each $k$.  Thus, orbits of fixed spaces are indexed by the partition $\lambda_0$ of
Proposition~\ref{prop:reflection length}. In this section, we
refine Theorem~\ref{G(d, 1, n) many factors theorem} by keeping track of this partition
$\lambda_0$ for each of the factors in a factorization of a
Coxeter element. 

\begin{theorem} \label{G(d, 1, n) many factors theorem cycle type}
For $d > 1$, let $G = G(d, 1, n)$, let $c$ be a fixed Coxeter element
in $G$, and let $a^{(d)}_{\lambda^{(1)},\ldots,\lambda^{(k)}}$ be the
number of factorizations of $c$ as a product of $k$ elements of $G$
with weight-$0$ cycle type $\lambda^{(1)},\ldots,\lambda^{(k)}$,
respectively.  Then
\begin{multline} \label{equation G(d,1,n) many factors cycle type}
\sum_{\lambda^{(1)},\ldots,\lambda^{(k)}} a_{\lambda^{(1)},\ldots,\lambda^{(k)}}^{(d)}
\prod_{i=1}^k p_{\lambda^{(i)}}(1, \underbrace{y^{(i)}_1, \ldots,
  y^{(i)}_1}_{d},  \underbrace{y^{(i)}_2, \ldots, y^{(i)}_2}_{d},
\ldots)  \,=\,\\
=
|G|^{k-1} \sum_{\mu^{(1)},\ldots,\mu^{(k)}} \frac{M^{n-1}_{q_1-1,\ldots,q_k-1}}{\prod_{i=1}^k
  \binom{n-1}{q_i-1}}
m_{\mu^{(1)}}(\yy_1)\cdots m_{\mu^{(k)}}(\yy_k),
\end{multline}
where the sum on the RHS is over partitions $\mu^{(i)}$ of size at most $n$ such that not all are of size $n$,
and $q_j := \begin{cases}
\ell(\mu^{(j)}) &\text{ if } |\mu^{(j)}| = n \\
\ell(\mu^{(j)})+1 & \text{ otherwise}
\end{cases}$.
\end{theorem}

The rest of this section is devoted to the proof of this result.

\begin{definition} \label{def: colored factorizations Sn by type}
Given compositions $\alpha^{(1)},\ldots,\alpha^{(k)}$ of $n$, let $\mathcal{C}^{\langle n\rangle}_{\alpha^{(1)},\ldots,\alpha^{(k)}}$ be the    set of
factorizations in $\mathfrak{S}_n$ of the fixed $n$-cycle $\cn$ as a
product $\pi_1\cdots \pi_k$ such that for each $i$, the cycles of
$\pi_i$ are colored with $\ell(\alpha^{(i)})$ colors and for $j=1,\ldots,\ell(\alpha^{(i)})$
the sum of the lengths of the cycles colored with the $j$th color is
$\alpha^{(i)}_j$. We say that such factorizations have \emph{type}
$\alpha^{(1)},\ldots,\alpha^{(k)}$.  Let
$C^{\langle n\rangle}_{\alpha^{(1)},\ldots,\alpha^{(k)}} := |\mathcal{C}^{\langle n\rangle}_{\alpha^{(1)},\ldots,\alpha^{(k)}}|$ be the number of colored
factorizations of type $\alpha^{(1)},\ldots,\alpha^{(k)}$. 
\end{definition}

By relabelling colors, one has $C^{\langle
  n\rangle}_{\alpha^{(1)},\ldots,\alpha^{(k)}} = C^{\langle n\rangle}_{\mu^{(1)},\ldots,\mu^{(k)}}$ where $\mu^{(i)}$ is the
underlying partition of $\alpha^{(i)}$ for $i=1,\ldots,k$.  Furthermore, one has $C^{\langle n\rangle}_{p_1,\ldots,p_k} =
\sum_{\alpha^{(i)}}
C^{\langle n\rangle}_{\alpha^{(1)},\ldots,\alpha^{(k)}}$ where the sum is over
compositions $\alpha^{(i)}$ of $n$ with $p_i$ parts for
$i=1,\ldots,k$  (see Remark~\ref{remark: independence color sets} about the
independence of choice of sets of colors).   As in Proposition~\ref{prop:change of basis}, a
change of basis in the generating polynomial for $a_{\lambda^{(1)},\ldots,\lambda^{(k)}}$ may be interpreted in terms of the enumeration of
colored factorizations. 

\begin{proposition}[{\cite[(2.2.15)]{M}}]
\label{prop:change of basis cycle type}
With $a_{\lambda^{(1)},\ldots,\lambda^{(k)}}$ and
$C^{\langle n\rangle}_{\mu^{(1)},\ldots,\mu^{(k)}}$ as defined above, one
has
\begin{equation}
\sum_{\lambda^{(1)},\ldots,\lambda^{(k)}}
a_{\lambda^{(1)},\ldots,\lambda^{(k)}} p_{\lambda^{(1)}}(\yy_1)\cdots
p_{\lambda^{(k)}}(\yy_k)  = 
\sum_{\mu^{(1)},\ldots,\mu^{(k)}}
C^{\langle n\rangle}_{\mu^{(1)},\ldots, \mu^{(k)}} m_{\mu^{(1)}}(\yy_1) \cdots  m_{\mu^{(k)}}(\yy_k),
\end{equation}
where the sums on both sides are over partitions of $n$.
\end{proposition}

The following corollary is an immediate consequence of Theorem~\ref{thm: Jackson cycle type}
and Proposition~\ref{prop:change of basis cycle type}. 

\begin{corollary} \label{corollary symmetric group cycle type}
Given compositions $\alpha^{(1)},\ldots,\alpha^{(k)}$ of $n$, with $p_i=\ell(\alpha^{(i)})$ for
$i=1,\ldots,k$, 
the number of colored factorizations of type 
$\alpha^{(1)},\ldots,\alpha^{(k)}$ of an $n$-cycle in $\mathfrak{S}_n$
 is 
\[
C^{\langle n\rangle}_{\alpha^{(1)},\ldots,\alpha^{(k)}} = (n!)^{k-1}\frac{M^{n-1}_{p_1-1,\ldots,p_k-1}}{\binom{n-1}{p_1-1}\cdots
\binom{n-1}{p_k-1}}.
\] 
\end{corollary}

In order to prove Theorem~\ref{G(d, 1, n) many factors theorem cycle type}, we follow the same approach involving projections of cycle-colorings as in
Theorem~\ref{G(d, 1, n) many factors theorem}. Given compositions
$\alpha^{(1)},\ldots,\alpha^{(k)}$ of size at most $n$ with
$p_i:=\ell(\alpha^{(i)})$, let
$\mathcal{C}^{(d,1,n)}_{\alpha^{(1)},\ldots,\alpha^{(k)}}$ be the set of
colored factorizations in $\mathcal{C}^{(d,1,n)}_{p_1,\ldots,p_k}$
such that for $i=1,\ldots,k$ and $j=1,\ldots,p_i$,   $\alpha^{(i)}_j$ is the sum of the lengths of the cycles colored with colors in the
strip $\{(j-1)d+1,\ldots,jd\}$. Let
$C_{\alpha^{(1)},\ldots,\alpha^{(k)}}^{(d,1,n)} =
|\mathcal{C}_{\alpha^{(1)},\ldots,\alpha^{(k)}}^{(d,1,n)}|$. By permuting the colors we have that
\begin{equation} \label{eq: permuting colors} 
C_{\alpha^{(1)},\ldots,\alpha^{(k)}}^{(d,1,n)}=C_{\mu^{(1)},\ldots,\mu^{(k)}}^{(d,1,n)},
\end{equation}
where $\mu^{(i)}$ is the underlying partition of the composition
$\alpha^{(i)}$ for $i=1,\ldots,k$.
 
The first part of the proof of the theorem is a formula for
$C^{(d,1,n)}_{\boldsymbol \alpha}$ in terms of the counts $C^{\langle n\rangle}_{\boldsymbol
  \gamma}$ of colored factorizations of the $n$-cycle in
$\mathfrak{S}_n$ reviewed above.

\begin{lemma} \label{G(d, 1, n) lemma cycle type}
Fix compositions $\alpha^{(1)},\ldots,\alpha^{(k)}$ of size at most $n$. If $|\alpha^{(i)}| = n$ for all $i = 1, \ldots, k$ then $C^{(d,1,n)}_{\alpha^{(1)},\ldots,\alpha^{(k)}} = 0$.  Otherwise, 
we have 
\[
C^{(d,1,n)}_{\alpha^{(1)},\ldots,\alpha^{(k)}} = d^{(k-1)n} \cdot C^{\langle n\rangle}_{\gamma^{(1)},\ldots,\gamma^{(k)}},
\]
where $\gamma^{(i)} := \begin{cases}
\alpha^{(i)} &\text{ if } |\alpha^{(i)}| = n \\
\left(n-|\alpha^{(i)}|\right) \oplus \alpha^{(i)}  & \text{ otherwise}
\end{cases}$, and $\oplus$ denotes concatenation.
\end{lemma}

\begin{proof}
Since $c$ has nonzero weight, in any colored factorization $c = u_1 \cdots u_k$ there is some factor $u_i$ of nonzero weight, and this factor has a cycle of nonzero weight.  Consequently, the composition $\alpha^{(i)}$ associated to this factor has size strictly less than $n$.  Therefore, given compositions $\alpha^{(1)},\ldots,\alpha^{(k)}$ with $|\alpha^{(i)}| = n$ for all $i$, the set $\mathcal{C}^{(d,1,n)}_{\alpha^{(1)},\ldots,\alpha^{(k)}}$ is empty.  This completes the first part of the lemma.

Now suppose $|\alpha^{(i)}| < n$ for some $i$.  Given a colored factorization $c = u_1 \cdots u_k$ in $\mathcal{C}^{(d,1,n)}_{\alpha^{(1)},\ldots,\alpha^{(k)}}$ of the Coxeter
element $c$ for $G(d, 1, n)$, we associate to it
a colored factorization $\cn = \pi_1 \cdots \pi_k$ of the $n$-cycle
$\cn$ in $\Symm_n$ with the same projection as in the proof of Lemma~\ref{G(d, 1, n) lemma}. Thus, in the resulting colored factorization of $\cn$,
the $i$th factor $\pi_i$ is colored in either $\ell(\alpha^{(i)})$ or $\ell(\alpha^{(i)})+1$ colors depending on whether or not  $\pi_i$ has a cycle
colored $0$, with every color appearing. 
Moreover, for $t=1,\ldots,\ell(\alpha^{(i)})$, the sum of the lengths of the cycles colored
$t$ is $\alpha^{(i)}_t$, and the sum of the
lengths of the cycles colored $0$ is $n-|\alpha^{(i)}|$. Therefore, the resulting colored factorization $\cn =
\pi_1 \cdots \pi_k$ is in
$\mathcal{C}^{\langle n\rangle}_{\gamma^{(1)},\ldots,\gamma^{(k)}}$ where
$\gamma^{(i)}$ is as defined in the statement of the lemma.  

Second, we must consider how many preimages each factorization in
$\mathcal{C}^{\langle n\rangle}_{\gamma^{(1)},\ldots,\gamma^{(k)}}$ has under this
projection map. By the same
  argument as in Lemma~\ref{G(d, 1, n) lemma} the number of preimages is
  $d^{n(k-1)}$.
\end{proof}

\begin{proof}[Proof of Theorem~\ref{G(d, 1, n) many factors theorem cycle type}]
Let $u$ be an element of $G(d,1,n)$ and consider a cycle of $u$ of size $m$. The symmetric
function 
\[
p_m(1, \underbrace{ y_1, \ldots, y_1}_{\textstyle d},
\underbrace{ y_2, \ldots, y_2}_{\textstyle d}, \ldots) = d \cdot
p_m(y_1,y_2,\ldots) + 1
\] 
is the generating function for coloring such a cycle with color set $\ZZ_{\geq 0}$, where the
exponent of the variable $y_j$ records how many elements of $[n]$ belong to a cycle that is colored from the $j$th $d$-strip $\{(j-1)d+1, \ldots, jd\}$.
Thus the LHS of \eqref{equation G(d,1,n) many factors cycle type} is the generating function of colored factorizations
$c=u_1\cdots u_k$ of the Coxeter element $c$ in which the cycles of
factor $u_i$ are colored with the color set $\ZZ_{\geq 0}$, where
the elements colored in the $d$-strip $\{(j-1)d+1,\ldots,jd\}$ are
encoded with the variable $y^{(i)}_j$. To count these
factorizations by the total number of elements in the cycles colored  in
each $d$-strip, we change basis to monomials $({\yy^{(i)}})^{\alpha^{(i)}}$, yielding
\[
\sum_{\lambda^{(1)},\ldots,\lambda^{(k)}} a_{\lambda^{(1)},\ldots,\lambda^{(k)}}^{(d)}
\prod_{i=1}^k p_{\lambda^{(i)}}(1, \underbrace{y^{(i)}_1, \ldots,
  y^{(i)}_1}_{d},  \underbrace{y^{(i)}_2, \ldots, y^{(i)}_2}_{d},
\ldots) =
\sum_{\alpha^{(1)},\ldots,\alpha^{(k)}}
C^{(d,1,n)}_{\alpha^{(1)},\ldots,\alpha^{(k)}} ({\yy^{(1)}})^{\alpha^{(1)}} \cdots ({\yy^{(k)}})^{\alpha^{(k)}},
\]
where the sum on the right is over weak compositions $\alpha^{(i)}$ of numbers no larger than $n$.
By Lemma~\ref{G(d, 1,
  n) lemma cycle type} and Corollary~\ref{corollary symmetric group cycle type} we have that 
\[
C^{(d,1,n)}_{\alpha^{(1)},\ldots,\alpha^{(k)}} = d^{(k-1)n} (n!)^{k-1} \frac{M^{n-1}_{q_1-1,\ldots,q_k-1}}{\prod_{i=1}^k
  \binom{n-1}{q_i-1}},
\]
where $q_i = \begin{cases}
\ell(\alpha^{(i)}) & \text{ if } |\alpha^{(i)}| =n\\
\ell(\alpha^{(i)})+1 & \text{ otherwise}
\end{cases}$ for $i=1,\ldots,k$. Finally, by \eqref{eq: permuting colors} we
can combine monomials $({\yy^{(i)}})^{\alpha^{(i)}}$  with the same
underlying partition $\mu^{(i)}$ to rewrite the RHS in the basis
$m_{\mu^{(i)}}$, as desired. 
\end{proof}

\begin{remark}
Theorem~\ref{G(d, 1, n) many factors theorem} can be recovered from
Theorem~\ref{G(d, 1, n) many factors theorem cycle type} as
follows. We do a
{stable
principal specialization} $\yy_i = 1^{x_i}$ for $i=1,\ldots,k$ (see
\eqref{eq: stable specialization}) in \eqref{equation G(d,1,n) many factors cycle type}. One has
\[
  p_{\lambda}(1, \underbrace{y_1, \ldots,
  y_1}_{d},  \underbrace{y_2, \ldots, y_2}_{d},
\ldots)\,\,
{\Bigg |}_{\yy = 1^{x}} 
\,\, =\,\, (xd+1)^{\ell(\lambda)}.
\]
This gives 
\begin{multline} \label{eq1: proof specialization cycle type}
\sum_{r_1,\ldots,r_k} \left(\sum_{\lambda^{(i)} \colon
    \ell(\lambda^{(i)})=r_i}
  a^{(d)}_{\lambda^{(1)},\ldots,\lambda^{(k)}} \right)
(x_1d+1)^{r_1} \cdots (x_kd+1)^{r_k} \,=\, \\
= |G|^{k-1}
\sum_{p_1,\ldots,p_k}  \binom{x_1}{p_1}\cdots \binom{x_k}{p_k} \sum_{\alpha^{(i)}\colon \ell(\alpha^{(i)})=p_i}
\frac{M^{n-1}_{q_1-1,\ldots,q_k-1}}{\binom{n-1}{q_1-1}\cdots \binom{n-1}{q_k-1}},
\end{multline}
where the sum on the RHS is over compositions $\alpha^{(i)}$ of size at most
$n$ with $p_i$ parts, with not all compositions of size $n$. 
Grouping tuples of compositions according to the subset of indices $i$
with $|\alpha^{(i)}| < n$, one counts the compositions in each case to
cancel the binomials, applies Proposition~\ref{prop:recurrence Ms}, and does a substitution $(x_id+1) \mapsto x_i$ to finish. 
\end{remark}

\subsection{The genus-$0$ case}

The leading terms in the LHS of the equation in Theorem~\ref{thm: Jackson cycle type}, 
of degree $\sum_i \ell(\lambda^{(i)}) = n(k-1)+1$,
count genus-$0$ factorizations for $\mathfrak{S}_n$ by the cycle type of the factors. 
Goulden and Jackson \cite[Thm.~3.2]{GJ} used Lagrange inversion to give a formula for such factorizations, called the \emph{cactus formula}:
\[
a_{\mu^{(1)},\ldots,\mu^{(k)}} =  n^{k-1}\cdot \prod_{i=1}^k \frac{
(\ell(\mu^{(i)})-1)!}{\Aut(\mu^{(i)})}.
\]
Krattenthaler and M\"uller \cite{KMtypeB} gave formulas counting genus-$0$ factorizations by cycle type for signed permutations (type $B_n$) and even-signed permutations (type $D_n$), also using Lagrange inversion. As an application of Theorem~\ref{G(d, 1, n) many factors theorem cycle type}, we obtain a formula for the genus-$0$ factorizations for $G=G(d,1,n)$ by their weight-$0$ cycle type that is independent of $d$ and coincides with the type $B_n$ formulas in \cite[Thm.~7(i)]{KMtypeB}. Since $\ell(\lambda_0(w)) = \dim\fix(w)$, the genus-$0$ factorizations are those counted by $a^{(d)}_{\lambda^{(1)}, \ldots, \lambda^{(k)}}$ such that $\sum_i \ell(\lambda^{(i)}) = n(k - 1)$. 

\begin{corollary} \label{cor: cycle type genus 0}
For $d>1$, let $G=G(d,1,n)$, and let $c$ be a Coxeter element in
$G$. Given partitions $\lambda^{(1)},\ldots,\lambda^{(k)}$ of size at most
$n$ with $\sum_{i=1}^k \ell(\lambda^{(i)}) = n(k-1)$, we have that
$a^{(d)}_{\boldsymbol \lambda}$ is nonzero if and only if there is a unique $j$ such that $|\lambda^{(j)}| < n$, and in this case 
\[
a^{(d)}_{\lambda^{(1)},\ldots,\lambda^{(k)}} = n^{k-1}\cdot \ell(\lambda^{(j)}) \cdot \prod_{i=1}^k \frac{
(\ell(\lambda^{(i)})-1)!}{\Aut(\lambda^{(i)})}.
\]
\end{corollary}

\begin{proof}
Let $F$ denote the generating function in \eqref{equation G(d,1,n)
  many factors cycle type} and let $\langle \cdot ,\cdot \rangle$ denote the
usual Hall inner product of symmetric functions, for which %the $p_\lambda$ form an orthogonal basis with 
$\langle p_\lambda, p_\mu\rangle = \delta_{\lambda\mu} z_\lambda$. Fix a tuple of partitions $(\lambda^{(i)})$ with $\sum_{i=1}^k
\ell(\lambda^{(i)}) = n(k-1)$.  On one hand, thinking about the LHS
formula for $F$ in \eqref{equation G(d,1,n) many factors cycle type}, we have
\[
a^{(d)}_{\lambda^{(1)},\ldots,\lambda^{(k)}} = \left \langle \prod_i \frac{p_{\lambda^{(i)}}(\yy_i)}{d^{\ell(\lambda^{(i)})} z_{\lambda^{(i)}}}, F\right \rangle,
\]
since the top-degree term of $p_\lambda(1, y_1, \ldots, y_1, y_2, \ldots, y_2, \ldots)$ is $d^{\ell(\lambda)} p_{\lambda}(\yy_i)$.
On the other hand, thinking about the RHS formula for
$F$ in \eqref{equation G(d,1,n) many factors cycle type}, we have that
\[
 \left \langle \prod_i p_{\lambda^{(i)}}(\yy_i), F\right \rangle = |G|^{k-1} \sum_{\mu^{(1)},\ldots,\mu^{(k)}} \frac{M^{n-1}_{q_1-1,\ldots,q_k-1}}{\prod_{i=1}^k
  \binom{n-1}{q_i-1}}
\prod \langle p_{\lambda^{(i)}}(\yy_i), m_{\mu^{(i)}}(\yy_i)\rangle.
 \]
In order for $ \langle p_{\lambda^{(i)}}, m_{\mu^{(i)}}\rangle$ not to be zero, we must have that $\mu^{(i)}$ is a refinement of $\lambda^{(i)}$, and so in particular $\ell(\mu^{(i)}) \geq \ell(\lambda^{(i)})$ for each $i$.  It follows that, in order for a term to contribute, it must be the case that
\begin{align*}
\sum_i (q_i - 1) & = \sum_i \ell(\mu^{(i)}) - \#\{i \colon |\mu^{(i)}| = n\} \\
& \geq \sum_i \ell(\lambda^{(i)}) - \#\{i \colon |\mu^{(i)}| = n\} \\
& = n(k - 1) - \#\{i \colon |\mu^{(i)}| = n\} \\ 
& \geq (n - 1)(k - 1),
\end{align*}
where in the penultimate step we use the fact that no terms appear on the RHS with all $\mu$s of size $n$.  On the other hand, in order for $M^{n - 1}_{q_1 - 1, \ldots, q_k - 1}$ to be nonzero, it must be the case that $\sum (q_i - 1) \leq (n - 1)(k - 1)$.  Thus, the $(\mu^{(i)})$ term on the right contributes if and only if $\mu^{(i)} = \lambda^{(i)}$ for all $i$ and also exactly $k - 1$ of the $\lambda^{(i)}$ have size $n$.  That is, if there is any nonzero contribution at all, it comes from the single term when $\mu^{(i)} = \lambda^{(i)}$ for all $i$, and in this case we have
\begin{align*}
 \left \langle \prod_i p_{\lambda^{(i)}}(\yy_i), F\right \rangle & = |G|^{k-1} \frac{M^{n-1}_{q_1-1,\ldots,q_k-1}}{\prod_{i=1}^k
  \binom{n-1}{q_i-1}}
\prod_{i=1}^k \langle p_{\lambda^{(i)}}, m_{\lambda^{(i)}}\rangle \\
& = |G|^{k - 1} \frac{\binom{n - 1}{n - q_1; \ldots ; n - q_k}} {\prod_{i=1}^k
  \binom{n-1}{q_i-1}} \prod_{i=1}^k \frac{z_{\lambda^{(i)}}}{\Aut(\lambda^{(i)})} 
\end{align*}
and therefore
\[
a^{(d)}_{\lambda^{(1)},\ldots,\lambda^{(k)}}  =  \frac{|G|^{k - 1}}{d^{\sum\ell(\lambda^{(i)})}} \frac{\binom{n - 1}{n - q_1; \ldots ; n - q_k}} {\prod_{i=1}^k
  \binom{n-1}{q_i-1}} \prod_{i=1}^k \frac{1}{\Aut(\lambda^{(i)})}
=
n^{k - 1} \prod_{i=1}^k \frac{(q_i-1)!}{\Aut(\lambda^{(i)})},
 \]
and the result follows immediately.
\end{proof}

\section{Remarks and open problems}
\label{sec:remarks}

\subsection{A uniform formula?}
\label{sec:uniform}

In light of the beautiful Chapuy--Stump formula (Theorem~\ref{thm:cs}), the main question raised by our work is the following one. 
\begin{question}
\label{uniform?}
Is there a uniform answer that incorporates the various formulas that appear in Theorems~\ref{G(d, 1, n) many factors theorem} and~\ref{thm:rank 2} and Proposition~\ref{rank 3 best}?
\end{question}

As mentioned in Remark~\ref{Orlik--Terao remark} and
Section~\ref{sec:exceptional higher ranks}, one obstruction is that
for the cases \emph{not} covered by these theorems (and in particular,
for $G(d, d, n)$ when $n > 3$), there are more than one set of
Orlik--Solomon coexponents, and so there is no obvious choice of basis
in which to write the factorization polynomial.

Our work on $G(d, d,
n)$ suggests another approach towards a uniform statement, namely
studying factorizations with a transitivity condition. This leads to the
next question.

\begin{question}
Can one give a definition of \emph{transitive factorization}, valid for all well generated complex reflection groups, generalizing the definition for $G(d, d, n)$ of Section~\ref{sec:ddn}?
\end{question}

One would hope that such a definition would yield a uniform generalization of Theorem~\ref{thm:mainG(d,d,n) k factors} and the aforementioned results.  

Another approach might start with the work of Douvropoulos \cite{Theo}.  He uses a natural grouping on the irreducible characters of an arbitrary well generated complex reflection group to give an elegant, uniform proof of the Chapuy--Stump formula, as well as a weighted generalization.  The grouping exhibits cancellation similar to what we observe in the proof of Theorem~\ref{thm:rank 2}.

\subsection{Maps and constellations}
\label{sec:maps}
In the symmetric group, combinatorial proofs of the formulas
counting factorizations of the long cycle into two or more factors are phrased in terms of {\em
  maps} or {\em constellations}, that is, certain graphs embedded in surfaces (e.g., see \cite{LZ,JacksonVisentin,GS}).
We briefly discuss two possible variants of maps for encoding
factorizations of a Coxeter element into two factors of the complex reflection group $G(d,1,n)$. 

The first version of maps is to encode the underlying factorization in the
symmetric group with the usual {\em rooted unicellular bipartite map}
\cite[\S  1.3.3, 1.5.1]{LZ} \cite[\S 3.1.2]{JacksonVisentin} \cite[Ex.~2.1]{SchaefferVassilieva} and add the
cyclic group weights on the edges, as follows.
Given a factorization $u\cdot v$ of the Coxeter element $c$ in $G(d,1,n)$, we assign to this factorization a map, with the following conventions (see Figure~\ref{fig:map}):
\begin{itemize}
\item the map has $n$ edges labelled $1,2,\ldots,n$;
\item each cycle of $u$ corresponds to a black vertex and each cycle of $v$ corresponds to a white vertex;
\item the labels of the edges incident to a black vertex are the entries in the corresponding cycle of $u$, and appear in clockwise order around the vertex; 
and similarly for white vertices and $v$; and
\item the edge labelled $i$ has additional labels corresponding to the weight of the $i$th column of $u$ and $i$th row of $v$.
\end{itemize}
The fact that $u\cdot v = c$ means that if we start at the rooted black vertex
and traverse the map, recording the labels on the edges when we go from a black to a white vertex, then we see the edges in the order $1, 2, \ldots, n$ given by the long cycle $\cn$. 

\begin{figure}
\centering
  \begin{subfigure}[b]{0.4\textwidth}
\includegraphics[scale=0.9]{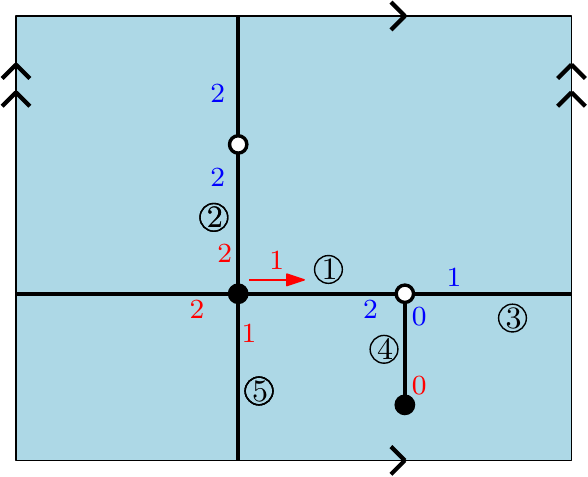}
\caption{}
\label{fig:map}
  \end{subfigure}
  \begin{subfigure}[b]{0.4\textwidth}
\includegraphics[scale=0.6]{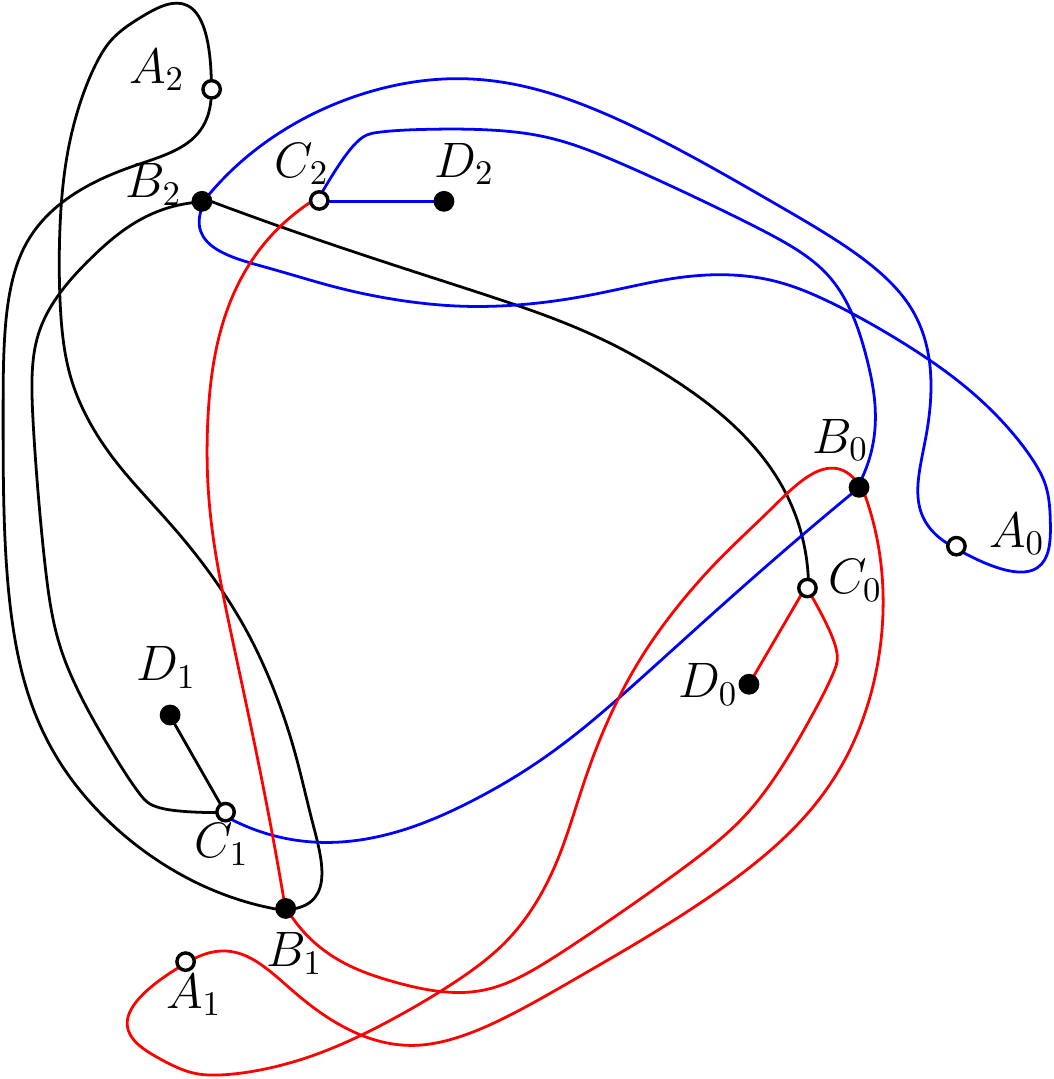}
\caption{}
\label{fig:unfolded map}
  \end{subfigure}
\caption{(a) The genus-$1$ map of the factorization 
$c = [(1532)(4); (1, 2,2,0,1)] \cdot [(134)(25); (1, 2,0,2, 2)]$ in Example~\ref{example:map}. (b) An ``unfolding'' of the map on the left, using the (blue) $v$-weights to determine which copies of the vertices to connect with each edge.}
\end{figure}

\begin{example}
\label{example:map}
Let $u=[(1532)(4); (1, 2,2,0,1)]$ and $v=[(134)(25); (1, 2,0,2, 2)]$ be the following elements in $G(3,1,5)$:
\[
u = \begin{bmatrix}
 & z^2  &  &  & \\
 &  & z^2 &  &\\
& & & & z\\
& & & 1 & \\
z & & & & 
\end{bmatrix}, \quad 
v = \begin{bmatrix}
 & &  &z^2 & \\
& &  && z^2 \\
z &  & & & \\
& & 1  & & \\
& z^2 & & & 
\end{bmatrix}, \quad u\cdot v = \begin{bmatrix}
 &  &  & & z\\
1&  & &  & \\
& 1 &  & & \\
&  & 1 &  & \\
& & &  1&  
\end{bmatrix} = c.
\] 
The corresponding weighted map is shown in Figure~\ref{fig:map}.
\end{example}

In the case of genus-$0$ factorizations with two factors, there are known
combinatorial objects that correspond to the factors: in $\Symm_n$,
the number of elements that can appear in a
genus-$0$ factorization of the long cycle $\cn$ is the Catalan number $C_n$, and they are in natural correspondence with noncrossing partitions on $n$ points.  In $G(d, 1, n)$ for $d > 1$, there are $\binom{2n}{n}$ (the type $B$ Catalan number) elements that can appear in a genus-$0$ factorization of the Coxeter element $c$, and they are in natural correspondence with noncrossing partitions on $dn$ points having $d$-fold rotational symmetry \cite{Reiner} -- see Figure~\ref{fig:nc}.
\begin{figure}
\includegraphics{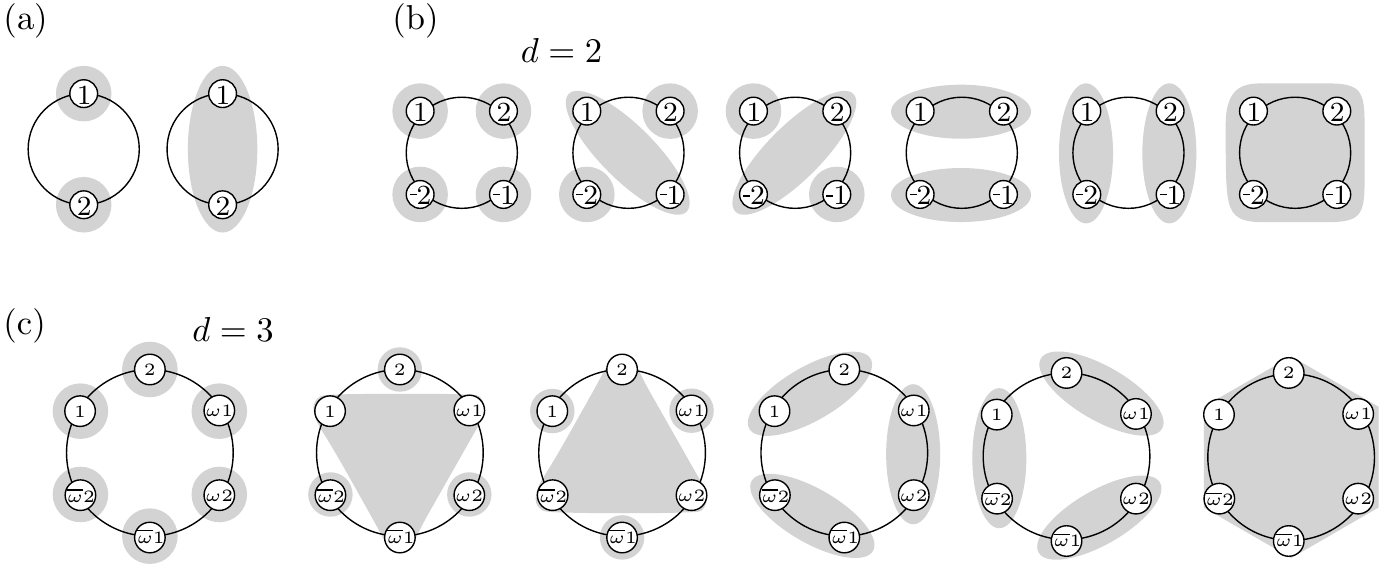}
\caption{(a) The $C_2 = 2$ set partitions on $n = 2$ points are both noncrossing.  (b, c) The $\binom{2 \cdot 2}{2} = 6$ noncrossing set partitions on $2d$ points with $d$-fold rotational symmetry, for $d = 2$ and $d = 3$. }
\label{fig:nc}
\end{figure}
Thus, one might hope that the map objects corresponding to
factorizations in $G(d, 1, n)$ should also have a $d$-fold symmetry.
This gives rise to a second version of maps on surfaces, in which one
``unfolds'' a weighted map into an object with $d$ copies of each vertex and edge, where an edge of weight $m$ that connected vertices $A$ and $B$ unfolds to edges connecting copies $A_{i}$ and $B_{i + m}$ (with indices in $\ZZ/d\ZZ$) -- see Figure~\ref{fig:unfolded map}.

Unfortunately, neither map correspondence seems completely
satisfactory: the topological genus of the map is controlled by the total number of cycles in the factors, rather than the number of weight-$0$ cycles. (For example, in the first variant, the genus of the map is exactly the genus of the underlying factorization in $\Symm_n$.)  Is there a better correspondence between factorizations in $G(d, 1, n)$ and some kind of maps?  

\subsection{Better combinatorial proofs in the symmetric group?}
\label{S_n questions}
Our work suggests (at least) two natural questions purely in the context of the symmetric group.  The first concerns the $(n - 1)$-cycle.
\begin{question}
\label{n - 1 cycle question}
Can one give a combinatorial proof of Theorem~\ref{S_n (n-1) cycle
  theorem k factors}, counting transitive factorizations of an $(n -
1)$-cycle in $\Symm_n$ into $k$ factors? 
\end{question}

For example, one might try to explain why 
$n^{k-1} (n-1)^k C^{\langle n-1,1\rangle}_{\bf p}  = p_1\cdots p_k C^{\langle n \rangle}_{\bf p+1}$. 
We are able to give a combinatorial proof for the case of $k=2$
factors. However, it does not seem to extend easily to general
$k$. We state this proof in terms of permutations but it could equivalently
be phrased in terms of maps using a contraction of a digon.

\begin{figure}
\centering
\includegraphics{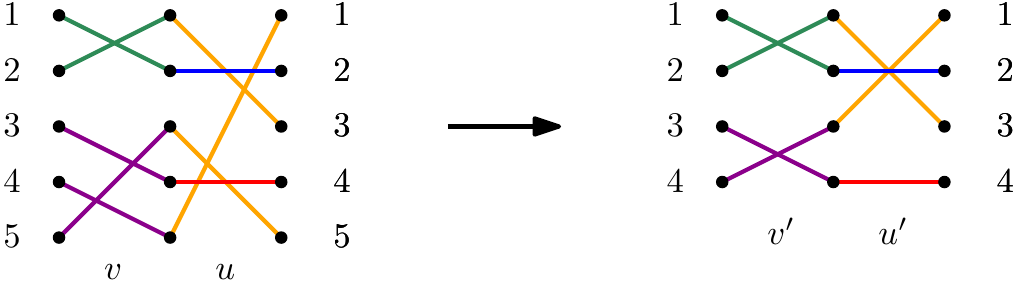}
\caption{The transitive factorization $(1234)(5) = (135)(2)(4)\cdot (12)(345)$ of a $4$-cycle in $\Symm_5$ is one of four that correspond to the factorization $(1234) = (13)(2)(4) \cdot (12)(34)$ of a $4$-cycle in $\Symm_4$.}
\label{fig:two factor (n-1)-cycle}
\end{figure}

\begin{proof}[Proof of Theorem~\ref{S_n (n-1) cycle theorem k factors} for $k = 2$ factors]
We give an $(n - 1)$-to-$1$ map from transitive factorizations of a fixed $(n - 1)$-cycle in $\Symm_n$ to factorizations of a fixed $(n - 1)$-cycle in $\Symm_{n - 1}$.

Without loss of generality, let $\cnn = (1 \cdots (n - 1))(n)$.  Thus for any factorization $u \cdot v = \cnn$ in $\Symm_n$, we have $v(n) = u^{-1}(n) =: t$.  This factorization is transitive if and only if the common value $t$ belongs to $[n - 1]$, and equivalently the value $n$ belongs to a cycle of length longer than $1$ in both $u$ and $v$.  Then $u' := u \cdot (t\; n)$ and $v' := (t\; n) \cdot v$ are the permutations that result from deleting $n$ from its cycle in $u$ and $v$, respectively; we view them as permutations in $\Symm_{n - 1}$ (rather than as permutations in $\Symm_n$ that fix $n$).  Moreover, $u'$ and $v'$ have the same number of cycles as $u$ and $v$, respectively.

On the other hand, for any factorization $u' \cdot v' = (1 \cdots (n - 1))$ in $\Symm_{n - 1}$ and any $t \in [n - 1]$, we may think of $u'$ and $v'$ as elements of $\Symm_n$ that fix $n$; then $u := u' \cdot (t\; n)$ and $v := (t \; n) \cdot v'$ are two permutations in $\Symm_n$ such that $u \cdot v = \cnn$ is a transitive factorization, and $u$ and $v$ have the same number of cycles as $u'$ and $v'$, respectively.

It follows that the number $b_{r, s}$ of transitive factorizations of $\cnn$ using factors with $r$ and $s$ cycles is equal to $n - 1$ times the number $a^{(n - 1)}_{r, s}$ of factorizations of an $(n - 1)$-cycle in $\Symm_{n - 1}$ using factors with $r$ and $s$ cycles.  By Theorem~\ref{S_n theorem}, one has
\begin{align*}
\sum_{r, s} b_{r, s} x^r y^s & = (n - 1) \cdot \sum_{r, s} a^{(n - 1)}_{r, s} x^r y^s \\
& = (n - 1) \cdot (n - 1)! \cdot \sum_{p,q} \binom{n-2}{p-1; q-1; n - p - q} \frac{(x)_{p}}{p!} \frac{(y)_{q}}{q!} \\
& = \frac{n!}{n^2} \cdot \sum_{p, q} \binom{n}{p; q; n - p - q}\frac{(x)_{p}}{(p - 1)!} \frac{(y)_{q}}{(q - 1)!}.
\end{align*}
When $k = 2$, $M^n_{p, q} = \binom{n}{p; q; n - p - q}$, and so this completes the proof.
\end{proof}

The correspondence used in the proof also behaves well with colored factorizations with $k = 2$ factors.

The existing combinatorial proofs of Theorem~\ref{S_n theorem} in the
case of $k$ factors use involved arguments with
maps or constellations (as in Section~\ref{sec:maps}), the BEST theorem, the matrix-tree theorem, and
sign-reversing involutions; see \cite{BM1,BM2}. The second question concerns the possibility of ``nicer'' combinatorial proofs for the $n$-cycle.  
\begin{question}
Proposition~\ref{prop:recurrence Ms} and Corollary~\ref{S_n corollary} imply the recurrence relation
\[
C^{\langle n \rangle}_{p_1,\ldots,p_k} = n^{k-1} \sum_{S \subsetneq [k]} C^{\langle n-1\rangle}_{\pp-\ee_S}
\]
for counts of colored factorizations of long cycles.  Is there a
direct combinatorial proof of this relation, perhaps in terms of
Bernardi's  {\em
  tree-rooted maps} (see \cite{B,BM1})?
\end{question}

\subsection{More refined counting: tracking cycles by weights}

While tracking the number of cycles of weight $0$ is natural from the
geometric perspective, from the combinatorial point of view it makes
just as much sense to ask for the full distribution of cycle weights.
Building on the work of Chapuy--Stump \cite[\S5.3]{ChapuyStump},
we were able to prove the
following result using the character-theory approach.
In contrast to the rest of the paper, the statement of this result \emph{does} depends on the weight of the Coxeter element chosen.
\begin{theorem}
\label{prop:all weights}
For $d > 1$, let $G = G(d, 1, n)$ and let $c$ be a Coxeter element in
$G$ of weight $1$.  For $i = 1, \ldots, k$, let $\rr_i = (r_{i, 0},
\ldots, r_{i, d - 1})$ be a tuple of nonnegative integers, and let
$a^{(d)}_{\rr_1, \dots, \rr_k}$ be the number of factorizations $c =
u_1 \cdots u_k$ of $c$ as a product of $k$ factors such that $u_i$ has
exactly $r_{i, j}$ cycles of weight $j$ for each $j = 0, \ldots, d -
1$.  Let $\xx_i$ denote the variable set $\{x_{i, 0}, \ldots, x_{i, d
  - 1}\}$.  Then
 \[
%F(\xx_1, \ldots, \xx_k) := 
\sum_{\rr_1, \ldots, \rr_k} a^{(d)}_{\rr_1, \dots, \rr_k} \xx_1^{\rr_1} \cdots \xx_k^{\rr_k}
=
|G|^{k - 1} \sum_{t^d = 1}   \sum_{\pp} M^{n - 1}_{p_1 - 1, \ldots, p_k - 1}  \prod_i \binom{\left(x_{i, 0} + t x_{i, 1} + \dots + t^{d - 1} x_{i, d - 1}\right)/d}{p_i}.
\] 
\end{theorem}

By setting $x_{i, 0} = x_i$ and $x_{i, j} = 1$ for $j = 1, \ldots, d - 1$, one immediately recovers on the LHS the polynomial of Theorem~\ref{G(d, 1, n) many factors theorem}; to make the RHS match requires a computation using Proposition~\ref{prop:recurrence Ms}.

\begin{remark}
The RHS in Theorem~\ref{prop:all weights} is, up to the power of $d$, the result of substituting $x_i \mapsto (x_{i, 0} + t x_{i, 1} + \dots + t^{d - 1} x_{i, d - 1})/d$ into the RHS of \eqref{eqn S_n k factors}, then extracting powers of $t$ modulo $t^d - 1$.  This raises the question of whether the same substitution has combinatorial meaning when factoring other elements.  Unfortunately, the answer seems to be negative: the factorization polynomial for the identity into two factors in $\Symm_2$ is $x^2y^2 + xy$, and substituting $x = (x_0 + t x_1)/2$ and $y = (y_0 + ty_1)/2$ and summing over $t = \pm1$ yields 
\[
\frac{x_0^2y_0^2 + x_1^2 y_0^2 + x_0^2y_1^2 + x_1^2y_1^2 + 4x_0x_1y_0y_1 + 4x_0y_0 + 4x_1y_1}{8}, 
\]
with fractional coefficients.  
For comparison, the polynomials counting two-factor factorizations of diagonal matrices in the signed permutation group $G(2, 1, 2)$ are
\begin{center}
\begin{tabular}{cp{1in}l}
$x_0^2y_0^2 + 2x_0x_1y_0y_1 + x_1^2y_1^2 + 2x_0y_0 + 2x_1y_1$ && (the identity), \\
$x_0^2y_0y_1 + x_0x_1y_0^2 + x_0x_1y_1^2 + x_1^2y_0y_1 + 2x_0y_1 + 2x_1y_0$ && (the two reflections), and \\
$x_0^2y_1^2 + 2x_0x_1y_0y_1 + x_1^2y_0^2 + 2x_0y_0 + 2x_1y_1$ && (the negative of the identity).
 \end{tabular}
 \end{center}
\end{remark}

\begin{question}
Is there a better expression for the polynomial that appears in Theorem~\ref{prop:all weights}, for example, in terms of a natural basis of polynomials?  Does it have a combinatorial proof?
\end{question}

\subsection{Refinement by fixed space orbit}

As discussed in Section~\ref{cycle type}, the cycle type result there makes sense for any complex reflection group $G$, where one refines by the $G$-orbit of fixed space of the factors.  

\begin{question}
Does Theorem~\ref{G(d, 1, n) many factors theorem cycle type} generalize?  What about its genus-$0$ special case?
\end{question}
A natural test-case would be transitive
factorizations of a Coxeter element in $G(d,d,n)$, refining
Theorem~\ref{thm:mainG(d,d,n) k factors}.  The first obvious obstruction to be overcome is the choice of a ``good'' basis of polynomials in which to express the generating function. 

\subsection{Intervals in two posets}
\label{sec:intervals in two posets}

It would be desirable to have a uniform proof of Corollary~\ref{cor:alt NC}, that the $G$-noncrossing partition lattice may be equivalently defined as an interval when $G$ is ordered by the reflection-length order $\leq_{R}$ or the codim-fix order $\leq$, even though the larger posets are not the same.

In light of Corollary~\ref{cor:alt NC}, it is natural to ask under what conditions the interval $[1, w]_{\leq}$ in the codim-fix order is equal to the interval $[1, w]_{\leq_R}$ in the reflection-length order; for example, does equality always hold when $\codim \fix(w) = \lR(w)$?  The next example shows that the answer to this question is negative.
\begin{example}
In $G(7, 7, 6)$, consider the elements $w = [123456; (1, 2, 3, 4, 5, 6)]$, $u = [123456; (1, 2, 0, 4, 0, 0)]$ and $v = [123456; (0, 0, 3, 0, 5, 6)]$.  It is easy to check that $w = u \cdot v$, $\lR(w) = \codim\fix(w) = 6$, $\codim\fix(u) = \codim \fix(v) = 3$, and $\lR(u) = \lR(v) = 4$.  Thus $u, v$ belong to the interval $[1, w]_{\leq}$ but not to the interval $[1, w]_{\leq_R}$.
\end{example}
\begin{question}
Can one characterize the elements $w$ for which $[1, w]_{\leq} = [1, w]_{\leq_R}$?  In particular, does equality hold for all regular elements $w$?
\end{question}

\subsection{Generating function by reflection length?}
One possibility, suggested by Remark~\ref{nc remark}, is that instead of writing a generating function for factorizations using fixed space dimension, one should write a generating function that records the reflection length $\lR(g)$ of the factors.  However, there is again an obstruction in terms of choosing a basis: typically in these groups there are elements of reflection length $> n$, so basis polynomials of degree $> n$ would be required, and it is not clear what would be a good choice of such polynomials.

\bibliography{CRG}
\bibliographystyle{alpha}

\appendix

\section{Tables of character polynomials}
\label{app:tables of character polynomials}

In this section, we collect tables of the character polynomials $f_{\chi}(x) := \sum_{w \in W} \normchi(w) x^{\dim \fix(w)}$ for the exceptional complex reflection groups.  These were computed on CoCalc using Sage \cite{sage} and GAP \cite{GAP}.  We show only the characters $\chi$ such that $\chi(c^{-1}) \neq 0$ where $c$ is a fixed Coxeter element.  One feature common to all groups is the appearance of the degrees with the trivial character
\[
f_{\mathrm{triv}}(x) = \sum_{w \in W} x^{\dim \fix w} = \prod (x - 1 + d_i)
\]
and the coexponents with the determinant character
\[
f_{\mathrm{det}}(x) = \sum_{w \in W} \det(w) x^{\dim \fix w} = \prod (x - e^*_i)
\]
(in the real cases, the sign character) -- cf.\ Section~\ref{sec:degrees and coexponents}.

\subsection{Rank two}
%Here is the fundamental data (degrees, codegrees, exponents, coexponents, etc.) associated to each of the rank-$2$ groups.  
%\[
%\begin{array}{c|ccccccc}
%\textrm{gp} & d 	& e 		& d^*  	& e^* 	& h 	& |R| & |A| \\ \hline
%I_2(r)& 2, r 	& 1, r - 1	& 0, r - 2	& 1, r - 1	& r 	& & \\
%G(r,1,2)& 2, 2r	& 1, 2r - 1	& 0, r 	& 1, r + 1	& 2r	& & \\
%4 	&  4, 6	& 3, 5	& 0, 2	& 1, 3	& 6	& & \\
%5 	&  6, 12	& 5, 11	& 0, 6	& 1, 7	& 12	& & \\
%6 	&  4, 12	& 3, 11	& 0, 8	& 1, 9	& 12	& & \\
%8 	&  8, 12	& 7, 11	& 0, 4	& 1, 5	& 12	& & \\
%9 	&  8, 24	& 7, 23	& 0, 16	& 1, 17	& 24	& & \\
%10 	&  12, 24	& 11, 23	& 0, 12	& 1, 13	& 24	& & \\
%14 	& 6, 24	& 5, 23	& 0, 18	& 1, 19	& 24	& & \\
%16 	& 20, 30	& 19, 29	& 0, 10	& 1, 11	& 30	& & \\
%17 	& 20, 60	& 19, 59	& 0, 40	& 1, 41	& 60	& & \\
%18 	& 30, 60	& 29, 59	& 0, 30	& 1, 31	& 60	& & \\
%20 	& 12, 30	& 11, 29	& 0, 18	& 1, 19	& 30	& & \\
%21 	& 12, 60	& 11, 59	& 0, 48	& 1, 49	& 60	& & 
%\end{array}
%\]

There are twelve well generated exceptional groups of rank $2$.  The tables below give the coexponents $(e^*_1, e^*_2)$, the degrees $(d_1, d_2)$, and, for each character $\chi$ such that $\chi(c^{-1}) \neq 0$, the dimension $\deg(\chi) = \chi(1)$, the character value $\chi(c^{-1})$, and the character polynomial $f_\chi$, expressed in the basis
\[
P_2 = \frac{(x - e^*_1)(x - e^*_2)}{d_1 \cdot d_2}, \qquad P_1 = \frac{x - 1}{d_1}, \qquad P_0 = 1.
\]
In the first two tables, two groups appear in each table; for the subsequent (larger) groups, each group appears in its own table, possibly in several columns.

\begin{table}[H]
% G_4
\raisebox{.5in}{
$
\begin{array}{c|c|c}
\deg(\chi) & \chi(c^{-1}) & f_{\chi}(x) \\\hline\hline
1 & 1 &  24P_2 + 48P_1 + 24\\
\hline
1 & \z_3 & 24P_2 \\
1 & -1 - \z_3 & \\
2 & 1 & \\
\hline
2 & \z_3 & 24P_2 + 24P_1 \\
2 & -1 - \z_3 &
\end{array}
$}
\qquad\qquad
% G_5
$
\begin{array}{c|c|c}
\deg(\chi) & \chi(c^{-1}) & f_{\chi}(x) \\\hline\hline
1 & 1 & 72P_2 + 144 P_1 + 72 \\\hline
1 &  \z_3   & 72P_2 + 72P_1 \\
1 & \z_3 & \\
1 & -1 - \z_3 & \\
1 & -1 - \z_3 & \\
\hline
1 & 1 & 72P_2 \\
1 & 1 & %72P_2 
\\
1 & \z_3 & %72P_2 
\\
1 & -1 - \z_3 & %72P_2 
\\\hline
3 &  -1   & 72P_2 + 48P_1 \\
3 &  -\z_3   & %72P_2 + 48P_1 
\\
3 &   1 + \z_3  & %72P_2 + 48P_1
\end{array}
$
\caption{$G_4$: $(e^*_1, e^*_2) = (1, 3)$, $(d_1, d_2) = (4, 6)$.  $G_5$: $(e^*_1, e^*_2) = (1, 7)$, $(d_1, d_2) = (6, 12)$.}
\end{table}

\newgeometry{margin=.5in}

\begin{table}[H]
% G_6
$
\begin{array}{c|c|c}
\deg(\chi) & \chi(c^{-1}) & f_{\chi}(x) \\\hline\hline
1 & 1 & 48P_2 + 96P_1 + 48 \\
\hline
1 & - \z_3 & 48P_2 \\
1 & 1 + \z_3 & \\
\hline
1 & -1 & 48P_2 + 48 P_1 \\
1 & \z_3 & \\
1 & -1 - \z_3 & \\
2 & \z_{12} & \\
2 & \z_{12}^5 & \\
2 & \z_{12}^7 & \\
2 & \z_{12}^{11} & \\
\hline
2 & i & 48P_2 + 24P_1 \\
2 & -i &
\end{array}
$
%\end{table}
%
\qquad\qquad
%\begin{table}[H]
%\caption{$G_8$}
$
\begin{array}{c|c|c}
\deg(\chi) & \chi(c^{-1}) & f_{\chi}(x) \\\hline\hline
1 & 1 & 96P_2 + 192P_1 + 96 \\
\hline
1 & 1 & 96P_2 \\
1 & -1 & \\
1 & -1 & \\
2 & 1 & \\
2 & i & \\
2 & -i & \\
\hline
2 & -1 & 96P_2 + 96P_1 \\ 
2 & i & \\
2 & -i & \\
\hline
4 & i & 96P_2 + 48P_1 \\
4 & -i & 
\end{array}
$
\caption{$G_6$: $(e^*_1, e^*_2) = (1, 9)$, $(d_1, d_2) = (4, 12)$.  $G_8$: $(e^*_1, e^*_2) = (1, 5)$, $(d_1, d_2) = (8, 12)$.}
\label{table:G_6}
\end{table}

\begin{table}[H]
% G_9
\raisebox{12pt}{$
\begin{array}{c|c|c}
\deg(\chi) & \chi(c^{-1}) & f_{\chi}(x) \\\hline\hline
1 & 1 & 192P_2 + 384P_1 + 192 \\
\hline
1 & -1 & 192P_2 + 192P_1 \\
1 & -1 & \\
1 & -i & \\
1 & i & \\
2 & 1 & \\
2 & -1 & \\
2 & \z_8 & \\
2 & - \z_8 & \\
2 & \z_8^3 & \\
2& - \z_8^3 & 
\end{array}
$}
\qquad
$
\begin{array}{c|c|c}
\deg(\chi) & \chi(c^{-1}) & f_{\chi}(x) \\\hline\hline
1 & 1 & 192P_2\\
1 & i & \\
1 & -i & \\
\hline
2 & i & 192P_2 + 96 P_1 \\
2 & -i & \\
2 & \z_8 & \\
2 & - \z_8 & \\
2& \z_8^3 & \\
2 & -\z_8^3 & \\
\hline
4 & \z_8 & 192P_2 + 144P_1 \\
4 & - \z_8 & \\
4 & \z_8^3 & \\
4 & - \z_8^3 &
\end{array}
$
\caption{$G_9$: $(e^*_1, e^*_2) = (1, 17)$, $(d_1, d_2) = (8, 24)$.}
\end{table}

\begin{table}[H]
% G_{10} 
$
\begin{array}{c|c|c}
\deg(\chi) & \chi(c^{-1}) & f_{\chi}(x) \\\hline\hline
1 & 1 & 288P_2 + 576P_1 + 288 \\
\hline
1 & -1 & 288P_2 + 288P_1 \\
1 & i & \\
1 & -i & \\
1 & \z_{3} & \\
1 &  -1 - \z_3 & \\
\hline
1 &  -\z_{3} & 288P_2 \\
1 & 1 + \z_3 & \\
1 &  \z_{12} & \\
1 & - \z_{12} & \\
1 &  \z_{12}^5 & \\
1 &  -\z_{12}^5 & 
\end{array}
$
\qquad
$
\begin{array}{c|c|c}
\deg(\chi) & \chi(c^{-1}) & f_{\chi}(x) \\\hline\hline
3 & 1 & 288P_2 + 96P_1 \\
3 &  \z_{3} & \\
3 &  -1 - \z_3 & \\
\hline
3 & -1 & 288P_2 + 192P_1 \\
3 & i & \\
3 & - i & \\
3 & -\z_{3}& \\
3 & 1 + \z_3& \\
3 &\z_{12} & \\
3 & -\z_{12}& \\
3 & \z_{12}^5 & \\
3 & - \z_{12}^5 & 
\end{array}
$
\caption{$G_{10}$: $(e^*_1, e^*_2) = (1, 13)$, $(d_1, d_2) = (12, 24)$.}
\end{table}

\begin{table}[H]
\raisebox{11.5pt}{$
\begin{array}{c|c|c}
\deg(\chi) & \chi(c^{-1}) & f_{\chi}(x) \\\hline\hline
1 & 1 & 144P_2 + 288P_1 + 144 \\
\hline
1 & -\z_3 & 144P_2 \\
1 & 1 + \z_{3} & %144P_2 
\\\hline
2 & \z_{8} + \z_{8}^3 & 144P_2 + 72P_1 \\
2 & -\z_{8} - \z_{8}^3 & %144P_2 + 72P_1 
\\\hline
3 & 1 & 144P_2 + 96P_1 \\
3 & \z_3 & %144P_2 + 96P_1 
\\
3 & -1 - \z_{3} & %144P_2 + 96P_1 
\end{array}
$}
\qquad
$
\begin{array}{c|c|c}
\deg(\chi) & \chi(c^{-1}) & f_{\chi}(x) \\\hline\hline
1 & -1 & 144P_2 + 144P_1 \\
1 & \z_{3} & %144P_2 + 144P_1 
\\
1 & -1 - \z_3 & %144P_2 + 144P_1 
\\
2 & \z_{24} - \z_{24}^7 & %144P_2 + 144P_1 
\\
2 & -\z_{24} + \z_{24}^7 & %144P_2 + 144P_1 
\\
2 & \z_{24}^{5} - \z_{24}^{11} & %144P_2 + 144P_1 
\\
2 & - \z_{24}^{5} + \z_{24}^{11} & %144P_2 + 144P_1 
\\
3 & -1 & %144P_2 + 144P_1 
\\
3 & -\z_{3} & %144P_2 + 144P_1 
\\
3 & 1 + \z_3 & %144P_2 + 144P_1 
\end{array}
$
\caption{$G_{14}$: $(e^*_1, e^*_2) = (1, 19)$, $(d_1, d_2) = (6, 24)$.}
\end{table}

\begin{table}[H]
\raisebox{11.5pt}{$
\begin{array}{c|c|c}
\deg(\chi) & \chi(c^{-1}) & f_{\chi}(x) \\\hline\hline
1 & 1 & 600P_2 + 1200P_1 + 600 \\
\hline
4 & \z_{5} & 600P_2 + 300P_1 \\
4 & -\z_{5} & %600P_2 + 300P_1
\\
4 & \z_{5}^2 & %600P_2 + 300P_1
\\
4 & -\z_{5}^2 & %600P_2 + 300P_1
\\
4 & \z_{5}^3 & %600P_2 + 300P_1
\\
4 & -\z_{5}^3 & %600P_2 + 300P_1
\\
4 & \z_{5}^4 & %600P_2 + 300P_1
\\
4 & -\z_{5}^4 & %600P_2 + 300P_1
\\\hline
5 & -1 & 600P_2 + 240P_1 \\

5 & -\z_{5} & %600P_2 + 240P_1
\\
5 & -\z_{5}^2 & %600P_2 + 240P_1
\\
5 & -\z_{5}^3 & %600P_2 + 240P_1
\\
5 & -\z_{5}^4 & %600P_2 + 240P_1
\end{array}
$}
\qquad
$
\begin{array}{c|c|c}
\deg(\chi) & \chi(c^{-1}) & f_{\chi}(x) \\\hline\hline
1 & \z_{5} & 600P_2 \\

1 & \z_{5}^2 & %600P_2 \\
\\
1 & \z_{5}^3 & %600P_2 \\
\\
1 & \z_{5}^4 & %600P_2 \\
\\
2 & 1 & %600P_2 \\
\\
2 & 1 & %600P_2 \\
\\
2 & \z_{5} & %600P_2 \\
\\
2 & \z_{5}^2 & %600P_2 \\
\\
2 & \z_{5}^3 & %600P_2 \\
\\
2 & \z_{5}^4 & %600P_2 \\
\\
4 & 1 & %600P_2 \\
\\
4 & -1 & %600P_2 \\
\\\hline
2 & \z_{5} & 600P_2 + 600P_1 \\

2 & \z_{5}^2 & %600P_2 + 600P_1
\\
2 & \z_{5}^3 & %600P_2 + 600P_1
\\
2 & \z_{5}^4 & %600P_2 + 600P_1
\end{array}
$
\caption{$G_{16}$: $(e^*_1, e^*_2) = (1, 11)$, $(d_1, d_2) = (20, 30)$.}
\end{table}

\begin{table}[H]
\raisebox{6pt}{
$
\begin{array}{c|c|c}
\deg(\chi) & \chi(c^{-1}) & f_{\chi}(x) \\\hline\hline
1 & 1 & 1200P_2 + \hfill \\
&&  + 2400P_1 + \\
&& \hfill + 1200 \\
\hline
1 & -\z_{5} & 1200P_2 \\

1 & -\z_{5}^2 & %1200P_2 
\\
1 & -\z_{5}^3 & %1200P_2 
\\
1 & -\z_{5}^4 & %1200P_2
\\\hline
1 & -1 & 1200P_2 + \hfill %1200P_1 
\\
1 & \z_{5} & %1200P_2 
\hfill + 1200P_1
\\
1 & \z_{5}^2 & %1200P_2 + 1200P_1
\\
1 & \z_{5}^3 & %1200P_2 + 1200P_1
\\
1 & \z_{5}^4 & %1200P_2 + 1200P_1
\\
2 & \z_{20} & %1200P_2 + 1200P_1
\\
2 & -\z_{20} & %1200P_2 + 1200P_1
\\
2 & \z_{20}^3 & %1200P_2 + 1200P_1
\\
2 & -\z_{20}^3 & %1200P_2 + 1200P_1
\\
2 & \z_{20}^7 & %1200P_2 + 1200P_1
\\
2 & -\z_{20}^7 & %1200P_2 + 1200P_1
\\
2 & \z_{20}^9 & %1200P_2 + 1200P_1
\\
2 & -\z_{20}^9 & %1200P_2 + 1200P_1
\end{array}
$}
\qquad
$
\begin{array}{c|c|c}
\deg(\chi) & \chi(c^{-1}) & f_{\chi}(x) \\\hline\hline
2 & i & 1200P_2 + \hfill %600P_1 
\\
2 & i & %1200P_2 
\hfill + 600P_1
\\
2 & -i & %1200P_2 + 600P_1
\\
2 & -i & %1200P_2 + 600P_1
\\
2 & \z_{20} & %1200P_2 + 600P_1
\\
2 & -\z_{20} & %1200P_2 + 600P_1
\\
2 & \z_{20}^3 & %1200P_2 + 600P_1
\\
2 & -\z_{20}^3 & %1200P_2 + 600P_1
\\
2 & \z_{20}^7 & %1200P_2 + 600P_1
\\
2 & -\z_{20}^7 & %1200P_2 + 600P_1
\\
2 & \z_{20}^9 & %1200P_2 + 600P_1
\\
2 & -\z_{20}^9 & %1200P_2 + 600P_1
\\
4 & 1 & %1200P_2 + 600P_1
\\
4 & -1 & %1200P_2 + 600P_1
\\
4 & i & %1200P_2 + 600P_1
\\
4 & -i & %1200P_2 + 600P_1
\\\hline
5 & -1 & 1200P_2 + \hfill % 960P_1 
\\
5 & -\z_{5} & %1200P_2 
\hfill + 960P_1
\\
5 & -\z_{5}^2 & %1200P_2 + 960P_1
\\
5 & -\z_{5}^3 & %1200P_2 + 960P_1
\\
5 & -\z_{5}^4 & %1200P_2 + 960P_1
\end{array}
$
\qquad
$
\begin{array}{c|c|c}
\deg(\chi) & \chi(c^{-1}) & f_{\chi}(x) \\\hline\hline

4 & \z_{5} & 1200P_2 + \hfill %900P_1 
\\
4 & -\z_{5} & %1200P_2 
\hfill + 900P_1
\\
4 & \z_{5}^2 & %1200P_2 + 900P_1
\\
4 & -\z_{5}^2 & %1200P_2 + 900P_1
\\
4 & \z_{5}^3 & %1200P_2 + 900P_1
\\
4 & -\z_{5}^3 & %1200P_2 + 900P_1
\\
4 & \z_{5}^4 & %1200P_2 + 900P_1
\\
4 & -\z_{5}^4 & %1200P_2 + 900P_1
\\
4 & \z_{20} & %1200P_2 + 900P_1
\\
4 & -\z_{20} & %1200P_2 + 900P_1
\\
4 & \z_{20}^3 & %1200P_2 + 900P_1
\\
4 & -\z_{20}^3 & %1200P_2 + 900P_1
\\
4 & \z_{20}^7 & %1200P_2 + 900P_1
\\
4 & -\z_{20}^7 & %1200P_2 + 900P_1
\\
4 & \z_{20}^9 & %1200P_2 + 900P_1
\\
4 & -\z_{20}^9 & %1200P_2 + 900P_1
\\\hline
5 & 1 & 1200P_2 + \hfill %720P_1 
\\
5 & \z_{5} & %1200P_2 
\hfill + 720P_1
\\
5 & \z_{5}^2 & %1200P_2 + 720P_1
\\
5 & \z_{5}^3 & %1200P_2 + 720P_1
\\
5 & \z_{5}^4 & %1200P_2 + 720P_1
\end{array}
$
\caption{$G_{17}$: $(e^*_1, e^*_2) = (1, 41)$, $(d_1, d_2) = (20, 60)$.}
\end{table}

\begin{table}[H]
$
\begin{array}{c|c|c}
\deg(\chi) & \chi(c^{-1}) & f_{\chi}(x) \\\hline\hline
1 & 1 & 1800P_2 + \hfill\\
&& +3600P_1 + \\
&& \hfill +1800 \\
\hline
3 & -1 & 1800P_2 + 1200P_1 \\
3 & -1 & %1800P_2 + 1200P_1
\\3 & -\z_{3} & %1800P_2 + 1200P_1
\\3 & -\z_{3} & %1800P_2 + 1200P_1
\\3 & 1 + \z_3 & %1800P_2 + 1200P_1
\\3 & 1 + \z_3 & %1800P_2 + 1200P_1
\\3 & -\z_{5} & %1800P_2 + 1200P_1
\\3 & -\z_{5}^2 & %1800P_2 + 1200P_1
\\3 & -\z_{5}^3 & %1800P_2 + 1200P_1
\\3 & -\z_{5}^4 & %1800P_2 + 1200P_1
\\3 & -\z_{15} & %1800P_2 + 1200P_1
\\3 & -\z_{15}^2 & %1800P_2 + 1200P_1
\\3 & -\z_{15}^4 & %1800P_2 + 1200P_1
\\3 & -\z_{15}^7 & %1800P_2 + 1200P_1
\\3 & -\z_{15}^{8} & %1800P_2 + 1200P_1
\\3 & -\z_{15}^{11} & %1800P_2 + 1200P_1
\\3 & -\z_{15}^{13} & %1800P_2 + 1200P_1
\\3 & -\z_{15}^{14} & %1800P_2 + 1200P_1
\end{array}
$
\qquad
$
\begin{array}{c|c|c}
\deg(\chi) & \chi(c^{-1}) & f_{\chi}(x) \\\hline\hline
1 & \z_{3} & 1800P_2 + \hfill % 1800P_1 
\\1 & -1 - \z_3 & %1800P_2 
\hfill + 1800P_1
\\1 & \z_{5} & %1800P_2 + 1800P_1
\\1 & \z_{5}^2 & %1800P_2 + 1800P_1
\\1 & \z_{5}^3 & %1800P_2 + 1800P_1
\\1 & \z_{5}^4 & %1800P_2 + 1800P_1
\\\hline
5 & 1 & 1800P_2 + \hfill %720P_1 
\\5 & \z_{5} & %1800P_2 
\hfill + 720P_1
\\5 & \z_{5}^2 & %1800P_2 + 720P_1
\\5 & \z_{5}^3 & %1800P_2 + 720P_1
\\5 & \z_{5}^4 & %1800P_2 + 720P_1
\\\hline
5 & \z_{3} & 1800P_2 + \hfill %1080P_1 
\\5 & -1 - \z_3 & %1800P_2 
\hfill + 1080P_1
\\5 & \z_{15} & %1800P_2 + 1080P_1
\\5 & \z_{15}^2 & %1800P_2 + 1080P_1
\\5 & \z_{15}^4 & %1800P_2 + 1080P_1
\\5 & \z_{15}^7 & %1800P_2 + 1080P_1
\\5 & \z_{15}^{8} & %1800P_2 + 1080P_1
\\5 & \z_{15}^{11} & %1800P_2 + 1080P_1
\\5 & \z_{15}^{13} & %1800P_2 + 1080P_1
\\5 & \z_{15}^{14} & %1800P_2 + 1080P_1
\end{array}
$
\qquad
\raisebox{6pt}{$
\begin{array}{c|c|c}
\deg(\chi) & \chi(c^{-1}) & f_{\chi}(x) \\\hline\hline
1 & \z_{15} & 1800P_2 \\
1 & \z_{15}^2 & %1800P_2
\\1 & \z_{15}^4 & %1800P_2
\\1 & \z_{15}^7 & %1800P_2
\\1 & \z_{15}^{8} & %1800P_2
\\1 & \z_{15}^{11} & %1800P_2
\\1 & \z_{15}^{13} & %1800P_2
\\1 & \z_{15}^{14} & %1800P_2
\\\hline
3 & -\z_{5} & 1800P_2 + \hfill %600P_1 
\\3 & -\z_{5}^2 & %1800P_2 
\hfill + 600P_1
\\3 & -\z_{5}^3 & %1800P_2 + 600P_1
\\3 & -\z_{5}^4 & %1800P_2 + 600P_1
\\3 & -\z_{15} & %1800P_2 + 600P_1
\\3 & -\z_{15}^2 & %1800P_2 + 600P_1
\\3 & -\z_{15}^4 & %1800P_2 + 600P_1
\\3 & -\z_{15}^7 & %1800P_2 + 600P_1
\\3 & -\z_{15}^{8} & %1800P_2 + 600P_1
\\3 & -\z_{15}^{11} & %1800P_2 + 600P_1
\\3 & -\z_{15}^{13} & %1800P_2 + 600P_1
\\3 & -\z_{15}^{14} & %1800P_2 + 600P_1
\end{array}
$}
\caption{$G_{18}$: $(e^*_1, e^*_2) = (1, 31)$, $(d_1, d_2) = (30, 60)$.}
\end{table}

\begin{table}[H]
\raisebox{12pt}{$
\begin{array}{c|c|c}
\deg(\chi) & \chi(c^{-1}) & f_{\chi}(x) \\\hline\hline
1 & 1 & 360P_2 + 7200P_1 + 360 \\
\hline
1 & \z_{3} & 360P_2 \\
1 & -1 - \z_3 & %360P_2
\\2 & -\z_{5} -\z_{5}^{4} & %360P_2
\\2 & -\z_{5}^2 - \z_{5}^3 & %360P_2
\\\hline
2 & -\z_{15} - \z_{15}^{4} & 360P_2 + 360P_1
\\2 & -\z_{15}^{2} - \z_{15}^{8} & %360P_2 + 360P_1
\\2 & -\z_{15}^{7} - \z_{15}^{13} & %360P_2 + 360P_1
\\2 & -\z_{15}^{11} - \z_{15}^{14} & %360P_2 + 360P_1
\\4 & 1 & %360P_2 + 360P_1
\\4 & -1 & %360P_2 + 360P_1
\end{array}
$}
\qquad
$
\begin{array}{c|c|c}
\deg(\chi) & \chi(c^{-1}) & f_{\chi}(x) \\\hline\hline
4 & \z_3 & 360P_2 + 180P_1 \\
4 & -\z_{3} & %360P_2 + 180P_1
\\4 & 1 + \z_3 & %360P_2 + 180P_1
\\4 & -1 - \z_{3} & %360P_2 + 180P_1
\\\hline
3 & -\z_{5} -\z_{5}^{4} & 360P_2 + 240P_1 \\
3 & -\z_{5}^2 - \z_{5}^3 & %360P_2 + 240P_1
\\3 & -\z_{15} - \z_{15}^{4} & %360P_2 + 240P_1
\\3 & -\z_{15}^{2} - \z_{15}^{8} & %360P_2 + 240P_1
\\3 & -\z_{15}^{7} - \z_{15}^{13} & %360P_2 + 240P_1
\\3 & -\z_{15}^{11} - \z_{15}^{14} & %360P_2 + 240P_1
\\6 & -1 & %360P_2 + 240P_1
\\6 & -\z_{3} & %360P_2 + 240P_1
\\6 & 1 + \z_3 & %360P_2 + 240P_1
\end{array}
$
\caption{$G_{20}$: $(e^*_1, e^*_2) = (1, 19)$, $(d_1, d_2) = (12, 30)$.}
\end{table}

\begin{table}[H]
$
\begin{array}{c|c|c}
\deg(\chi) & \chi(c^{-1}) & f_{\chi}(x) \\\hline\hline
1 & 1 & 720P_2 + 1440P_1 + 720 \\
\hline
1 & -\z_{3} & 720P_2 \\
1 & 1 + \z_3 & %720P_2
\\\hline
1 & -1 & 720P_2 + 720P_1 \\
1 & \z_{3} & %720P_2 + 720P_1
\\1 & -1 - \z_3 & %720P_2 + 720P_1
\\2 & %-i - \z_{60}^{13} + \z_{60}^7 + \z_{12} 
\z_{60} - \z_{60}^{19} & %720P_2 + 720P_1
\\2 & %i + \z_{60}^{13} - \z_{60}^7 - \z_{12} 
-\z_{60} + \z_{60}^{19} & %720P_2 + 720P_1
\\2 & \z_{60}^{7} - \z_{60}^{13} & %720P_2 + 720P_1
\\2 & -\z_{60}^{7} + \z_{60}^{13} & %720P_2 + 720P_1
\\2 & % -i - \z_{60}^{13} + \z_{60}^{11} + \z_{20}^3 + \z_{60}^7 + \z_{12} - \z_{60}
\z_{60}^{11} - \z_{60}^{29}  & %720P_2 + 720P_1
\\2 & %i + \z_{60}^{13} - \z_{60}^{11} - \z_{20}^3 - \z_{60}^7 - \z_{12} + \z_{60} 
- \z_{60}^{11} + \z_{60}^{29} & %720P_2 + 720P_1
\\2 & % -i - \z_{60}^{13} + \z_{60}^{11} + \z_{20}^3 + \z_{60}^7 - \z_{60}
\z_{60}^{17} - \z_{60}^{23}  & %720P_2 + 720P_1
\\2 & %i + \z_{60}^{13} - \z_{60}^{11} - \z_{20}^3 - \z_{60}^7 + \z_{60} 
- \z_{60}^{17} + \z_{60}^{23} & %720P_2 + 720P_1
\\3 & \z_{5} + \z_{5}^{4} & %720P_2 + 720P_1
\\3 & \z_{5}^2 + \z_{5}^3 & %720P_2 + 720P_1
\\3 & \z_{15} + \z_{15}^4 & %720P_2 + 720P_1
\\3 & %\z_{15}^7 - \z_{3} + \z_{15}^4 - \z_{5} + \z_{15}^2 + \z_{15} - 1 
\z_{15}^{2} + \z_{15}^{8} & %720P_2 + 720P_1
\\3 & \z_{15}^7 + \z_{15}^{13} & %720P_2 + 720P_1
\\3 & % -\z_{15}^7 - \z_{15}^4 + \z_{5} - \z_{15}^2 - \z_{15} + 1 
\z_{15}^{11} + \z_{15}^{14} & %720P_2 + 720P_1
\\4 & 1 & %720P_2 + 720P_1
\\4 & -1 & %720P_2 + 720P_1
\\4 & i & %720P_2 + 720P_1
\\4 & -i & %720P_2 + 720P_1
\end{array}
\qquad
\begin{array}{c|c|c}
\deg(\chi) & \chi(c^{-1}) & f_{\chi}(x) \\\hline\hline
2 & %\z_{20}^7 - i + \z_{20}^3 
	\z_{20} + \z_{20}^9 & 720P_2 + 360P_1
\\2 & \z_{20}^3 + \z_{20}^7 & %720P_2 + 360P_1
\\2 & %-\z_{20}^7 + i - \z_{20}^3 & 720P_2 + 360P_1
	-\z_{20} - \z_{20}^9 & 
\\2 & -\z_{20}^7 - \z_{20}^3 & %720P_2 + 360P_1
\\\hline
3 & -\z_{5} - \z_{5}^{4} & 720P_2 + 480P_1 \\
3 & -\z_{5}^2 - \z_{5}^3 & %720P_2 + 480P_1
\\3 & -\z_{15} - \z_{15}^{4} & %720P_2 + 480P_1
\\3 & -\z_{15}^{2} - \z_{15}^{8} & %720P_2 + 480P_1
\\3 & -\z_{15}^{7} - \z_{15}^{13} & %720P_2 + 480P_1
\\3 & -\z_{15}^{11} - \z_{15}^{14} & %720P_2 + 480P_1
\\\hline
4 & \z_{3} & 720P_2 + 540P_1 \\
4 & -\z_{3} & %720P_2 + 540P_1
\\4 & 1 + \z_3 & %720P_2 + 540P_1
\\4 & -1 - \z_3 & %720P_2 + 540P_1
\\4 & \z_{12} & %720P_2 + 540P_1
\\4 & -\z_{12} & %720P_2 + 540P_1
\\4 & \z_{12}^5 & %720P_2 + 540P_1
\\4 & -\z_{12}^5 & %720P_2 + 540P_1
\\\hline
6 & i & 720P_2 + 600P_1 \\
6 & -i & %720P_2 + 600P_1
\\6 & \z_{12} & %720P_2 + 600P_1
\\6 & -\z_{12} & %720P_2 + 600P_1
\\6 & \z_{12}^5 & %720P_2 + 600P_1
\\6 & -\z_{12}^5 & %720P_2 + 600P_1
\end{array}
$
\caption{$G_{21}$: $(e^*_1, e^*_2) = (1, 49)$, $(d_1, d_2) = (12, 60)$.}
\end{table}

\subsection{Rank three}

For the five exceptional groups of rank $3$, all of which are well generated, we give for each character $\chi$ such that $\chi(c^{-1}) \neq 0$ the character polynomial $f_\chi$ expressed in the given basis (see Section~\ref{sec:rank 3}).

\begin{table}[H]
\raisebox{12pt}{$
\begin{array}{c|c|c}
\deg(\chi) & \chi(c^{-1}) & f_{\chi}(x) \\\hline\hline
1 & 1 & 120P_3 + 360P_2 + 360P_1 + 120 \\
\hline
1 & -1 & 120P_3 \\
\hline
3 & -\z_5 - \z_5^4 & 120P_3 + 120P_2 \\
3 & -\z_5^2 - \z_5^3 & \\
\hline
3 &  \z_5 + \z_5^4 & 120P_3 + 240P_2 + 120P_1 \\
3 &  \z_5^2 + \z_5^3 & \\
\hline
4 & 1 & 120P_3 + 120P_2 + 60P_1 \\
4 & - 1 & 
\end{array}
$}
%\caption{$G_{23}$ (Coxeter $H_3$): $P_3 = \frac{(x - 1)(x - 5)(x - 9)}{2 \cdot 6 \cdot 10}$, $P_2 = \frac{(x - 1)(x - 5)}{2 \cdot 6}$, $P_1 = \frac{x - 1}{2}$.}
%\end{table}
\qquad\qquad
%
%\begin{table}[H]
$
\begin{array}{c|c|c}
\deg(\chi) & \chi(c^{-1}) & f_{\chi}(x) \\\hline\hline
1 & 1 & 336P_3 + 1008P_2 + 1008P_1 + 336 \\
\hline
1 & -1 & 336P_3 \\
\hline
3 & \z_{7} + \z_{7}^2 + \z_{7}^4 & 336P_3 + 336P_2 \\
3 & \z_{7}^3 + \z_{7}^5 + \z_{7}^6 & %336P_3 + 336P_2 
\\
6 & 1 & %336P_3 + 336P_2 
\\\hline
3 & -\z_{7} - \z_{7}^2 - \z_{7}^4 & 336P_3 + 672P_2 + 336P_1 \\
3 & -\z_{7}^3 - \z_{7}^5 - \z_{7}^6 & %336P_3 + 672P_2 + 336P_1 
\\6 & -1 & %336P_3 + 672P_2 + 336P_1 
\\\hline
8 & 1 & 336P_3 + 504P_2 + 126P_1 \\
8 & -1 & %336P_3 + 504P_2 + 126P_1 
\end{array}
$
%\caption{$G_{24}$: $P_3 = \frac{(x - 1)(x - 9)(x - 11)}{4 \cdot 6 \cdot 14}$, $P_2 = \frac{(x - 1)(x - 7)}{4 \cdot 6}$, $P_1 = \frac{x - 1}{3}$.}
\caption{$G_{23}$ (Coxeter $H_3$): $P_3 = \frac{(x - 1)(x - 5)(x - 9)}{2 \cdot 6 \cdot 10}$, $P_2 = \frac{(x - 1)(x - 5)}{2 \cdot 6}$, $P_1 = \frac{x - 1}{2}$.  $G_{24}$: $P_3 = \frac{(x - 1)(x - 9)(x - 11)}{4 \cdot 6 \cdot 14}$, $P_2 = \frac{(x - 1)(x - 7)}{4 \cdot 6}$, $P_1 = \frac{x - 1}{3}$.}
\end{table}

\begin{table}[H]
$
\begin{array}{c|c|c}
\deg(\chi) & \chi(c^{-1}) & f_{\chi}(x) \\\hline\hline
1 & 1 & 648P_3 + 1944P_2 + 1944P_1 + 648 \\
\hline
1 & 1 & 648P_3 \\
1 & 1 & %648P_3 
\\
3 & \z_{3} & %648P_3 
\\
3 & -1 - \z_3 & %648P_3 
\\\hline
3 & \z_{3} & 648P_3 + 648P_2 \\
3 & -1 - \z_3 & %648P_3 + 648P_2 
\\\hline
3 & -1 & 648P_3 + 648P_2 + 216P_1 \\
\hline
3 & \z_{3} & 648P_3 + 1296P_2 + 648P_1 \\
3 & -1 - \z_3 & %648P_3 + 1296P_2 + 648P_1 
\\\hline
9 & -\z_{3} & 648P_3 + 648P_2 + 72P_1 \\
9 & 1 + \z_3 & %648P_3 + 648P_2 + 72P_1 
\end{array}
$
\caption{$G_{25}$: $P_3 = \frac{(x - 1)(x - 4)(x - 7)}{6 \cdot 9 \cdot 12}$, $P_2 = \frac{(x - 1)(x - 4)}{6 \cdot 9}$, $P_1 = \frac{x - 1}{6}$.}
\end{table}

\begin{table}[H]
$
\begin{array}{c|c|c}
\deg(\chi) & \chi(c^{-1}) & f_{\chi}(x) \\\hline\hline
1 & 1 & 1296P_3 + 3888P_2 + 3888P_1 + 1296 \\
\hline
1 & -1 & 1296P_3 + 2592P_2 + 1296P_1 
\\2 & \z_{3} & %1296P_3 + 2592P_2 + 1296P_1 
\\2 & -1 - \z_3 & %1296P_3 + 2592P_2 + 1296P_1 
\\\hline
1 & \z_{3} & 1296P_3 + 1296P_2 \\
1 & -1 - \z_3 & %1296P_3 + 1296P_2 
\\
2 & 1 & %1296P_3 + 1296P_2 
\\
2 & -\z_{3} & %1296P_3 + 1296P_2 
\\
2 & 1 + \z_3 & %1296P_3 + 1296P_2 
%\\\hline
\end{array}
\qquad
\begin{array}{c|c|c}
\deg(\chi) & \chi(c^{-1}) & f_{\chi}(x) \\\hline\hline
1 & -\z_{3} & 1296P_3 \\
1 & 1 + \z_3 & %1296P_3 
\\
2 & -1 & %1296P_3 
\\\hline
8 & 1 & 1296P_3 + 1944P_2 + 648P_1 \\
8 & -1 & %1296P_3 + 1944P_2 + 648P_1 
\\\hline
8 & \z_{3} & 1296P_3 + 1296P_2 + 162P_1 \\
8 & -\z_{3} & %1296P_3 + 1296P_2 + 162P_1 
\\
8 & 1 + \z_3 & %1296P_3 + 1296P_2 + 162P_1 
\\
8 & -1 - \z_3 & %1296P_3 + 1296P_2 + 162P_1 
\end{array}
$
\caption{$G_{26}$: $P_3 = \frac{(x - 1)(x - 7)(x - 13)}{6 \cdot 12 \cdot 18}$, $P_2 = \frac{(x - 1)(x - 7)}{6 \cdot 12}$, $P_1 = \frac{x - 1}{6}$.}
\end{table}

\begin{table}[H]
$
\begin{array}{c|c|c}
\deg(\chi) & \chi(c^{-1}) & f_{\chi}(x) \\\hline\hline
1 & 1 & 2160P_3 + 6480P_2 + 6480P_1 + 2160 \\
\hline
1 & -1 & 2160P_3 \\
\hline
8 & \z_{5} + \z_{5}^{4} & 2160P_3 + 3240P_2 + 810P_1 
\\8 & - \z_{5} - \z_{5}^{4} & %2160P_3 + 3240P_2 + 810P_1 
\\8 & \z_{5}^2 + \z_{5}^3 & %2160P_3 + 3240P_2 + 810P_1 
\\8 & -\z_{5}^2 - \z_{5}^3 & %2160P_3 + 3240P_2 + 810P_1 
\\\hline
9 & -1 & 2160P_3 + 3600P_2 + 1260P_1 \\
9 & -\z_{3} & %2160P_3 + 3600P_2 + 1260P_1 
\\
9 & 1 + \z_3 & %2160P_3 + 3600P_2 + 1260P_1 
\\\hline
9 & 1 & 2160P_3 + 2880P_2 + 540P_1 
\\
9 & \z_{3} & %2160P_3 + 2880P_2 + 540P_1 
\\
9 & -1 - \z_3 & %2160P_3 + 2880P_2 + 540P_1
\end{array}
\qquad
\begin{array}{c|c|c}
\deg(\chi) & \chi(c^{-1}) & f_{\chi}(x) \\\hline\hline
3 & \z_{15} + \z_{15}^4 & 2160P_3 + 4320P_2 + 2160P_1 
\\
3 & \z_{15}^{2} + \z_{15}^{8} & %2160P_3 + 4320P_2 + 2160P_1 
\\
3 & \z_{15}^7 + \z_{15}^{13} & %2160P_3 + 4320P_2 + 2160P_1 
\\
3 & \z_{15}^{11} + \z_{15}^{14} & %2160P_3 + 4320P_2 + 2160P_1 
\\
6 & \z_{3} & %2160P_3 + 4320P_2 + 2160P_1 
\\
6 & -1 - \z_3 & %2160P_3 + 4320P_2 + 2160P_1 
\\\hline
3 & -\z_{15} - \z_{15}^{4} & 2160P_3 + 2160P_2 
\\
3 & -\z_{15}^{2} - \z_{15}^{8} & %2160P_3 + 2160P_2 
\\
3 & -\z_{15}^{7} - \z_{15}^{13} & %2160P_3 + 2160P_2 
\\
3 & -\z_{15}^{11} - \z_{15}^{14} & %2160P_3 + 2160P_2 
\\
6 & -\z_{3} & %2160P_3 + 2160P_2 
\\
6 & 1 + \z_3 & %2160P_3 + 2160P_2 
\end{array}
$
\caption{$G_{27}$: $P_3 = \frac{(x - 1)(x - 19)(x - 25)}{6 \cdot 12 \cdot 30}$, $P_2 = \frac{(x - 1)(x - 15)}{6 \cdot 12}$, $P_1 = \frac{2(x - 1)}{9}$.}
\end{table}

\subsection{Higher ranks}
We give the relevant character polynomials.  Except in $G_{32}$, we do not reexpress them in a different basis.  

\begin{table}[H]
$
\begin{array}{c|c|c}
\deg(\chi) & \chi(c^{-1}) & f_{\chi}(x) \\\hline\hline
1 & 1 & x^4 + 24x^3 + 190x^2 + 552x + 385 \\
\hline
1 & 1 & x^4 - 24x^3 + 190x^2 - 552x + 385 \\
\hline
1 & 1 & x^4 - 26x^2 + 25 \\

1 & 1 & %x^4 - 26x^2 + 25 
\\\hline
2 & -1 & x^4 - 12x^3 + 34x^2 + 12x - 35 \\

2 & -1 & %x^4 - 12x^3 + 34x^2 + 12x - 35 
\\\hline
2 & -1 & x^4 + 12x^3 + 34x^2 - 12x - 35 \\

2 & -1 & %x^4 + 12x^3 + 34x^2 - 12x - 35 
\\\hline
4 & 1 & x^4 - 14x^2 + 13 \\
\hline
6 & -1 & x^4 - 18x^2 + 17 \\
\hline
6 & -1 & x^4 + 6x^2 - 7 \\
\hline
12 & 1 & x^4 - 6x^2 + 5
\end{array}
$
\caption{$G_{28} = F_4$.}
\end{table}

\begin{table}[H]
$
\begin{array}{c|c|c}
\deg(\chi) & \chi(c^{-1}) & f_{\chi}(x) \\\hline\hline
1 & 1 & x^4 + 40x^3 + 530x^2 + 2720x + 4389 \\
\hline
1 & 1 & x^4 - 40x^3 + 530x^2 - 2480x + 1989 \\
\hline
4 & -1 & x^4 + 20x^3 + 110x^2 + 100x - 231 \\

4 & i & %x^4 + 20x^3 + 110x^2 + 100x - 231 
\\
4 & -i & %x^4 + 20x^3 + 110x^2 + 100x - 231 
\\\hline
4 & -1 & x^4 - 20x^3 + 110x^2 - 100x + 9 \\

4 & i & %x^4 - 20x^3 + 110x^2 - 100x + 9 
\\
4 & -i & %x^4 - 20x^3 + 110x^2 - 100x + 9 
\\\hline
6 & -1 & x^4 + 10x^2 + 40x - 51 \\

6 & -1 & %x^4 + 10x^2 + 40x - 51 
\end{array}
\qquad
\begin{array}{c|c|c}
\deg(\chi) & \chi(c^{-1}) & f_{\chi}(x) \\\hline\hline
6 & 1 & x^4 - 30x^2 - 40x + 69 \\

6 & -1 & %x^4 - 30x^2 - 40x + 69 
\\
6 & -1 & %x^4 - 30x^2 - 40x + 69 
\\\hline
16 & i & x^4 - 10x^3 + 20x^2 + 10x - 21 
\\
16 & -i & %x^4 - 10x^3 + 20x^2 + 10x - 21 
\\\hline
16 & i & x^4 + 10x^3 + 20x^2 - 10x - 21 
\\
16 & -i & %x^4 + 10x^3 + 20x^2 - 10x - 21 
\\\hline
24 & 1 & x^4 - 10x^2 + 9 
\\
24 & i & %x^4 - 10x^2 + 9 
\\
24 & -i & %x^4 - 10x^2 + 9
\end{array}
$
\caption{$G_{29}$.}
\end{table}

\begin{table}[H]
$
\begin{array}{c|c|c}
\deg(\chi) & \chi(c^{-1}) & f_{\chi}(x) \\\hline\hline
1 & 1 & x^4 + 60x^3 + 1138x^2 + 7140x + 6061 
\\\hline
1 & 1 & x^4 - 60x^3 + 1138x^2 - 7140x + 6061 
\\\hline
4 & \z_5 + \z_5^4 & x^4 - 30x^3 + 208x^2 + 30x - 209 
\\
4 & \z_5^2 + \z_5^3 & %x^4 - 30x^3 + 208x^2 + 30x - 209 
\\\hline
4 & \z_5 + \z_5^4 & x^4 + 30x^3 + 208x^2 - 30x - 209 
\\
4 & \z_5^2 + \z_5^3 & %x^4 + 30x^3 + 208x^2 - 30x - 209 
\\\hline
6 & -\z_5 - \z_5^4 & x^4 - 102x^2 + 101 
\\
6 & -\z_5^2 - \z_5^3 & %x^4 - 102x^2 + 101 
\\\hline
16 & 1 & x^4 - 15x^3 + 43x^2 + 15x - 44 
\\
16 & -1 & %x^4 - 15x^3 + 43x^2 + 15x - 44 
\\\hline
16 & 1 & x^4 + 15x^3 + 43x^2 - 15x - 44 
\\
16 & -1 & %x^4 + 15x^3 + 43x^2 - 15x - 44 
\end{array}
\qquad
\begin{array}{c|c|c}
\deg(\chi) & \chi(c^{-1}) & f_{\chi}(x) \\\hline\hline
8 & -1 & x^4 + 28x^2 - 29 
\\\hline
10 & -1 & x^4 + 10x^2 - 11 
\\\hline
16 & \z_5 + \z_5^4 & x^4 - 32x^2 + 31 
\\
16 & \z_5^2 + \z_5^3 & %x^4 - 32x^2 + 31 
\\\hline
24 & 1 & x^4 - 12x^2 + 11 
\\
24 & 1 & %x^4 - 12x^2 + 11 
\\
24 & -\z_5 - \z_5^4 & %x^4 - 12x^2 + 11 
\\
24 & -\z_5^2 - \z_5^3 & %x^4 - 12x^2 + 11 
\\
48 & -1 & %x^4 - 12x^2 + 11 
\\\hline
30 & \z_5 + \z_5^4 & x^4 - 30x^2 + 29 
\\
30 & \z_5^2 + \z_5^3 & %x^4 - 30x^2 + 29 
\\\hline
40 & 1 & x^4 - 20x^2 + 19
\end{array}
$
\caption{$G_{30} = H_4$.}
\end{table}

\begin{table}[H]
\raisebox{12pt}{$
\begin{array}{c|c|c}
\deg(\chi) & \chi(c^{-1}) & f_{\chi}(x) \\\hline\hline
1 & 1 & 155520P_4 + 622080P_3 + \hfill \\
&& \hfill + 933120P_2 +  622080P_1 + 155520 
\\\hline
1 & \z_{3} & 155520P_4 
\\1 & -1 - \z_3 & % 155520P_4 
\\4 & \z_{3} & % 155520P_4 
\\4 & -1 - \z_3 & % 155520P_4 
\\6 & 1 & % 155520P_4 
\\\hline
4 & 1 & 155520P_4 + 155520P_3 
\\4 & 1 & % 155520P_4 + 155520P_3 
\\\hline
4 & \z_{3} & 155520P_4 + 466560P_3 + \hfill %466560P_2 + 155520P_1 
\\4 & -1 - \z_3 & % 155520P_4 + 466560P_3 
\hfill + 466560P_2 + 155520P_1 
\\\hline
6 & \z_{3} & 155520P_4 + 311040P_3 + 155520P_2 
\\6 & -1 - \z_3 & % 155520P_4 + 311040P_3 + 155520P_2 
\\\hline
24 & -1 & 155520P_4 + 311040P_3 + \hfill \\
&&\hfill + 194400P_2 + 38880P_1 
\end{array}
$}
\qquad
$
\begin{array}{c|c|c}
\deg(\chi) & \chi(c^{-1}) & f_{\chi}(x) \\\hline\hline
24 & -\z_{3} & 155520P_4 + 155520P_3 + 38880P_2 
\\24 & 1 + \z_3 & % 155520P_4 + 155520P_3 + 38880P_2 
\\\hline
36 & -1 & 155520P_4 + 155520P_3 + 25920P_2 
\\36 & -1 & % 155520P_4 + 155520P_3 + 25920P_2 
\\36 & -\z_{3} & % 155520P_4 + 155520P_3 + 25920P_2 
\\36 & 1 + \z_3 & % 155520P_4 + 155520P_3 + 25920P_2 
\\\hline
36 & -\z_{3} & 155520P_4 + 311040P_3 +  \hfill %181440P_2 + 25920P_1 
\\36 & 1 + \z_3 & % 155520P_4 + 311040P_3 
 \hfill + 181440P_2 + 25920P_1 
\\\hline
64 & 1 & 155520P_4 + 155520P_3 + 29160P_2 
\\64 & -1 & % 155520P_4 + 155520P_3 + 29160P_2 
\\\hline 
64 & \z_{3} & 155520P_4 + 233280P_3 + \hfill % 87480P_2 + 4860P_1 
\\64 & -\z_{3} & % 155520P_4 + 233280P_3 
 \hfill + 87480P_2 + 4860P_1 
\\64 & 1 + \z_3 & % 155520P_4 + 233280P_3 + 87480P_2 + 4860P_1 
\\64 & -1 - \z_3 & % 155520P_4 + 233280P_3 + 87480P_2 + 4860P_1 
\\\hline
81 & 1 & 155520P_4 + 207360P_3 +  \hfill %69120P_2 + 3840P_1 
\\81 & \z_{3} & % 155520P_4 + 207360P_3 
 \hfill + 69120P_2 + 3840P_1 
\\81 & -1 - \z_3 & % 155520P_4 + 207360P_3 + 69120P_2 + 3840P_1
\end{array}
$
\caption{$G_{32}$: $(e^*_i) = (1, 7, 13, 19)$, $(d_i) = (12, 18, 24, 30)$, $P_i = \prod_{j = 1}^i \frac{x - e^*_j}{d_j}$.}
\label{table:G_32}
\end{table}

\begin{table}[H]
$
\begin{array}{c|c|c}
\deg(\chi) & \chi(c^{-1}) & f_{\chi}(x) \\\hline\hline
1 & 1 & x^5 + 45x^4 + 750x^3 + \hfill \\
&& \hfill + 5750x^2 + 20049x + 25245 \\
\hline
1 & -1 & x^5 - 45 x^4 + 750 x^3 - \hfill \\
&& \hfill - 5590 x^2 + 17169 x - 12285 \\
\hline
5 & \z_3 & x^5 + 27x^4 + 246x^3 + \hfill % 818x^2 + 393x - 1485 
\\
5 & -1 - \z_{3} & %x^5 + 27x^4 + 246x^3 
\hfill + 818x^2 + 393x - 1485 
\\\hline
5 & -\z_3 &x^5 - 27 x^4 + 246 x^3 - \hfill % 802 x^2 + 393 x + 189 
\\
5 & 1 + \z_{3} & %x^5 - 27 x^4 + 246 x^3 
\hfill -  802 x^2 + 393 x + 189\\
\hline
10 & \z_{3} & x^5 + 9x^4 - 6x^3 - \hfill % 190x^2 - 219x + 405 
\\
10 & -1 - \z_3 & %x^5 + 9x^4 - 6x^3
 \hfill - 190x^2 - 219x + 405 
%\\\hline
\end{array}
$
\qquad
$
\begin{array}{c|c|c}
\deg(\chi) & \chi(c^{-1}) & f_{\chi}(x) \\\hline\hline
10 & -\z_{3} &  x^5 - 9x^4 - 6x^3 + 134x^2 + 69x - 189 \\
10 & 1 + \z_3 & %x^5 - 9x^4 - 6x^3 + 134x^2 + 69x - 189 
\\\hline
20 & -1 &  x^5 + 9x^4 + 30x^3 + 62x^2 + 33x - 135 \\
\hline
20 & 1 &  x^5 - 9x^4 + 30x^3 - 46x^2 - 111x + 135 \\
\hline
40 & -\z_3 & x^5 + 9x^4 + 12x^3 - 64x^2 - 93x + 135 \\
40 & 1 + \z_{3} & %x^5 + 9x^4 + 12x^3 - 64x^2 - 93x + 135 
\\\hline
40 & \z_3 & x^5 - 9x^4 + 12x^3 + 44x^2 - 21x - 27 \\
40 & -1 - \z_{3} & %x^5 - 9x^4 + 12x^3 + 44x^2 - 21x - 27 
\\\hline
64 & 1 & x^5 - 15x^3 - 10x^2 + 24x \\
64 & -1 & %x^5 - 15x^3 - 10x^2 + 24x
\end{array}
$
\caption{$G_{33}$.}
\end{table}

\begin{table}[H]
$
\begin{array}{c|c|c}
\deg(\chi) & \chi(c^{-1}) & f_{\chi}(x) \\\hline\hline
1 & 1 & x^6 + 126x^5 + 6195x^4 + 151060x^3 + \hfill \\
&& \hfill + 1904679x^2 + 11559534x + 25569445
\\\hline
1 & 1 & x^6 - 126x^5 + 6195x^4 - 148820x^3 + \hfill\\
&& \hfill + 1763559x^2 - 8703534x + 7082725 
\\\hline
6 & \z_3 & x^6 + 84x^5 + 2625x^4 + 37240x^3 + \hfill 
\\
6 & -1 - \z_3 & \hfill + 226779x^2 + 356916x - 623645 
\\\hline
6 & \z_3 & x^6 - 84x^5 + 2625x^4 - 36680x^3 + \hfill %208299x^2 - 254436x + 80275 
\\
6 & -1 - \z_3 & %x^6 - 84x^5 + 2625x^4 - 36680x^3 
\hfill + 208299x^2 - 254436x + 80275 
\\\hline
15 & \z_3 & x^6 - 42x^5 + 483x^4 + 196x^3 - \hfill %19929x^2 - 9114x + 28405 
\\
15 & -1 - \z_3 & %x^6 - 42x^5 + 483x^4 + 196x^3 
\hfill - 19929x^2 - 9114x + 28405 
\\\hline
15 & \z_3 & x^6 + 42x^5 + 483x^4 - 644x^3 - \hfill %33369x^2 - 88998x + 122485 
\\
15 & -1 - \z_3 & %x^6 + 42x^5 + 483x^4 - 644x^3 
\hfill - 33369x^2 - 88998x + 122485 
\\\hline
20 & 1 & x^6 - 231x^4 - 392x^3 + \hfill %12915x^2 + 38472x - 50765 
\\
20 & 1 & %x^6 - 231x^4 - 392x^3 
\hfill + 12915x^2 + 38472x - 50765 
\\\hline
90 & -1 & x^6 - 42x^5 + 651x^4 - 4508x^3 + \hfill \\
&& \hfill + 11319x^2 + 14406x - 21827 
\\\hline
90 & -1 & x^6 + 42x^5 + 651x^4 + 4732x^3 + \hfill \\
&& \hfill +  16023x^2 + 13146x - 34595 
\\\hline
120 & -1 & x^6 + 21x^4 + 112x^3 + \hfill %819x^2 + 5712x - 6665 
\\
120 & -1 & %x^6 + 21x^4 + 112x^3 
\hfill + 819x^2 + 5712x - 6665 
\\\hline
120 & -\z_3 & x^6 - 42x^5 + 609x^4 - 3332x^3 + \hfill %3507x^2 + 8526x - 9269 
\\
120 & 1 + \z_3 & %x^6 - 42x^5 + 609x^4 - 3332x^3 
\hfill + 3507x^2 + 8526x - 9269 
\\\hline
120 & -\z_3 & x^6 + 42x^5 + 609x^4 + 3388x^3 + \hfill %3675x^2 - 12390x + 4675 
\\
120 & 1 + \z_3 & %x^6 + 42x^5 + 609x^4 + 3388x^3 
\hfill + 3675x^2 - 12390x + 4675 
%\\\hline
\end{array}
$
\qquad
\raisebox{13pt}{$
\begin{array}{c|c|c}
\deg(\chi) & \chi(c^{-1}) & f_{\chi}(x) \\\hline\hline
384 & \z_3 & x^6 - 21x^5 + 105x^4 + 175x^3 - \hfill %1176x^2 - 924x + 1840 
\\
384 & -\z_3 & %x^6 - 21x^5 + 105x^4 + 175x^3 
\hfill - 1176x^2 - 924x + 1840 
\\
384 & 1 + \z_3 & %x^6 - 21x^5 + 105x^4 + 175x^3 - 1176x^2 - 924x + 1840 
\\
384 & -1 - \z_3 & %x^6 - 21x^5 + 105x^4 + 175x^3 - 1176x^2 - 924x + 1840 
\\\hline
384 & \z_3 & x^6 + 21x^5 + 105x^4 - 245x^3 - \hfill % 2226x^2 - 1176x + 3520 
\\
384 & -\z_3 & %x^6 + 21x^5 + 105x^4 - 245x^3 
\hfill - 2226x^2 - 1176x + 3520 
\\
384 & 1 + \z_3 & %x^6 + 21x^5 + 105x^4 - 245x^3 - 2226x^2 - 1176x + 3520 
\\
384 & -1 - \z_3 & %x^6 + 21x^5 + 105x^4 - 245x^3 - 2226x^2 - 1176x + 3520 
\\\hline
540 & -1 & x^6 - 63x^4 - 56x^3 + \hfill %819x^2 + 504x - 1205 
\\
540 & -1 & %x^6 - 63x^4 - 56x^3 
\hfill + 819x^2 + 504x - 1205 
\\
540 & -\z_3 & %x^6 - 63x^4 - 56x^3 + 819x^2 + 504x - 1205 
\\
540 & 1 + \z_3 & %x^6 - 63x^4 - 56x^3 + 819x^2 + 504x - 1205 
\\\hline
720 & -\z_3 & x^6 - 21x^4 + 28x^3 - \hfill %189x^2 - 924x + 1105 
\\
720 & 1 + \z_3 & %x^6 - 21x^4 + 28x^3 
\hfill - 189x^2 - 924x + 1105 
\\\hline
729 & 1 & x^6 + 14x^5 + 35x^4 - 140x^3 - \hfill % 441x^2 + 126x + 405 
\\
729 & \z_3 & %x^6 + 14x^5 + 35x^4 - 140x^3 
\hfill - 441x^2 + 126x + 405 
\\
729 & -1 - \z_3 & %x^6 + 14x^5 + 35x^4 - 140x^3 - 441x^2 + 126x + 405 
\\\hline
729 & 1 & x^6 - 14x^5 + 35x^4 + 140x^3 - \hfill %441x^2 - 126x + 405 
\\
729 & \z_3 & %x^6 - 14x^5 + 35x^4 + 140x^3 
\hfill - 441x^2 - 126x + 405 
\\
729 & -1 - \z_3 & %x^6 - 14x^5 + 35x^4 + 140x^3 - 441x^2 - 126x + 405 
\\\hline
1280 & 1 & x^6 - 42x^4 - 14x^3 + \hfill % 441x^2 + 294x - 680 
\\
1280 & -1 & %x^6 - 42x^4 - 14x^3 
\hfill + 441x^2 + 294x - 680
\end{array}
$}
\caption{$G_{34}$.}
\end{table}

\begin{table}[H]
$
\begin{array}{c|c|c}
\deg(\chi) & \chi(c^{-1}) & f_{\chi}(x) \\\hline\hline
1 & 1 & x^6 + 36x^5 + 510x^4 + 3600x^3 + 13089x^2 + 22284x + 12320 \\\hline
1 & 1 & x^6 - 36x^5 + 510x^4 - 3600x^3 + 13089x^2 - 22284x + 12320 \\\hline
6 & -1 & x^6 - 24x^5 + 210x^4 - 780x^3 + 909x^2 + 804x - 1120 \\\hline
6 & -1 & x^6 + 24x^5 + 210x^4 + 780x^3 + 909x^2 - 804x - 1120 \\\hline
10 & -1 & x^6 + 6x^4 + 57x^2 - 64 \\\hline
15 & -1 & x^6 - 12x^5 + 54x^4 - 120x^3 + 105x^2 + 132x - 160 \\\hline
15 & -1 & x^6 + 12x^5 + 54x^4 + 120x^3 + 105x^2 - 132x - 160 \\\hline
20 & 1 & x^6 - 30x^4 + 237x^2 - 208 \\\hline
30 & 1 & x^6 - 12x^5 + 42x^4 - 12x^3 - 123x^2 + 24x + 80 \\\hline
30 & 1 & x^6 + 12x^5 + 42x^4 + 12x^3 - 123x^2 - 24x + 80 \\\hline
60 & 1 & x^6 - 6x^4 - 27x^2 + 32 \\\hline
90 & -1 & x^6 - 18x^4 + 81x^2 - 64
\end{array}
$
\caption{$G_{35} = E_6$.}
\end{table}

\begin{table}[H]
$
\begin{array}{c|c|c}
\deg(\chi) & \chi(c^{-1}) & f_{\chi}(x) \\\hline\hline
1 & 1 & x^7 + 63x^6 + 1617x^5 + 21735x^4 + 162939x^3 + 663957x^2 + 1286963x + 765765 \\\hline
1 & -1 & x^7 - 63 x^6 + 1617 x^5 - 21735 x^4 + 162939 x^3 - 663957 x^2 + 1286963 x - 765765 \\\hline
7 & 1 & x^7 - 45 x^6 + 789 x^5 - 6705 x^4 + 27219 x^3 - 38295 x^2 - 28009 x + 45045 
\\\hline
7 & -1 & x^7 + 45x^6 + 789x^5 + 6705x^4 + 27219x^3 + 38295x^2 - 28009x - 45045 \\\hline
35 & 1 & x^7 - 27x^6 + 285x^5 - 1503x^4 + 4035x^3 - 3825x^2 - 4321x + 5355 \\\hline
35 & -1 & x^7 + 27x^6 + 285x^5 + 1503x^4 + 4035x^3 + 3825x^2 - 4321x - 5355 \\\hline
35 & -1 & x^7 - 9x^6 - 39x^5 + 423x^4 + 363x^3 - 5139x^2 - 325x + 4725 \\\hline
35 & 1 & x^7 + 9x^6 - 39x^5 - 423x^4 + 363x^3 + 5139x^2 - 325x - 4725 \\\hline
56 & -1 & x^7 - 27x^6 + 267x^5 - 1125x^4 + 1479x^3 + 1467x^2 - 1747x - 315 \\\hline
56 & 1 & x^7 + 27x^6 + 267x^5 + 1125x^4 + 1479x^3 - 1467x^2 - 1747x + 315 \\\hline
70 & 1 & x^7 - 9x^6 + 33x^5 - 81x^4 + 219x^3 - 675x^2 - 253x + 765 \\\hline
70 & -1 & x^7 + 9x^6 + 33x^5 + 81x^4 + 219x^3 + 675x^2 - 253x - 765 \\\hline
280 & 1 & x^7 + 9x^6 + 15x^5 - 45x^4 - 177x^3 - 369x^2 + 161x + 405 \\\hline
280 & -1 & x^7 - 9x^6 + 15x^5 + 45x^4 - 177x^3 + 369x^2 + 161x - 405 \\\hline
280 & 1 & x^7 - 9x^6 - 3x^5 + 171x^4 - 141x^3 - 747x^2 + 143x + 585 \\\hline
280 & -1 & x^7 + 9x^6 - 3x^5 - 171x^4 - 141x^3 + 747x^2 + 143x - 585 \\
\hline
512 & 1 & x^7 - 21x^5 + 84x^3 - 64x \\
512 & -1 & %x^7 - 21x^5 + 84x^3 - 64x
\end{array}
$
\caption{$G_{36} = E_7$.}
\end{table}

\begin{table}[H]
$
\begin{array}{c|c|c}
\deg(\chi) & \chi(c^{-1}) & f_{\chi}(x) \\\hline\hline
1 & 1 & x^8 + 120x^7 + 6020x^6 + 163800x^5 + 2616558x^4 + 24693480x^3 + 130085780x^2 + 323507400x + 215656441
\\\hline
 1 & 1 & x^8 - 120x^7 + 6020x^6 - 163800x^5 + 2616558x^4 - 24693480x^3 + 130085780x^2 - 323507400x + 215656441
\\\hline
 8 & -1 & x^8 + 90x^7 + 3290x^6 + 62370x^5 + 644028x^4 + 3400110x^3 + 6789110x^2 - 3462570x - 7436429
\\\hline
 8 & -1 & x^8 - 90x^7 + 3290x^6 - 62370x^5 + 644028x^4 - 3400110x^3 + 6789110x^2 + 3462570x - 7436429
\\\hline
 56 & 1 & x^8 + 30x^7 + 170x^6 - 2250x^5 - 20532x^4 + 38970x^3 + 489830x^2 - 36750x - 469469
\\\hline
 56 & 1 & x^8 - 30x^7 + 170x^6 + 2250x^5 - 20532x^4 - 38970x^3 + 489830x^2 + 36750x - 469469
\\\hline
 70 & -1 & x^8 - 220x^6 + 15630x^4 - 362380x^2 + 346969
\\\hline
 84 & -1 & x^8 - 60x^7 + 1460x^6 - 18540x^5 + 130398x^4 - 481140x^3 + 667940x^2 + 499740x - 799799
\\\hline
 84 & -1 & x^8 + 60x^7 + 1460x^6 + 18540x^5 + 130398x^4 + 481140x^3 + 667940x^2 - 499740x - 799799
\\\hline
 112 & 1 & x^8 + 60x^7 + 1430x^6 + 17100x^5 + 105288x^4 + 293940x^3 + 182570x^2 - 311100x - 289289
\\\hline
 112 & 1 & x^8 - 60x^7 + 1430x^6 - 17100x^5 + 105288x^4 - 293940x^3 + 182570x^2 + 311100x - 289289
\\\hline
 420 & -1 & x^8 + 20x^6 + 510x^4 + 12740x^2 - 13271
\\\hline
 448 & -1 & x^8 - 130x^6 + 5208x^4 - 61870x^2 + 56791
\\\hline
 448 & -1 & x^8 + 30x^7 + 350x^6 + 2070x^5 + 7008x^4 + 17370x^3 + 32450x^2 - 19470x - 39809
\\\hline
 448 & -1 & x^8 - 30x^7 + 350x^6 - 2070x^5 + 7008x^4 - 17370x^3 + 32450x^2 + 19470x - 39809
\\\hline
 1008 & -1 & x^8 + 30x^7 + 290x^6 + 630x^5 - 4572x^4 - 18630x^3 + 2510x^2 + 17970x + 1771
\\\hline
 1008 & -1 & x^8 - 30x^7 + 290x^6 - 630x^5 - 4572x^4 + 18630x^3 + 2510x^2 - 17970x + 1771
\\\hline
 1134 & 1 & x^8 - 60x^6 + 878x^4 - 4140x^2 + 3321
\\\hline
 1344 & 1 & x^8 + 30x^7 + 320x^6 + 1350x^5 + 618x^4 - 11430x^3 - 28120x^2 + 10050x + 27181
\\\hline
 1344 & 1 & x^8 - 30x^7 + 320x^6 - 1350x^5 + 618x^4 + 11430x^3 - 28120x^2 - 10050x + 27181
\\\hline
 1680 & 1 & x^8 - 100x^6 + 3030x^4 - 26500x^2 + 23569
\\\hline
 2016 & 1 & x^8 - 10x^6 - 72x^4 - 5590x^2 + 5671
\\\hline
 4096 & 1 & x^8 + 15x^7 + 35x^6 - 315x^5 - 1092x^4 + 1260x^3 + 4640x^2 - 960x - 3584
\\ 4096 & -1 & %x^8 + 15x^7 + 35x^6 - 315x^5 - 1092x^4 + 1260x^3 + 4640x^2 - 960x - 3584
\\\hline 
4096 & 1 & x^8 - 15x^7 + 35x^6 + 315x^5 - 1092x^4 - 1260x^3 + 4640x^2 + 960x - 3584
\\ 4096 & -1 & %x^8 - 15x^7 + 35x^6 + 315x^5 - 1092x^4 - 1260x^3 + 4640x^2 + 960x - 3584
\\\hline
 4480 & 1 & x^8 - 40x^6 + 510x^4 - 2200x^2 + 1729
\\\hline
 4536 & 1 & x^8 - 60x^6 + 1118x^4 - 6540x^2 + 5481
\\\hline
 5670 & -1 & x^8 - 60x^6 + 1070x^4 - 6060x^2 + 5049
\\\hline
 7168 & -1 & x^8 - 40x^6 + 348x^4 + 1040x^2 - 1349
\end{array}
$
\caption{$G_{37} = E_8$.}
\end{table}

\end{document}